\documentclass{amsart}
\usepackage[utf8]{inputenc}
\usepackage{amsfonts}
\usepackage{amsmath}
\numberwithin{equation}{section}
\usepackage{amssymb}
\usepackage{amsthm}
\usepackage{mathrsfs}
\usepackage{mathtools}
\usepackage[hidelinks]{hyperref}
\usepackage[capitalise]{cleveref}
\usepackage{tikz}
\usetikzlibrary{arrows.meta,positioning,calc,fit,backgrounds}
\usepackage{pgfplots}
\pgfplotsset{compat=1.16}
\usepgfplotslibrary{fillbetween}
\usepackage{graphicx}
\usepackage[top=3cm,bottom=3cm,left=2.5cm,right=2.5cm,headsep=10pt,letterpaper]{geometry} % Page margins

\usepackage[
    backend=biber,
    style=alphabetic,
    sorting=nyt
]{biblatex}
\theoremstyle{plain}
\newtheorem{theorem}{Theorem}[section]
\newtheorem{proposition}[theorem]{Proposition}
\newtheorem{lemma}[theorem]{Lemma}
\newtheorem{corollary}[theorem]{Corollary}

\newtheorem*{thmA}{Theorem A}
\newtheorem*{thmB}{Theorem B}
\newtheorem*{thmC}{Theorem C}

\theoremstyle{definition}

\theoremstyle{remark}
\newtheorem{remark}[theorem]{Remark}

\allowdisplaybreaks[2]
\newcommand{\R}{\mathbb R}
\newcommand{\E}{\mathbb E}
\newcommand{\eps}{\varepsilon}
\newcommand{\norm}[1]{\left\|#1\right\|}
\newcommand{\Acal}{\mathfrak A}

\newcommand{\sA}{\mathsf A}
\newcommand{\sH}{\mathsf H}
\newcommand{\sG}{\mathsf G}
\newcommand{\sL}{\mathsf L}

\title[Statistical uniformization by heat-kernel Fisher information]{Statistical Uniformization of Surfaces by Heat-Kernel Fisher Information}

\author[G. Cho]{Gunhee Cho}
\address{Department of Mathematics \\ Texas State University \\ San Marcos, TX 78666, USA}
\email{wvx17@txstate.edu}
\thanks{Corresponding author: Gunhee Cho (\texttt{wvx17@txstate.edu}).}

\author[H. C. Jang]{Hyun Chul Jang}
\address{Department of Mathematics \\ Texas State University \\ San Marcos, TX 78666, USA}
\email{hcjang@txstate.edu}



\hypersetup{pdftitle={Statistical Uniformization of Surfaces by Heat-Kernel Fisher Information},pdfauthor={Gunhee Cho, Hyun Chul Jang}}

\begin{document}

\begin{abstract}
A Brownian particle observed at time $t$ carries information about where it
started. On a closed surface this makes the surface itself a statistical
model, whose Fisher information pulls back to a heat-smoothed metric
$h_t=2tJ_t$; its volume $V_t$ is a normalized Jeffreys volume, and
$\log(V_t/2t)$ is the model-dependent constant in Rissanen's
minimum-description-length complexity.
We ask whether maximizing this statistical size among metrics of fixed area
in a conformal class recovers the uniformizing metric. First,
$9(A-2\pi\chi t-V_t)/t^2$ converges to the curvature energy $\int K^2\,dA$
(the Calabi energy up to a topological constant), uniformly on classes of
bounded geometry; hence, within a fixed-area conformal class of bounded
geometry containing the uniformizer, approximate maximizers of $V_t$
converge to the constant-curvature metric (on the sphere, to the round
metrics). Second, near every flat torus we prove a sharp finite-scale
version: for small perturbations whose oscillations are resolved by the
diffusion, $t|k|^2\le\theta<\tau_*$, the flat-subtracted defect is
comparable to $\|\Delta f\|_2^2$, uniformly in the frequency cutoff and over
noncollapsing lattices. The explicit threshold $\tau_*=4.408\ldots$ is
optimal, and it is exactly the observation time at which the information
about one curvature mode is maximal. Third, the flat structure is seen only
at an exponential scale: the volume loss of a flat torus measures the
ambiguity of winding and is comparable to
$t^{-1/2}e^{-\operatorname{sys}^2/(16t)}$, which selects the hexagonal lattice as $t\to0$ but not at all times.
\end{abstract}

\maketitle
\enlargethispage{3pt}

\section{Introduction}

Every closed surface carries, in each conformal class, a metric of
constant curvature. We ask whether this metric can be singled out by a
purely statistical criterion. Release a diffusion with generator
$\Delta_g$ from an unknown source point $x$ and record its position at
time $t$. The Fisher information $J_t(x)$ of this experiment measures how
sharply the endpoint locates the source. With the natural normalization
$h_t=2tJ_t$, its total volume $V_t=\operatorname{Vol}(\mathcal M,h_t)$ is the
normalized Jeffreys volume of the source model: a classical measure of
how many statistically distinguishable source positions the surface
contains \cite{Jeffreys1946,Rissanen1996}. On the Euclidean plane $h_t$ is
exactly the metric. On a surface, curvature changes the Fisher volume
element, and on a torus the possibility that the diffusion winds around a
handle blurs the source and lowers it (Figure~\ref{fig:statmodel}).

We call \emph{statistical uniformization} the problem of recovering the
uniformizing metric by maximizing $V_t$ among metrics of fixed area in a
conformal class, or equivalently by minimizing the normalized deficit
\[
 E_t(g)=\frac9{t^2}\bigl(A-V_t(g)\bigr).
\]
Three questions arise, and their answers live at three different scales
(Figure~\ref{fig:scales}). Does the Fisher volume select constant
curvature as $t\to0$? How fine a curvature oscillation does a diffusion
observed at time $t$ resolve stably, and where does this fail? On a torus,
where every conformal class is already flat, does the Fisher volume also
prefer a particular flat structure? The heat-kernel Fisher construction
itself is classical \cite{ISS,AmariNagaoka,AJLS}; our subject is its
finite-scale variational geometry.

\begin{figure}[t]
\centering
\begin{tikzpicture}[scale=0.92,
  lbl/.style={font=\footnotesize,align=center},
  path1/.style={thin,blue!55!black},
  path2/.style={thin,red!60!black}]
% panel (a): plane
\begin{scope}
 \fill[gray!6] (-1.9,-1.6) rectangle (1.9,1.6);
 \foreach \r/\o in {1.2/12,0.85/22,0.5/35} {\fill[blue!\o] (0,0) circle (\r);}
 \draw[path1] (0,0) -- (0.25,0.18) -- (0.18,0.45) -- (0.5,0.62) -- (0.72,0.5) -- (0.95,0.8);
 \draw[path1] (0,0) -- (-0.2,-0.15) -- (-0.45,-0.05) -- (-0.6,-0.38) -- (-0.92,-0.45);
 \fill (0.95,0.8) circle (1.1pt) node[right,font=\scriptsize]{$y$};
 \fill (-0.92,-0.45) circle (1.1pt) node[left,font=\scriptsize]{$y'$};
 \draw[thick,black] (0,0) circle (0.32);
 \fill (0,0) circle (1.4pt) node[below right=-1pt,font=\scriptsize]{$x$};
 \node[lbl] at (0,-2.05) {(a) plane: $h_t=g$,\\ $V_t=$ area};
\end{scope}
% panel (b): curved surface
\begin{scope}[xshift=5cm]
 \fill[gray!6] (-1.9,-1.6) rectangle (1.9,1.6);
 \foreach \y in {-1.2,-0.6,0,0.6,1.2}
   \draw[gray!45] plot[domain=-1.9:1.9,samples=40] (\x,{\y+0.35*exp(-(\x*\x+\y*\y)/0.6)*(1-\y)});
 \foreach \xx in {-1.2,-0.6,0,0.6,1.2}
   \draw[gray!45] plot[domain=-1.6:1.6,samples=40] ({\xx+0.35*exp(-(\xx*\xx+\x*\x)/0.6)*\xx},\x);
 \foreach \r/\o in {1.05/12,0.75/22,0.45/35} {\fill[blue!\o,rotate=20] (0,0) ellipse ({\r*1.12} and {\r*0.9});}
 \draw[thick,black,rotate=20] (0,0) ellipse (0.37 and 0.29);
 \fill (0,0) circle (1.4pt) node[below right=-1pt,font=\scriptsize]{$x$};
 \node[lbl] at (0,-2.05) {(b) curved: $h_t=(1-tK)g+O(t^2)$,\\ $A-V_t=2\pi\chi t+\frac{t^2}9\!\int\! K^2+\dots$};
\end{scope}
% panel (c): torus with winding
\begin{scope}[xshift=10cm]
 \draw[thick] (-1.6,-1.4) rectangle (1.6,1.4);
 \draw[->,thick] (-0.15,1.4) -- (0.15,1.4); \draw[->,thick] (-0.15,-1.4) -- (0.15,-1.4);
 \draw[->>,thick] (-1.6,-0.15) -- (-1.6,0.15); \draw[->>,thick] (1.6,-0.15) -- (1.6,0.15);
 \begin{scope}
  \clip (-1.6,-1.4) rectangle (1.6,1.4);
  \foreach \r/\o in {0.9/14,0.6/24} {\fill[blue!\o] (1.1,0) circle (\r);}
  \foreach \r/\o in {0.9/14,0.6/24} {\fill[blue!\o] (-2.1,0) circle (\r);}
 \end{scope}
 \fill (1.1,0) circle (1.4pt) node[above=1pt,font=\scriptsize]{$x$};
 \draw[path2] (1.1,0) -- (1.3,0.12) -- (1.6,0.05);
 \draw[path2] (-1.6,0.05) -- (-1.35,-0.1) -- (-1.15,0.08);
 \fill (-1.15,0.08) circle (1.1pt) node[below,font=\scriptsize]{$y$};
 \node[lbl] at (0,-2.05) {(c) torus: winding blurs $x$,\\ $1-V_t/A\asymp t^{-1/2}e^{-\operatorname{sys}^2/(16t)}$};
\end{scope}
\end{tikzpicture}
\caption{The statistical model behind $V_t$. A diffusion started at an
unknown source $x$ is observed at time $t$; the shaded blob is the law of
its endpoint, and the black ellipse is the unit ball of the normalized
Fisher information $h_t=2tJ_t$. (a) On the plane $h_t$ is exactly the
metric. (b) Curvature changes the Fisher volume element at order $t$;
after Gauss--Bonnet the leading curvature penalty is $t^2\int K^2/9$
(Theorem~\ref{thm:uniformCalabi}). (c) On a flat torus a path may leave
through one side and return through the opposite side, so an endpoint is
consistent with several lifts of the source; this winding ambiguity
lowers the Fisher volume by an exponentially small amount
(Theorem~\ref{thm:systolicloss}). The pictures are schematic.}
\label{fig:statmodel}
\end{figure}
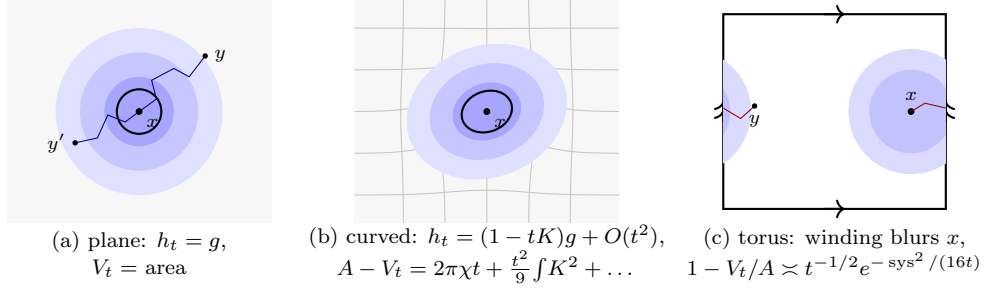

The first result answers the selection question in every genus.

\begin{thmA}[Recovery of uniformization]
Let $\mathcal M$ be a closed, connected, oriented surface and fix a class of metrics of area $A$
with bounded geometry \textup{(}uniform ellipticity, finitely many bounded
derivatives and an injectivity-radius bound\textup{)}. Uniformly on this
class,
\[
 V_t(g)=A-2\pi\chi(\mathcal M)\,t-\frac{t^2}{9}\int_{\mathcal M}K_g^2\,dA_g+O(t^3)
 \qquad(t\to0).
\]
Consequently, approximate maximizers of $V_t$ \textup{(}with error $o(t^2)$\textup{)} in a precompact class of
this kind inside one conformal class, containing its constant-curvature
metric, converge to that metric in $C^2$ as $t\to0$ \textup{(}on the sphere, to the
set of round metrics\textup{)}.
\end{thmA}

\noindent This is Theorem~\ref{thm:uniformCalabi} and
Corollary~\ref{cor:statisticaluniformization}. Equivalently, after removing
the topological term $18\pi\chi/t$, the deficit $E_t$ converges to the
curvature energy $\int K^2\,dA$, which differs from the Calabi energy
$\operatorname{Ca}(g)=\int(K-\bar K)^2\,dA$ by the constant $(2\pi\chi)^2/A$.
The coefficient comes from the heat-entropy asymptotics of Menkovski,
Portegies and Ravelonanosy~\cite{MPR} and an integrated de Bruijn identity;
it agrees with an independent flat computation.

Theorem~A is a statement at a fixed metric. It does not see a curvature
oscillation whose wavelength shrinks with $t$, and this is the regime where
the statistics are most interesting. Near a flat torus
$T_{\mathcal L}=\R^2/\mathcal L$ of area $A$ we write $g_f=(1+f)g_0$ with
$\int f=0$, expand $f$ in the dual frequencies $k\in\mathcal L^\vee$, and
measure its size by the Wiener norm $\sA(f)=\sum_k|\widehat f_k|$ and the
curvature energy $N(f)=\|\Delta f\|_2^2\approx4\int K^2\,dA$.

\begin{thmB}[Sharp finite-scale stability]
Fix $A,\ell_0>0$ and $0<\theta<\tau_*$. There are
$\varepsilon_*,t_*,c_*,C_*>0$, depending only on $A,\ell_0,\theta$, such
that on every flat torus of area $A$ and systole at least $\ell_0$,
\[
 c_*N(f)\le E_t(g_f)-E_t(g_0)\le C_*N(f)
 \quad(0<t<t_*)
\]
whenever
\[
 \int f=0,\quad \sA(f)\le\varepsilon_*,\quad
 \operatorname{supp}\widehat f\subset\{|k|\le\Lambda\},\quad
 t\Lambda^2\le\theta.
\]
The constants are independent of the cutoff $\Lambda$ and of the number of
modes. Here $\tau_*\approx4.408$ \textup{(}we prove $4<\tau_*<5$\textup{)} is
the unique positive zero of an explicit response function $q$, and no
positive constant uniform over full bands extends to $\theta\ge\tau_*$.
\end{thmB}

\noindent In statistical terms, the flat metric is a strict, nonlinearly
stable local maximizer of the Jeffreys volume among all resolved conformal
perturbations, with constants that do not deteriorate as the resolution
grows (Theorem~\ref{thm:flatlattices}). The resolution condition
$t|k|^2<\tau_*$ says that the wavelength of every oscillation exceeds
$2\pi\sqrt{t/\tau_*}\approx3\sqrt t$, about twice the per-coordinate
diffusion length $\sqrt{2t}$.

The third result concerns the flat structure itself.

\begin{thmC}[The flat structure appears at an exponential scale]
For every flat torus $0<V_t<A$, and the defect $d_t=1-V_t/A$ is the
information hidden by the winding of the diffusion. As $t\to0$,
\[
 d_t(\mathcal L)=\frac{\sqrt\pi\,m\,s}{8\sqrt t}\,
 e^{-s^2/(16t)}\bigl(1+O(t)\bigr),
\]
where $s$ is the systole and $m$ the number of shortest lattice vectors,
counting $\pm\ell$ \textup{(}$m=6$ for the hexagonal torus\textup{)}; the
bounds are uniform over lattices with systole bounded below.
Hence, among flat tori of fixed area and systole bounded below, maximizers
of $V_t$ converge to the hexagonal torus as $t\to0$, and the hexagonal torus beats every fixed
competitor at small times, although not at all times.
\end{thmC}

\noindent This is Proposition~\ref{prop:wrappingloss},
Theorem~\ref{thm:systolicloss}, Corollary~\ref{cor:hexselection} and
Remark~\ref{rem:hexnotalltimes}. The exponent $s^2/16t$, rather than
$s^2/4t$, reflects the overlap of two images at the midpoint of a shortest
vector (Figure~\ref{fig:winding}). Theorems~A--C separate three scales:
the conformal factor is seen at order $t^2$, curvature oscillations are
resolved up to $t|k|^2=\tau_*$, and the flat structure only at the
exponentially small scale $e^{-\operatorname{sys}^2/16t}$.

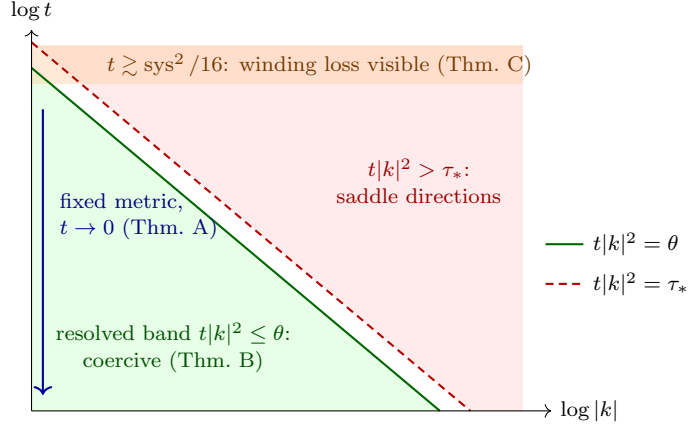
\begin{figure}[t]
\centering
\begin{tikzpicture}[x=2.5cm,y=1.05cm,font=\footnotesize]
% x = log|k| in [0,2.6], y = log t in [-4,0.6]
\fill[green!10] (0,-4) -- (0,0.322) -- (2.161,-4) -- cycle;
\fill[red!8] (0,0.644) -- (0,0.6) -- (2.6,0.6) -- (2.6,-4) -- (2.322,-4) -- cycle;
\fill[orange!35,opacity=0.45] (0,0.12) rectangle (2.6,0.6);
\draw[thick,green!40!black] (0,0.322) -- (2.161,-4);
\draw[thick,densely dashed,red!70!black] (0,0.644) -- (2.322,-4);
\draw[->] (0,-4) -- (2.75,-4) node[right]{$\log|k|$};
\draw[->] (0,-4) -- (0,0.8) node[above]{$\log t$};
\node[align=center,green!30!black] at (0.75,-3.2) {resolved band $t|k|^2\le\theta$:\\ coercive (Thm.~B)};
\node[align=center,red!60!black] at (2.05,-1.1) {$t|k|^2>\tau_*$:\\ saddle directions};
\node[orange!45!black,anchor=west] at (0.35,0.36) {$t\gtrsim\operatorname{sys}^2/16$: winding loss visible (Thm.~C)};
\draw[->,thick,blue!60!black] (0.06,-0.2) -- (0.06,-3.8);
\node[align=left,blue!50!black,anchor=west] at (0.1,-1.55) {fixed metric,\\ $t\to0$ (Thm.~A)};
\draw[thick,green!40!black] (2.72,-1.9) -- ++(0.2,0) node[right,black]{$t|k|^2=\theta$};
\draw[thick,densely dashed,red!70!black] (2.72,-2.4) -- ++(0.2,0) node[right,black]{$t|k|^2=\tau_*$};
\end{tikzpicture}
\caption{Three scales of statistical uniformization on a torus, in the
plane of frequency $|k|$ and observation time $t$ (logarithmic axes). A
fixed metric is reached along vertical lines $t\to0$ (Theorem~A). An
oscillation of frequency $k$ is resolved stably below the line
$t|k|^2=\theta$ (Theorem~B); above the dashed line $t|k|^2=\tau_*$ it
creates saddle directions (as $t\to0$, Corollary~\ref{cor:responseobstruction}).
Theorem~B also assumes small amplitude. The flat structure contributes at
every $t$, but becomes visible only once $t$ is comparable to
$\operatorname{sys}^2/16$ (Theorem~C). Schematic.}
\label{fig:scales}
\end{figure}

The threshold $\tau_*$ comes from an exact computation. On the plane lift the second variation of $E_t$ at the flat metric is
diagonal in Fourier space,
\[
 \tfrac12D^2E_t^\infty(0)[f,f]=A\sum_{k\ne0}|k|^4q(t|k|^2)|\widehat f_k|^2,
 \qquad q(0)=\tfrac14,
\]
and $q$ has a closed form through the hyperbolic sine integral. We prove by
hand that $q$ has a unique positive zero $\tau_*\in(4,5)$, that
$q\ge3/1000$ on $[0,4]$, and that $q<0$ beyond $\tau_*$. The zero has a
statistical meaning. Perturb the metric by one mode,
$(1+\eta\cos(k\cdot x))g_0$, and estimate the amplitude $\eta$, at $\eta=0$, from a
uniformly drawn source and the unwrapped (plane-lift) endpoint of the
diffusion at time $t$. The Fisher
information of this experiment is
\[
 I(\tau)=\tfrac12\mathscr B(\tau),\qquad
 \mathscr B(\tau)=\int_0^1\!\!\int_0^1e^{-\tau(u+v-2uv)}(1+\tau uv)^2\,du\,dv,
 \qquad q(\tau)=\frac9\tau\,I'(\tau),
\]
with $\tau=t|k|^2$ (Proposition~\ref{prop:amplitudeinformation}). Short
observations reveal the mode through the local speed of the diffusion,
long ones through its equilibrium density, and both give information
$\frac12$; in between the information is larger, and it is maximal at
exactly $\tau_*$.
Thus the Fisher volume detects a curvature mode stably precisely while
observing longer still increases the information about that mode
(Figure~\ref{fig:threeresponses}). By contrast, the second variation of the
fixed-time heat trace is positive at every scale
(Proposition~\ref{prop:heattrace}), so the threshold is a property of the
source Fisher information, not of spectral data.

The main contributions of the paper are the following.
\begin{enumerate}
\item \emph{Statistical recovery of uniformization in every genus}: a
uniform small-time expansion of the Fisher volume on bounded-geometry
classes, and convergence of approximate maximizers to the uniformizing
metric (Theorem~\ref{thm:uniformCalabi},
Corollary~\ref{cor:statisticaluniformization}).
\item \emph{An exact response with a sharp threshold}: the multiplier
$|k|^4q(t|k|^2)$, its closed form, hand-proved bounds, and saddle
directions beyond $\tau_*$ (Sections~\ref{sec:response}
and~\ref{sec:compact}).
\item \emph{Nonlinear stability uniform in the resolution}: Theorem~B, with
a spectral-tail variant (Theorem~\ref{thm:softstability}), the same result
for entropy production (Theorem~\ref{thm:entropy}), radial convexity
(Theorem~\ref{thm:radialselection}) and a tame Hessian estimate with a
contracting finite-dimensional gradient flow
(Theorem~\ref{thm:mixedhessian}, Corollary~\ref{cor:galerkinflow}).
\item \emph{A statistical meaning of $\tau_*$}: it is the unique optimal
observation time for one curvature mode
(Proposition~\ref{prop:amplitudeinformation}).
\item \emph{The flat structure}: the exact winding loss, its systolic
exponential scale and asymptotic selection of the hexagonal torus
(Section~\ref{sec:latticestatistics}).
\item \emph{The curved limit of the method}: on the round sphere conformal
automorphisms give exact zero modes, so the flat response formula cannot
hold uniformly on the full band at the round sphere
(Proposition~\ref{prop:sphericalgauge}).
\end{enumerate}

The main difficulty in Theorem~B is to control third and fourth variations of $E_t$ without losing
powers of the frequency. On the plane the heat kernel of a flat metric is
an exact Gaussian, and a perturbed kernel is close to Gaussian uniformly in
the resolution (Section~\ref{sec:kernel}). Regressing the displacement on
the score splits the deficit into a clock term, a compensated trace and an
anisotropy term (Figure~\ref{fig:closure}). Two exact cancellations,
of the first-gradient part of the score and of the third and fourth
displacement moments, leave only terms with two Hessian factors in $L^2$, plus an anisotropy
quartic in $\nabla f$,
which is what makes the estimates independent of the cutoff.

Our comparison points are heat-kernel Fisher pullbacks~\cite{ISS},
heat-kernel and eigenfunction embeddings~\cite{BBG,JMS,Portegies},
heat-entropy asymptotics~\cite{MPR}, conformal heat perturbations
\cite{Morpurgo,Zhao}, and conformal spectral extremal problems
\cite{OPS,LaugesenMorpurgo}; Theorem~B is a finite-scale, local analogue of
the extremal role of uniform metrics for the determinant of the
Laplacian~\cite{OPS}. The Calabi energy and its surface flows provide the
classical variational reference~\cite{Chrusciel,ChenCalabi,StruweCalabi}.
Other heat-Fisher and entropy constructions treat large-scale volume
growth~\cite{Ge}, evolving Ricci-flow kernels~\cite{CK}, or pointwise
curvature recovery~\cite{Sangha}. Montgomery's theorem on minimal theta
functions~\cite{Montgomery1988} is the spectral precedent for hexagonal
selection among flat tori.

All stability results are perturbative. Theorem~A assumes the
compactness of the competing class; it neither proves an unrestricted
minimization theorem nor gives a new proof of uniformization. Theorem~B
concerns the value of $E_t$ on a small Wiener ball; strong convexity holds
on smaller $H^2$ balls (Theorem~\ref{thm:mixedhessian},
Corollary~\ref{cor:galerkinflow}), and no
infinite-dimensional gradient flow is constructed. The optimality of
$\tau_*$ refers to the one-mode experiment described above. Theorem~C is a
small-time statement, and numerics show that the ranking of lattices can
reverse at moderate times. Whether the threshold $\tau_*$ is universal for
high-frequency perturbations of curved metrics is open
(Section~\ref{sec:curvedquestions}).

Section~\ref{sec:framework} fixes notation and the regression identity.
Sections~\ref{sec:kernel}--\ref{sec:closure} prove the nonlinear estimates.
Section~\ref{sec:response} computes the response and its statistical
meaning, and Sections~\ref{sec:compact}--\ref{sec:entropy} prove Theorem~B
and its variants. Section~\ref{sec:radialconsequences} proves Theorem~A and
radial stability, Section~\ref{sec:mixedhessian} treats mixed variations,
Section~\ref{sec:latticestatistics} the flat structure, and
Section~\ref{sec:curvedquestions} curved backgrounds.

\section{Periodic geometry and displacement regression}\label{sec:framework}
\subsection{Notation and conventions}\label{subsec:notation}
Throughout the periodic analysis, $\mathcal L\subset\R^2$ is a Euclidean
lattice of covolume $A$ and systole at least $\ell_0>0$,
$\mathbb T^2=T_{\mathcal L}=\R^2/\mathcal L$, and
$\mathcal L^\vee=\{k:k\cdot\ell\in2\pi\mathbb Z\ (\ell\in\mathcal L)\}$.
All Fourier sums are over $\mathcal L^\vee$, and
$\widehat v_k=A^{-1}\int_{\mathbb T^2}v(x)e^{-ik\cdot x}\,dx$.
The square torus is used only for explicitly identified examples.
We sometimes write $N_{\mathcal L}=N$ and $\mathsf A_{\mathcal L}=\mathsf A$
to emphasize the period cell. Compact constants may depend on
$A,\ell_0$; plane estimates use only the norms displayed below.

For a periodic function $v$ we write
\[
 \sA(v)=\|v\|_{\mathcal A}=\sum_k|\widehat v_k|,\quad
 \sH(v)=\|\Delta v\|_2,\quad
 \sG(v)=\|\nabla v\|_4,\quad
 \sL(v)=\|\nabla v\|_\infty,\quad
 N(v)=\sH(v)^2 .
\]
Norms use unnormalized Lebesgue measure on $\mathbb T^2$. For the perturbation $f$ we
put $\varepsilon=\sA(f)$, and $\Lambda$ bounds its Fourier frequencies;
$\mathcal I(\rho)=\int_{\mathbb T^2}|\nabla\log\rho|^2$, and $q$ always denotes the
response multiplier of Section~\ref{sec:response}.

\emph{Gaussians.} The exponents of Section~\ref{sec:kernel} are
$\kappa_\pm=(1\pm\gamma)/4$; auxiliary exponents $\kappa_0,\kappa_1$ strictly
between $\kappa_-$ and $1/4$ are fixed in Sections~\ref{naani:sec}
and~\ref{sec:secondscore}. A Roman letter with an exponent subscript,
$G_\kappa(u,w)=(\kappa/\pi)u^{-1}e^{-\kappa|w|^2/u}$, is a normalized
envelope; a calligraphic letter with a coefficient subscript,
$\mathcal G_\varrho(u,w)=\varrho(4\pi u)^{-1}e^{-\varrho|w|^2/(4u)}$, is the
exact flat kernel of $\varrho^{-1}\Delta$.

\emph{Subscripts and complex parameters.} On kernels and tensors a
subscript $s$ or $z$ ($r_s$, $h_s$, $r_z$) records the coefficient and a
subscript $t$ ($r_t$, $h_t$, $Q_t$) the observation time at a fixed
coefficient; when both matter the time is an argument, as in
$r_s(t,x,y)$. Two equivalent complex normalizations occur. In
Theorem~\ref{thm:neargaussian}(b) and Sections~\ref{sec:compclock}
and~\ref{mom:section} the coefficient is
$1+zF$ with $F=f/\varepsilon$, $\sA(F)=1$, $|z|\le r$, and real parameter
$z=s\varepsilon$; in Sections~\ref{naani:sec}
and~\ref{sec:secondscore} it is $\rho_z=1+zf$ with $|z|\varepsilon\le r$
and real parameter $z=s$. In Section~\ref{mom:section}, $\sH,\sG$
abbreviate $\sH(F),\sG(F)$, and the result is converted to
$N(f)=\varepsilon^2N(F)$ at the end.

\emph{Section~\ref{sec:response}.} The functions $\Upsilon$, $\mathsf a$,
$\mathsf b$, $p$, $\mathscr F$, $\mathscr S$, $\mathscr V$, $\mathscr L_k$ and
$\Omega$ there are auxiliary functions of the
scale variable ($\tau$, or $z=\tau/2$), unrelated to the Laplacian,
the area $A$, the score, the volume and the generator; the italic $a,b$ of
Subsection~\ref{subsec:closedform} are the integration variables
$u+v-2uv$ and $uv$.

\emph{Constants.} Unless stated otherwise, $c,C,C_\theta,C_{\theta,j},\dots$
denote positive finite constants depending only on $\theta$ (or on the
Lipschitz constant $L_*$ of Section~\ref{sec:soft}), on the fixed Gaussian
margin $\gamma$, and on the order $j$ of differentiation. Compact estimates also depend
on $A,\ell_0$ through the lattice bound and periodization. We take
$\gamma=1/2$ for the complex upper bounds of
Sections~\ref{sec:compclock} and~\ref{mom:section}, and a fixed
$\gamma<1/11$ for the real quotient estimates of Sections~\ref{naani:sec},
\ref{sec:secondscore}, \ref{sec:response}, \ref{sec:compact}
and~\ref{sec:mixedhessian}; amplitude radii are the
minimum of those obtained in these sections. Constants never depend
on the cutoff $\Lambda$, the number of Fourier modes, the observation time
in the stated range, or the profile $f$.

All source derivatives hold the target point fixed. Local proof symbols
are defined where they occur; $N(f)$ always denotes the curvature-scale
quadratic norm, whereas $\mathcal N$ denotes a residual score numerator.

\subsection{Heat-Fisher geometry and regression}
Let $\Delta=\partial_1^2+\partial_2^2$ on the flat torus defined above.
For a real smooth periodic $f$ put
\[
 \varepsilon=\sA(f),\qquad N(f)=\sH(f)^2,
 \qquad \rho_s=1+sf,\quad L_s=\rho_s^{-1}\Delta.
\]
The plane transition density $r_s(t,x,y)$ is relative to fixed Lebesgue
measure, and every source derivative holds $y$ fixed. The normalized
Fisher tensor and volume per source cell are
\[
 J_s=\int_{\mathbb R^2}\frac{\nabla_xr_s\otimes\nabla_xr_s}{r_s}\,dy,
 \qquad h_s=2tJ_s,\qquad V_t^\infty(sf)=\int_{\mathbb T^2}\sqrt{\det h_s}\,dx.
\]
For $g_s=\rho_s g_0$ the Riemannian heat density is $p_s=r_s/\rho_s(y)$;
its source score is unchanged, and in two dimensions the conformal
density factors cancel in the volume formula. If $f$ has mean zero, define
\[
 E_t^\infty(sf)=\frac9{t^2}\bigl(A-V_t^\infty(sf)\bigr),\qquad
 \mathcal C(sf)=\int_{\mathbb T^2} K_{g_s}^2\,dA_{g_s},\qquad
 R_t^\infty=E_t^\infty-\mathcal C.
\]
The compact versions $h_s^T,V_t^T,E_t^T,R_t^T$ use the periodized
transition and are compared with the plane versions in
Section~\ref{sec:compact}.

\paragraph{Jeffreys volume and source-model complexity.}
For a fixed compact surface metric $g$ and time $t>0$, consider the
statistical model $\{p_t^g(x,\cdot)\,dA_g:x\in\mathcal M\}$ in which the
unknown parameter is the source $x$. Its Fisher tensor is precisely
$J_t^g$. The coordinate-invariant Jeffreys measure~\cite{Jeffreys1946} is
$d\operatorname{vol}_{J_t^g}$, with normalizing constant
\begin{equation}\label{eq:jeffreysvolume}
 Z_t(g)=\int_{\mathcal M}d\operatorname{vol}_{J_t^g}
       =\frac{V_t(g)}{2t}.
\end{equation}
The factor $2t$ is essential: $V_t$ is the \emph{normalized} Jeffreys
volume in dimension two. For $n$ independent endpoints from the same
unknown source, the regular fixed-$t$ parametric-complexity asymptotics
of Rissanen~\cite{Rissanen1996} have the form
\[
 \log C_n(t,g)=\log\frac{n}{2\pi}+\log Z_t(g)+o(1),\qquad n\to\infty.
\]
This interpretation presupposes the usual uniform regularity and
identifiability conditions for that asymptotic statement; it is not a
joint assertion uniform in $t\downarrow0$. Its geometric constant is
$\log Z_t$, rather than $Z_t$ or $V_t$ themselves.
At fixed area and time, minimizing $E_t$ is exactly maximizing $Z_t$.
The resolved-band theorem (Theorem~\ref{thm:flatlattices}: $0<t<t_*$,
Wiener amplitude at most $\varepsilon_*$, $t\Lambda^2\le\theta<\tau_*$) therefore
makes the flat metric a strict local maximizer of the source model's Jeffreys volume within the permitted
mean-zero conformal perturbations. This is a finite-scale, local analogue
of the variational role of uniform metrics in the determinant extremal
theorem of Osgood--Phillips--Sarnak~\cite{OPS}; it is neither an identity
with $\det'\Delta$ nor a global extremal theorem for $V_t$.

\begin{lemma}[Interpolation and Wiener bounds]\label{lem:interp}
\begin{enumerate}
\item[(a)] For real smooth periodic $v$,
\begin{equation}\label{eq:GNbasic}
 \|\nabla v\|_4^2\le\sqrt{10}\,\|v\|_\infty\|\Delta v\|_2 .
\end{equation}
\item[(b)] For complex smooth periodic $v$,
$\|\nabla v\|_4^2\le2\sqrt5\,\|v\|_\infty\|\Delta v\|_2$, where
$|\nabla v|^2=\nabla v\cdot\overline{\nabla v}$.
\item[(c)] If $v$ has mean zero, then
\begin{equation}\label{eq:WienerHtwo}
 \sA(v)\le C_W\sH(v),\qquad
 C_W^2=\frac1A\sum_{k\in\mathcal L^\vee\setminus\{0\}}|k|^{-4}<\infty.
\end{equation}
\end{enumerate}
\end{lemma}
\begin{proof}
(a) Periodic integration by parts gives
$\int_{\mathbb T^2}|\nabla v|^4dx=-\int_{\mathbb T^2} v\,\langle D^2v,\;|\nabla v|^2I
 +2\nabla v\otimes\nabla v\rangle\,dx$.
The matrix $|\nabla v|^2I+2\nabla v\otimes\nabla v$ has eigenvalues
$3|\nabla v|^2$ and $|\nabla v|^2$, hence Frobenius norm
$\sqrt{10}\,|\nabla v|^2$. Cauchy--Schwarz in the Frobenius pairing and in
$L^2$, with $\|D^2v\|_2=\|\Delta v\|_2$, gives
$\|\nabla v\|_4^4\le\sqrt{10}\|v\|_\infty\|\Delta v\|_2\|\nabla v\|_4^2$.
(b) Write $v=v_1+iv_2$. Then
$\|\nabla v\|_4^2\le\|\nabla v_1\|_4^2+\|\nabla v_2\|_4^2$, and (a) with
$\|\Delta v_1\|_2+\|\Delta v_2\|_2\le\sqrt2\|\Delta v\|_2$ gives the claim.
(c) Cauchy--Schwarz on $\sum_{k\ne0}|k|^{-2}\cdot|k|^2|\widehat v_k|$
and Parseval give the inequality. Lemma~\ref{lem:latticeWiener} below
gives a uniform bound in terms of $A,\ell_0$.
\end{proof}
The constant in (a) is not optimal but is carried explicitly through the
later estimates and Proposition~\ref{prop:curvatureequivalence}. Since
$\|f\|_\infty\le\sA(f)$, (a) gives the form
$\sG(f)^2\le\sqrt{10}\,\varepsilon\,\sH(f)$ used throughout. Lemma~\ref{lem:cauchyjets} below passes from complex bounds to real
amplitude jets.

\begin{lemma}[A uniform lattice bound]\label{lem:latticeWiener}
If $\operatorname{sys}(\mathcal L)\ge\ell_0>0$ and $f$ has mean zero, then
\[
 \mathsf A_{\mathcal L}(f)\le C_W(\mathcal L)\|\Delta f\|_2,\qquad
 C_W(\mathcal L)^2=\frac1A\sum_{\xi\in\mathcal L^\vee\setminus\{0\}}|\xi|^{-4}
 \le\frac{40}{3A}\left(\frac{A}{2\pi\ell_0}\right)^4.
\]
\end{lemma}
\begin{proof}
For the quarter-turn $J$, $\mathcal L^\vee=(2\pi/A)J\mathcal L$ (an oriented
basis $v_1,v_2$ of determinant $A$ has reciprocal basis
$-2\pi Jv_2/A,2\pi Jv_1/A$), so
distinct dual points are at least
$\ell_0^\vee=2\pi\ell_0/A$ apart. The disks of radius $\ell_0^\vee/2$ about dual points of length
at most $R$, including the origin, are disjoint and lie in the
disk of radius $R+\ell_0^\vee/2$. Hence the nonzero counting function obeys
$\#(R)\le4R^2/(\ell_0^\vee)^2+4R/\ell_0^\vee$, with $\#(R)=0$ for $R<\ell_0^\vee$,
and by Tonelli
$\sum_{\xi\ne0}|\xi|^{-4}
 =4\int_{\ell_0^\vee}^\infty\#(R)R^{-5}\,dR
 \le(8+\frac{16}3)(\ell_0^\vee)^{-4}=\frac{40}3(\ell_0^\vee)^{-4}$.
Cauchy--Schwarz and Parseval finish the proof.
\end{proof}

\begin{lemma}[Cauchy jets along a real segment]\label{lem:cauchyjets}
Let $X$ be a complex Banach space, $r>0$ and $0<\varepsilon\le r/4$. Let
$\Phi$ be holomorphic from $\{|\zeta|<r\}$ to $X$ with $\|\Phi(\zeta)\|\le M$
for $|\zeta|\le3r/4$. Then, for real $|s|\le1$ and every integer $j\ge0$,
\[
 \Bigl\|\frac{d^j}{ds^j}\Phi(s\varepsilon)\Bigr\|
 \le j!\Bigl(\frac2r\Bigr)^j\varepsilon^jM .
\]
\end{lemma}
\begin{proof}
The derivative equals $\varepsilon^j\Phi^{(j)}(s\varepsilon)$. The circle
of radius $r/2$ about $s\varepsilon$ lies in $|\zeta|\le\varepsilon+r/2\le3r/4$,
and Cauchy's estimate on it gives $\|\Phi^{(j)}(s\varepsilon)\|\le
j!(2/r)^jM$.
\end{proof}

\begin{lemma}[Regression and ellipticity]\label{lem:regression}
For a positive smooth periodic $\rho$ in the small-amplitude class
of Theorem~\ref{thm:neargaussian}, let $S=\nabla_x\log r_t(x,X_t)$,
$w=X_t-x$, and
\[
 m_t=\frac1t\int_0^tP_u(\rho^{-1})\,du,\quad
 \alpha_t=m_t^{-1},\quad
 Q_t=\mathbb E_x[(S-w/(2tm_t))^{\otimes2}].
\]
Then
\[
 \mathbb E_x w=0,\qquad \mathbb E_x[ww^T]=2tm_t I,\qquad
 \mathbb E_x[Sw^T]=I,\qquad h_t=\alpha_tI+2tQ_t.
\]
In particular $Q_t\succeq0$, and $h_t\succeq(1-\varepsilon)I$ for
$\rho=1+sf$, $|s|\le1$ and $\varepsilon<1$.
\end{lemma}
\begin{proof}
It\^o's formula for $dX_u=\sqrt{2/\rho(X_u)}\,dW_u$ gives the first two
identities. Differentiating the zero displacement mean at fixed target gives
the third; the source Gaussian bound of Theorem~\ref{thm:neargaussian}
justifies differentiation and the displacement moments. Expanding the
residual square proves the fourth (for the semigroup calculus see
\cite{BGL}). Since $m_t\le(1-\varepsilon)^{-1}$, the last assertion follows.
\end{proof}

With $\mathcal I(\rho)=\int_{\mathbb T^2}|\nabla\log\rho|^2dx$ define
\begin{align}
 \mathfrak A_t&=\int_{\mathbb T^2}(\alpha_t-\rho)dx+\frac t2\mathcal I(\rho),\notag\\
 \mathfrak B_t&=t\int_{\mathbb T^2}\operatorname{tr}Q_tdx-\frac t2\mathcal I(\rho),\label{eq:ABdef}\\
 D_t&=\int_{\mathbb T^2}\left(\frac12\operatorname{tr}h_t-\sqrt{\det h_t}\right)dx.
 \notag
\end{align}
The two-dimensional determinant identity and $\int_{\mathbb T^2}\rho=A$
(mean zero) give
\begin{equation}\label{eq:master}
 A-V_t^\infty=-\mathfrak A_t-\mathfrak B_t+D_t,
 \qquad E_t^\infty=\frac9{t^2}(-\mathfrak A_t-\mathfrak B_t+D_t).
\end{equation}
The proof estimates each compensated term before summing them;
Figure~\ref{fig:closure} shows the decomposition and where each term is
controlled.
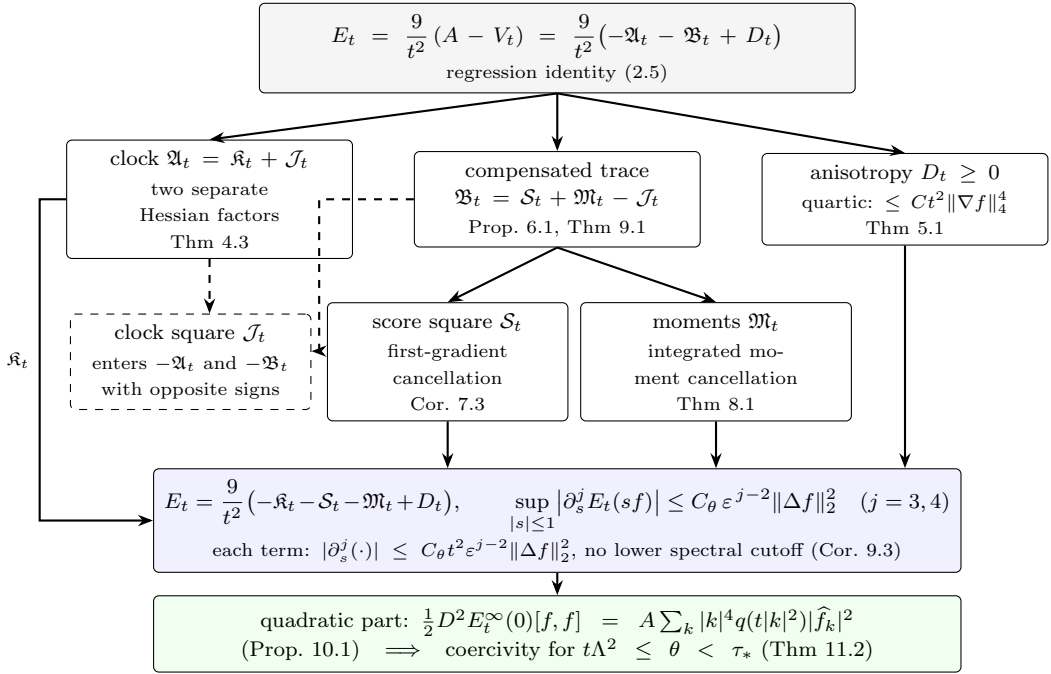
\begin{figure}[tbp]
\centering
\begin{tikzpicture}[
 box/.style={draw,rounded corners=2pt,align=center,font=\footnotesize,inner sep=4pt},
 note/.style={font=\scriptsize,align=center,text=black!75},
 arr/.style={-{Stealth[length=2mm]},thick}]
\node[box,fill=gray!8,text width=7.6cm] (E) at (0,0)
 {$E_t=\dfrac9{t^2}\,(A-V_t)=\dfrac9{t^2}\bigl(-\mathfrak A_t-\mathfrak B_t+D_t\bigr)$
  \\[1pt] {\scriptsize regression identity \eqref{eq:master}}};
\node[box,text width=3.5cm] (A) at (-4.6,-2.0)
 {clock $\mathfrak A_t=\mathfrak K_t+\mathcal J_t$\\[1pt]
  {\scriptsize two separate Hessian factors}\\{\scriptsize Thm~\ref{naclock:thm:clockderivative}}};
\node[box,text width=3.5cm] (B) at (0,-2.0)
 {compensated trace\\ $\mathfrak B_t=\mathcal S_t+\mathfrak M_t-\mathcal J_t$\\[1pt]
  {\scriptsize Prop.~\ref{namom:identity}, Thm~\ref{thm:scalar}}};
\node[box,text width=3.5cm] (D) at (4.6,-2.0)
 {anisotropy $D_t\ge0$\\[1pt]{\scriptsize quartic: $\le C t^2\|\nabla f\|_4^4$}\\
  {\scriptsize Thm~\ref{naani:main}}};
\node[box,text width=2.9cm] (S) at (-1.45,-4.15)
 {score square $\mathcal S_t$\\[1pt]{\scriptsize first-gradient cancellation}\\
  {\scriptsize Cor.~\ref{score:cor:square}}};
\node[box,text width=3.3cm] (M) at (2.1,-4.15)
 {moments $\mathfrak M_t$\\[1pt]{\scriptsize integrated moment cancellation}\\
  {\scriptsize Thm~\ref{mom:main}}};
\node[box,dashed,text width=2.9cm] (J) at (-4.85,-4.15)
 {clock square $\mathcal J_t$\\[1pt]{\scriptsize enters $-\mathfrak A_t$ and $-\mathfrak B_t$}\\
  {\scriptsize with opposite signs}};
\node[box,fill=blue!6,text width=10.4cm] (R) at (0,-6.25)
 {$\displaystyle E_t=\frac9{t^2}\bigl(-\mathfrak K_t-\mathcal S_t-\mathfrak M_t+D_t\bigr),\qquad
  \sup_{|s|\le1}\bigl|\partial_s^jE_t(sf)\bigr|\le C_\theta\,\varepsilon^{\,j-2}\|\Delta f\|_2^2\quad(j=3,4)$\\[1pt]
  {\scriptsize each term: $|\partial_s^j(\cdot)|\le C_\theta t^2\varepsilon^{j-2}\|\Delta f\|_2^2$, no lower spectral cutoff
  (Cor.~\ref{cor:planevariations})}};
\node[box,fill=green!6,text width=10.4cm] (Q) at (0,-7.75)
 {quadratic part: $\tfrac12D^2E_t^\infty(0)[f,f]=A\sum_k|k|^4q(t|k|^2)|\widehat f_k|^2$
  (Prop.~\ref{prop:response}) $\;\Longrightarrow\;$ coercivity for $t\Lambda^2\le\theta<\tau_*$
  (Thm~\ref{thm:stability})};
\draw[arr] (E.south) -- (A.north);
\draw[arr] (E.south) -- (B.north);
\draw[arr] (E.south) -- (D.north);
\draw[arr] (B.south) -- (S.north);
\draw[arr] (B.south) -- (M.north);
\draw[arr,dashed] (A.south) -- (A.south |- J.north);
\draw[arr,dashed] (B.west) -- (-3.15,0 |- B.west) |- ([yshift=4pt]J.east);
\draw[arr] (S.south) -- (S.south |- R.north);
\draw[arr] (M.south) -- (M.south |- R.north);
\draw[arr] (D.south) -- (D.south |- R.north);
\draw[arr] (A.west) -- ++(-0.35,0) |- node[pos=0.25,left,font=\scriptsize]{$\mathfrak K_t$} (R.west);
\draw[arr] (R.south) -- (Q.north);
\end{tikzpicture}
\caption{The energy decomposition behind the nonlinear estimates. The
regression identity splits the deficit into the clock $\mathfrak A_t$, the
compensated trace $\mathfrak B_t$ and the anisotropy $D_t$. The clock square
$\mathcal J_t$ occurs in $\mathfrak A_t$ and in $\mathfrak B_t$ with opposite
signs and cancels exactly (dashed), leaving four terms. Each is controlled
at the relative energy scale $t^2\|\Delta f\|_2^2$ without a lower
spectral cutoff: the clock retains two separate Hessian factors, the score
square and the moments are compensated by exact cancellations, and the
anisotropy is quartic. Combined with the exact quadratic multiplier, this
gives the main theorem.}\label{fig:closure}
\end{figure}
A mixed source norm is denoted
$\|F(t;w)\|_{[p]}=\|x\mapsto F(t,x,x+w)\|_{L^p(\mathbb T^2)}$;
the source derivative is always taken before this displacement substitution.

\paragraph{The only role of the resolution hypothesis in Sections~\ref{sec:kernel}--\ref{sec:closure}.}
For a smooth real periodic $f$ with $\varepsilon=\sA(f)<\infty$ and
$L_*>0$, consider the diffusion-scale Lipschitz condition
\begin{equation}\label{eq:HL}
 \sqrt t\,\|\nabla f\|_\infty\le L_*\varepsilon .
\end{equation}
If $\operatorname{supp}\widehat f\subset\{|k|\le\Lambda\}$ and
$t\Lambda^2\le\theta$, then $|\nabla f|\le\sum_k|k||\widehat f_k|\le
\Lambda\varepsilon$, so \eqref{eq:HL} holds with $L_*=\sqrt\theta$. In
Sections~\ref{sec:kernel}--\ref{sec:closure} the resolution hypotheses
enter the proofs only through this consequence: in
Step~1 of the proof of Theorem~\ref{thm:neargaussian} (the bound
$\|\nabla\phi\|_\infty\le\sqrt\theta$ and \eqref{eq:paramlip}), in the
convergence of the anisotropy and score defect series (the bound
$|z|\sqrt t\|\nabla f\|_\infty\le r\sqrt\theta$ for $|z|\varepsilon\le r$),
and through these kernel bounds in the clock and moment sections; the
interpolation inequalities of Lemma~\ref{lem:interp} hold for all smooth
periodic functions. Hence every result of
Sections~\ref{sec:kernel}--\ref{sec:closure} holds verbatim for smooth
real periodic $f$ satisfying \eqref{eq:HL}, with $\sqrt\theta$ replaced by
$L_*$ and constants depending on $L_*$ (and on $\gamma$ and the jet
order) instead of $\theta$. We state them in the resolved form; the
general form is used in Section~\ref{sec:soft}.

\section{Uniform kernels on a resolved Wiener ball}
\label{sec:kernel}
We use Levi's parametrix \cite{Levi} in its classical target-frozen
parabolic form \cite{Friedman,IKO}, recording the Gaussian exponents and
parameter dependence needed for real Fisher quotients. Classical two-sided
Gaussian bounds \cite{Aronson,Davies,Grigoryan} do not by themselves give
exponents $\kappa_\pm$ arbitrarily close to $1/4$ or holomorphic parameter
dependence. All Levi estimates of the paper (here and in
Sections~\ref{naani:sec}, \ref{sec:secondscore} and~\ref{sec:mixedhessian})
rest on the two lemmas below.

\subsection{Frozen kernels and marked Levi sums}\label{subsec:levitools}
Let $g$ be a smooth complex periodic function, $\rho=1+g$ and $a=\rho^{-1}$.
For a complex coefficient value $\varrho$ let
$\mathcal G_\varrho(u,v)=\varrho(4\pi u)^{-1}e^{-\varrho|v|^2/(4u)}$. The
target-frozen kernel $Z(u,x,y)=\mathcal G_{\rho(y)}(u,x-y)$ and its defect
$K=(a(x)-a(y))\Delta_xZ$ satisfy $(\partial_u-a(x)\Delta_x)Z=-K$. The Levi
series is $r=Z+\sum_{n\ge1}Z*K^{*n}$, with
$(U*W)(u,x,y)=\int_0^u\int U(u-v,x,\xi)W(v,\xi,y)\,d\xi\,dv$. All kernels
are simultaneously periodic, the mixed norms are those of
Section~\ref{sec:framework}, and in a convolution chain every intermediate
point is written as the source plus a displacement.

\begin{lemma}[Frozen kernels]\label{lem:frozen}
Let $\Re\varrho\ge\varrho_0>0$ and $|\varrho|\le2$ for all coefficient values
used, and let $0<\kappa<\varrho_0/4$. Then, with $G=G_\kappa$ and
$p\in\{4,\infty\}$,
\begin{align}
 |\nabla_v^\ell\mathcal G_\varrho(u,v)|&\le Cu^{-\ell/2}G(u,v),\qquad
 \ell=0,1,2,\label{eq:frozenG}\\
 |\nabla_v^\ell(\mathcal G_{\varrho_1}-\mathcal G_{\varrho_0'})(u,v)|
   &\le C|\varrho_1-\varrho_0'|u^{-\ell/2}G(u,v),\qquad\ell=0,1,
   \label{eq:frozenGdiff}\\
 \|K(u;w)\|_{[p]}&\le C\|\nabla a\|_p\,u^{-1/2}G(u,w).\label{eq:frozenK}
\end{align}
The same bounds hold for complex time $\mathfrak t=u\omega$, $|\omega-1|\le\eta$,
whenever $\Re(\varrho/\omega)\ge\varrho_0$. The constants depend only on
$\varrho_0$, the gap $\varrho_0/4-\kappa$ and, for complex time, on $\eta$.
\end{lemma}
\begin{proof}
One has $\nabla_v\mathcal G_\varrho=-\varrho v\mathcal G_\varrho/(2u)$,
$\Delta_v\mathcal G_\varrho=(-\varrho/u+\varrho^2|v|^2/(4u^2))\mathcal G_\varrho$
and $|\mathcal G_\varrho|\le Cu^{-1}e^{-\varrho_0|v|^2/(4u)}$; the gap absorbs
the polynomials in $|v|/\sqrt u$, giving \eqref{eq:frozenG}.
For \eqref{eq:frozenGdiff} integrate the $\varrho$-derivative along the segment
from $\varrho_0'$ to $\varrho_1$. For \eqref{eq:frozenK}, $|K|\le|a(x)-a(x+w)|
\cdot Cu^{-2}(1+|w|^2/u)e^{-\varrho_0|w|^2/(4u)}$, while
$a(x+w)-a(x)=\int_0^1\nabla a(x+\lambda w)\cdot w\,d\lambda$, Minkowski's
inequality and periodic translation give
$\|a(\cdot)-a(\cdot+w)\|_p\le|w|\|\nabla a\|_p$; the gap absorbs $|w|/\sqrt u$
and $|w|^2/u$. For complex time replace $\varrho$ by $\varrho/\omega$.
\end{proof}

\begin{lemma}[Marked beta summation]\label{lem:markedbeta}
Let $G=G_\kappa$ and let $Z,K$ be kernels with $|Z|\le C_0G$,
$|\nabla_xZ|\le C_0u^{-1/2}G$ and
$\|K(u;w)\|_{[p]}\le C_0L_pu^{-1/2}G(u,w)$ for $p\in\{4,\infty\}$ (for $m=0$
only $p=\infty$ is needed). Let
$m\in\{0,1,2\}$ and $1/p_m=m/4$. For $n\ge\max(m,1)$,
\begin{align*}
 \|Z*K^{*n}(t;w)\|_{[p_m]}
 &\le C_0^{n+1}L_4^mL_\infty^{n-m}
   \frac{\Gamma(1/2)^n}{\Gamma(1+n/2)}t^{n/2}G(t,w),\\
 \|(\nabla_xZ)*K^{*n}(t;w)\|_{[p_m]}
 &\le C_0^{n+1}L_4^mL_\infty^{n-m}
   \frac{\Gamma(1/2)^{n+1}}{\Gamma((n+1)/2)}t^{(n-1)/2}G(t,w),
\end{align*}
where the source derivative acts only on the leftmost factor, at fixed
intermediate variables. If $L_\infty\sqrt t\le M$ and $m'=\max(m,1)$,
the sums over $n\ge m'$ are at most
$C(C_0,M)L_4^mL_\infty^{m'-m}t^{m'/2}G(t,w)$ and
$C(C_0,M)L_4^mL_\infty^{m'-m}t^{(m'-1)/2}G(t,w)$, respectively. For $m=0$
the $[\infty]$ norm is the pointwise bound. The same bounds hold along a
complex time ray, with constants changed by the ray length factors.
\end{lemma}
\begin{proof}
Put $m$ factors $K$ in $L^4$ and all others in $L^\infty$. Minkowski's
inequality in the intermediate displacements, periodic translation, and
H\"older's inequality $\|h(x+v)k(x+v')\|_{L^2_x}\le\|h\|_4\|k\|_4$ for the
marked factors bound the $[p_m]$ norm of each chain by the corresponding
scalar chain. Normalized Gaussians convolve exactly in the displacement,
and the Dirichlet integral
$\int_{u_0+\dots+u_n=t}\prod_ju_j^{\alpha_j-1}
=t^{\sum\alpha_j-1}\prod\Gamma(\alpha_j)/\Gamma(\sum\alpha_j)$ with
$\alpha_0=1$ (respectively $\frac12$) and $\alpha_j=\frac12$ gives the time
factors. The sums converge since $\Gamma(1+n/2)$ grows faster than any
exponential.
\end{proof}

\subsection{Near-Gaussian parameter jets}
\begin{theorem}[Resolved near-Gaussian parameter jets]\label{thm:neargaussian}
Fix $\theta<\infty$ and $0<\gamma<1$.
There are $\varepsilon_0=\varepsilon_0(\theta,\gamma)>0$, $c>0$, and
constants $C_j$, $j\ge0$, and $C_\theta$, such that the following holds.

\smallskip\noindent\emph{(a) Real parameters.} Let $f$ be a real trigonometric
polynomial supported in $|k|\le\Lambda$, with
$\varepsilon=\|f\|_{\mathcal A}\le\varepsilon_0$ and $t\Lambda^2\le\theta$.
For the fixed-Lebesgue plane kernel $r_s$ of $(1+sf)^{-1}\Delta$, $|s|\le1$,
put $\kappa_-=(1-\gamma)/4$ and $\kappa_+=(1+\gamma)/4$. Then
\begin{equation}\label{eq:nonzerolower}
 r_s(t,x,y)\ge c t^{-1}e^{-\kappa_+|y-x|^2/t},
\end{equation}
and, for every integer $j\ge0$,
\begin{align}
 |\partial_s^jr_s|+|t\partial_t\partial_s^jr_s|
 &\le C_j\varepsilon^jt^{-1}e^{-\kappa_-|y-x|^2/t},\label{eq:nonzerodensityjets}\\
 |\nabla_x\partial_s^jr_s|+|t\partial_t\nabla_x\partial_s^jr_s|
 &\le C_j\varepsilon^jt^{-3/2}e^{-\kappa_-|y-x|^2/t}.\label{eq:nonzerosourcejets}
\end{align}
The constants are independent of the cutoff, mode count, and period after
rescaling. The indicated derivatives commute. Time differentiation is at
fixed coefficient $f$. The case $f=0$ is interpreted by the exact flat kernel.

\smallskip\noindent\emph{(b) Complex parameters.} There is
$\bar r_\theta=\bar r_\theta(\theta,\gamma)>0$ with the following property. Let $F$ be a real
trigonometric polynomial supported in $|k|\le\Lambda$, with
$\|F\|_{\mathcal A}\le1$, and let $t\Lambda^2\le\theta$. The plane kernel
$r_z$ of $(1+zF)^{-1}\Delta$ is holomorphic in $|z|<2\bar r_\theta$, and,
uniformly for $0<u\le t$ and $|z|\le\bar r_\theta$,
\begin{equation}\label{eq:complexkernelupper}
 |r_z(u,x,y)|\le C_\theta u^{-1}e^{-\kappa_-|x-y|^2/u},\qquad
 |\nabla_xr_z(u,x,y)|\le C_\theta u^{-3/2}e^{-\kappa_-|x-y|^2/u},
\end{equation}
with the same $\kappa_-$. No lower bound or positivity is asserted for
nonreal $z$.
\end{theorem}
\begin{proof}
\emph{Step 1: freeze the scaling before differentiating.}
Fix the physical target time $t$ and an anchor $x_0$, rescale space by
$\sqrt t$, and put $\phi(\xi)=f(x_0+\sqrt t\,\xi)/\varepsilon$. The Wiener
bound gives $\|\phi\|_\infty\le1$ and $\|\nabla\phi\|_\infty\le\sqrt\theta$;
only these enter below, and the rescaled period enters no constant. For a
complex parameter $\zeta$ set $\rho_\zeta=1+\zeta\phi$, $a_\zeta=\rho_\zeta^{-1}$.
For $|\zeta|\le r_0<1/4$,
\begin{equation}\label{eq:paramlip}
 |a_\zeta(x)-a_\zeta(y)|\le L_\zeta|x-y|,
 \qquad L_\zeta=\frac{|\zeta|\sqrt\theta}{(1-r_0)^2}.
\end{equation}
The real parameter is $\zeta=s\varepsilon$. The physical kernel is
$t^{-1}$ times the rescaled kernel at time one; the anchor and the rescaled
profile stay fixed under source and time differentiation, so
$t\partial_{t'}$ at $t'=t$ is the dimensionless time derivative at one.

\emph{Step 2: a convergent complex parametrix.}
For complex time $\mathfrak t$ near a positive real time, use
$Z_\zeta(\mathfrak t,x,y)=\rho_\zeta(y)(4\pi\mathfrak t)^{-1}
\exp(-\rho_\zeta(y)|x-y|^2/(4\mathfrak t))$ and
$K_\zeta=(a_\zeta(x)-a_\zeta(y))\Delta_xZ_\zeta$, so that
$(\partial_\mathfrak t-a_\zeta(x)\Delta_x)Z_\zeta=-K_\zeta$.
Choose $r_0,\eta>0$ with $(r_0+\eta)/(1-\eta)\le\gamma/4$.
If $\mathfrak t=u\omega$ with $u>0$ and $|\omega-1|\le\eta$, then
$|\rho_\zeta(y)/\omega-1|\le(r_0+\eta)/(1-\eta)$, so
$\Re(\rho_\zeta(y)/\omega)\ge1-\gamma/4$.
Set $G=G_{\kappa_-}$. Since $\kappa_-<(1-\gamma/4)/4$, Lemma~\ref{lem:frozen},
applied in rescaled variables where $\|\nabla a_\zeta\|_\infty\le L_\zeta$ by
\eqref{eq:paramlip}, gives along such rays
\begin{equation}\label{eq:levibounds}
 |Z_\zeta|\le CG,\qquad
 |\nabla_xZ_\zeta|\le Cu^{-1/2}G,\qquad
 |K_\zeta|\le CL_\zeta u^{-1/2}G.
\end{equation}
For complex time the convolution is taken along the straight ray,
$(U*W)(\mathfrak t,x,y)=\mathfrak t\int_0^1\int U((1-v)\mathfrak t,x,\xi)
W(v\mathfrak t,\xi,y)\,d\xi\,dv$. Lemma~\ref{lem:markedbeta} with $m=0$
gives
\begin{equation}\label{eq:levicorrection}
 |Z_\zeta*K_\zeta^{*n}|\le
 \frac{C(CL_\zeta)^n\Gamma(1/2)^n}{\Gamma(1+n/2)}u^{n/2}G,\qquad
 |\nabla_x(Z_\zeta*K_\zeta^{*n})|
 \le\frac{C(CL_\zeta)^n\Gamma(1/2)^{n+1}}{\Gamma((n+1)/2)}u^{(n-1)/2}G,
\end{equation}
and the same Dirichlet integral gives
$|K_\zeta^{*n}|\le(CL_\zeta)^n\Gamma(1/2)^n\Gamma(n/2)^{-1}u^{n/2-1}G$.
Hence $\Phi_\zeta=\sum_{n\ge1}K_\zeta^{*n}$ and
$\mathcal R_\zeta=Z_\zeta+Z_\zeta*\Phi_\zeta$ converge normally on compact
subsets of the parameter and time domains, together with the source
derivative, and $|\mathcal R_\zeta|\le C_\theta G$,
$|\nabla_x\mathcal R_\zeta|\le C_\theta u^{-1/2}G$ uniformly for $0<u\le2$.
Factoring the first power of $L_\zeta\sqrt u$ in
\eqref{eq:levicorrection} also gives, on the real time axis,
\begin{equation}\label{eq:levilocalerror}
 |\mathcal R_\zeta-Z_\zeta|
 \le C_\theta L_\zeta\sqrt u\,G,\qquad 0<u\le1.
\end{equation}
Distributionally,
$(\partial_\mathfrak t-a_\zeta\Delta)(Z_\zeta*\Phi_\zeta)
=\Phi_\zeta-K_\zeta*\Phi_\zeta$, which cancels $-K_\zeta$ since
$\Phi_\zeta=K_\zeta+K_\zeta*\Phi_\zeta$.
The initial trace of $Z_\zeta$ is the delta kernel; the correction has zero
initial trace by \eqref{eq:levilocalerror}. For real positive coefficients,
uniqueness of bounded distributional heat solutions, followed by regularity
for smooth uniformly elliptic coefficients, identifies $\mathcal R_\zeta$
with the transition density. Normal convergence of the complex integral
formulas proves holomorphy in $\zeta$ and $\mathfrak t$.

\emph{Step 3: parameter and time jets.}
Assume $\varepsilon\le r_0/4$. A parameter circle of radius $r_0/2$ about
$\zeta=s\varepsilon$ remains inside $|\zeta|<r_0$, so Cauchy's formula and
$\partial_s=\varepsilon\partial_\zeta$ give
$|\partial_s^j\mathcal R_\zeta|\le j!(2\varepsilon/r_0)^jC_\theta G$ and
$|\nabla_x\partial_s^j\mathcal R_\zeta|\le j!(2\varepsilon/r_0)^jC_\theta u^{-1/2}G$.
A time circle of radius $\eta u/2$ gives the same bounds for
$u\partial_u$ and the indicated mixed derivatives, with an extra factor at
most $2/\eta$. Holomorphy and normal convergence justify commutation of all
derivatives in the theorem.

\emph{Step 4: a near-diagonal real comparison.}
For real $|\zeta|\le\varepsilon$, let $a_-=1/(1+\varepsilon)$ and let
$\mathcal G_{1+\varepsilon}(u,w)=(4\pi a_-u)^{-1}e^{-|w|^2/(4a_-u)}$ be the
flat kernel for the constant coefficient $\rho=1+\varepsilon$. The exact
frozen formula yields
$Z_\zeta(u,x,y)\ge\frac{1-\varepsilon}{1+\varepsilon}\mathcal G_{1+\varepsilon}(u,x-y)$.
Let $R_*>2$ be the number fixed in Step~5; it depends only on $\gamma$. On
$|x-y|\le R_*\sqrt u$, $u\le1$, the ratio of the error in
\eqref{eq:levilocalerror} to $\mathcal G_{1+\varepsilon}$ is at most
$C_\theta\varepsilon\sqrt\theta\exp([(1+\varepsilon)/4-\kappa_-]R_*^2)$.
Reduce $\varepsilon_0$ so this is at most $1/4$ and
$\varepsilon_0\le1/8$. Since $(1-\varepsilon)/(1+\varepsilon)\ge7/9$,
\begin{equation}\label{eq:localgaussianlower}
 \mathcal R_\zeta(u,x,y)\ge\tfrac12\mathcal G_{1+\varepsilon}(u,x-y),
 \qquad |x-y|\le R_*\sqrt u,\ u\le1.
\end{equation}

\emph{Step 5: an explicit global chaining estimate.}
Write $w=y-x$, $n=\max\{1,\lceil |w|^2/[u(R_*-2)^2]\rceil\}$ and $u_n=u/n$.
Use the positive real semigroup and restrict the $n-1$ intermediate points to
$x_j=x+jw/n+\delta_j$, $|\delta_j|\le\sqrt{u_n}$, $\delta_0=\delta_n=0$.
Each increment has length at most $R_*\sqrt{u_n}$, so
\eqref{eq:localgaussianlower} applies to every factor, and the cross terms
telescope exactly:
$\sum_{j=1}^n|w/n+\delta_j-\delta_{j-1}|^2
=|w|^2/n+\sum_{j=1}^n|\delta_j-\delta_{j-1}|^2\le|w|^2/n+4nu_n$.
Each restricted disk has area $\pi u_n$. Hence
\[
 \mathcal R_\zeta(u,x,y)
 \ge 2^{-n}(4\pi a_-u_n)^{-n}(\pi u_n)^{n-1}
       e^{-|w|^2/(4a_-u)-n/a_-}
 =\frac1{\pi u_n}
       e^{-n[\log(8a_-)+1/a_-]}e^{-|w|^2/(4a_-u)}.
\]
For $1/2\le a_-\le1$, $\log(8a_-)+1/a_-\le4$.
Using $u_n\le u$ and $n\le1+|w|^2/[u(R_*-2)^2]$, a deliberately weakened
bound is
\begin{equation}\label{eq:sharpchainedlower}
 \mathcal R_\zeta(u,x,y)\ge\frac{e^{-4}}{4\pi u}
 \exp\left[-\left(\frac{1+\varepsilon}{4}
           +\frac4{(R_*-2)^2}\right)\frac{|w|^2}{u}\right].
\end{equation}
Choose $R_*$ so that $4/(R_*-2)^2\le\gamma/8$, then reduce
$\varepsilon_0$ so $\varepsilon_0\le\gamma/2$; the exponent in
\eqref{eq:sharpchainedlower} is then at most $\kappa_+$.
All choices depend only on $\theta,\gamma$. Restoring the fixed physical
scaling gives \eqref{eq:nonzerolower}--\eqref{eq:nonzerosourcejets}.
\end{proof}

\begin{proof}[Proof of part (b)]
Apply Steps~1 and~2 with $f$, $\varepsilon$, $\zeta$ replaced by $F$, $1$,
$z$, so $\phi(\xi)=F(x_0+\sqrt t\,\xi)$ has $\|\phi\|_\infty\le1$ and
Lipschitz constant at most $\sqrt\theta$; Step~2 applies for $|\zeta|<r_0$,
with $r_0$ depending only on $\gamma$. Put $\bar r_\theta=r_0/2$. Physical
times $0<u\le t$ correspond to rescaled times in $(0,1]$, and the bounds on
$\mathcal R_\zeta$ and $\nabla_x\mathcal R_\zeta$ of Step~2 become
\eqref{eq:complexkernelupper} after restoring the physical scaling.
Normal convergence gives holomorphy. (The proof gives existence of a
cutoff-independent amplitude radius, not a useful numerical value.)
\end{proof}

\section{The compensated displacement clock}
\label{sec:compclock}
For a smooth positive periodic function $\rho$, put
\[
 L_\rho=\rho^{-1}\Delta,\qquad P_u^\rho=e^{uL_\rho},\qquad a_\rho=\rho^{-1},
 \qquad m_t=\frac1t\int_0^tP_u^\rho a_\rho\,du,\qquad \alpha_t=m_t^{-1}.
\]
The same semigroup acts on periodic functions on the plane. The lifted
diffusion has displacement covariance $2tm_t(x)I$, so $\alpha_t$ is
exactly the scalar term in the efficient displacement residual
decomposition. With $\mathcal I(\rho)=\int_{\mathbb T^2}\rho^{-2}|\nabla\rho|^2dx$
as in Section~\ref{sec:framework},
\begin{equation}\label{naclock:eq:Adef}
 \Acal_t(\rho)=\int_{\mathbb T^2}(\alpha_t-\rho)\,dx+\frac t2\mathcal I(\rho).
\end{equation}
For complex $\rho$ near $1$, $\mathcal I$ denotes the holomorphic continuation
$\int_{\mathbb T^2}\rho^{-2}(\nabla\rho\cdot\nabla\rho)dx$ (\emph{no complex
conjugation}); the operators and the scalar inverse are likewise continued
holomorphically wherever the latter is nonzero.

\begin{lemma}[Exact two-factor cancellation]\label{naclock:lem:exact}
Write $a=\rho^{-1}$, $L=L_\rho$, $P_u=P_u^\rho$, and $d=m_t-a$. Then
\begin{equation}\label{naclock:eq:d}
 d=\frac1t\int_0^t(t-u)P_uLa\,du,
\end{equation}
and
\begin{equation}\label{naclock:eq:exact}
 \Acal_t(\rho)
 =-\frac1t\int_0^t(t-u)\int_0^u
       \int_{\mathbb T^2}\rho(P_vL\rho)La\,dx\,dv\,du
    +\int_{\mathbb T^2}\frac{\rho^2d^2}{m_t}\,dx.
\end{equation}
These formulas also hold on a complex parameter disk on which the
semigroup is holomorphic and $m_t$ stays away from zero.
\end{lemma}
\begin{proof}
The semigroup identity $P_ua-a=\int_0^uP_vLa\,dv$ gives \eqref{naclock:eq:d}.
The scalar reciprocal identity
$\alpha_t-\rho=-\rho^2d+\rho^2d^2/m_t$ is exact, and periodic integration by
parts gives
$\int_{\mathbb T^2}\rho^2La\,dx=\int_{\mathbb T^2}\rho\Delta(\rho^{-1})\,dx
=\int_{\mathbb T^2}\frac{\nabla\rho\cdot\nabla\rho}{\rho^2}\,dx=\mathcal I(\rho)$. Thus the linear-in-time contribution of $d$ cancels the
second term in \eqref{naclock:eq:Adef}. For the remainder use weighted
bilinear symmetry,
\begin{equation}\label{naclock:eq:symmetry}
 \int_{\mathbb T^2}\rho\,u_1P_vu_2\,dx=\int_{\mathbb T^2}\rho\,(P_vu_1)u_2\,dx,
\end{equation}
which for real $\rho$ is self-adjointness in $L^2(\rho dx)$. Consequently
$\int_{\mathbb T^2}\rho^2(P_u-I)La\,dx
=\int_{\mathbb T^2}\rho\,[(P_u-I)\rho]La\,dx
=\int_0^u\int_{\mathbb T^2}\rho(P_vL\rho)La\,dx\,dv$,
and substitution into \eqref{naclock:eq:d} proves \eqref{naclock:eq:exact}.
For smooth fixed profiles all time integrals are integrable at zero. On a
complex disk, \eqref{naclock:eq:symmetry} first extends from real parameters
by the identity theorem for arbitrary fixed smooth test functions, and then
applies to the smooth parameter-dependent functions used above; the same
argument extends the semigroup and generator identities. No positivity is
used at complex parameters.
\end{proof}

The point of \eqref{naclock:eq:exact} is that $L\rho$ and $La$ enter as
\emph{two separate second-derivative factors}; an uncancelled Taylor
estimate of $P_uLa-La$ would instead require $L^2a$.

\subsection*{Uniform nonzero-amplitude derivatives}

\begin{lemma}[The precise complex input]\label{naclock:lem:complexinput}
Fix $\theta<\infty$. For a real trigonometric polynomial $F$ with
$\|F\|_{\mathcal A}\le1$, supported in $|k|\le\Lambda$, and
$t\Lambda^2\le\theta$, there are $r_0>0$ and $M_\theta<\infty$,
independent of $F,\Lambda,t$, such that the periodic semigroups for
$L_z=(1+zF)^{-1}\Delta$ are holomorphic on $|z|<r_0$ and satisfy
\begin{equation}\label{naclock:eq:complexP}
 \sup_{0<u\le t,\,|z|\le r_0}
 \max\{\|P_u^z\|_{2\to2},\|P_u^z\|_{\infty\to\infty}\}
 \le M_\theta,
 \qquad P_u^z1=1.
\end{equation}
The radius can be decreased if needed so the estimates hold on its closure.
\end{lemma}
\begin{proof}[Proof]
Theorem~\ref{thm:neargaussian}(b), with $\gamma=1/2$ and $r_0=\bar r_\theta$,
gives holomorphy and, uniformly for $0<u\le t$ and $|z|\le r_0$,
$|r_z(u,x,y)|\le C_\theta u^{-1}\exp(-c|x-y|^2/u)$. Periodize this kernel.
The integral of its absolute value in either variable over $\mathbb T^2$ is
bounded by the integral of this Gaussian over $\mathbb R^2$, independently
of $u$ and the cutoff. The two Schur bounds give the $L^2$ estimate and the
row bound the $L^\infty$ estimate. The real identity $P_u^z1=1$ extends to
the complex disk by holomorphy.
\end{proof}

\begin{theorem}[Compensated clock derivatives at the relative energy scale]
\label{naclock:thm:clockderivative}
For every $\theta<\infty$ there are $\varepsilon_\theta>0$ and
$C_{\theta,j}<\infty$, independent of the Fourier cutoff, such that the
following holds. Let $f$ be a real trigonometric polynomial supported in
$|k|\le\Lambda$, let $\varepsilon=\|f\|_{\mathcal A}\le
\varepsilon_\theta$, and suppose $t\Lambda^2\le\theta$.
For each fixed integer $j\ge2$,
\begin{equation}\label{naclock:eq:jbound}
 \sup_{|s|\le1}\left|
 \frac{d^j}{ds^j}\Acal_t(1+sf)\right|
 \le C_{\theta,j}t^2\varepsilon^{j-2}N(f).
\end{equation}
In particular $|\partial_s^3\Acal_t(1+sf)|\le C_{\theta,3}t^2\varepsilon N(f)$
and $|\partial_s^4\Acal_t(1+sf)|\le C_{\theta,4}t^2\varepsilon^2N(f)$.
The derivatives are taken at genuinely nonconstant backgrounds $1+sf$.
Mean zero is not needed for this theorem; it is needed for the heat-energy
identification below.
\end{theorem}
\begin{proof}
The case $f=0$ is exact. Otherwise write $F=f/\varepsilon$ and use
Lemma~\ref{naclock:lem:complexinput}. Choose $0<r_\theta<\min(r_0,1/4)$ so
small that $M_\theta r_\theta/(1-r_\theta)\le1/2$.
Set $\rho_z=1+zF$ and $a_z=\rho_z^{-1}$. Since $P_u^z1=1$,
$m_t(z)-1=t^{-1}\int_0^tP_u^z(a_z-1)\,du$, so
$\|m_t(z)-1\|_\infty\le M_\theta|z|/(1-r_\theta)\le\frac12$ and
$m_t^{-1}$ is uniformly bounded and holomorphic on this disk; only the
inverse of a time-averaged scalar is needed, not of a kernel.

Write $b_\mp=1\mp r_\theta$, $K_1=b_-^{-1}$ and
$K_2=b_-^{-3}+2\sqrt{10}\,r_\theta b_-^{-4}$. Lemma~\ref{lem:interp}(a) for
the real profile $F$, $\|F\|_\infty\le1$, and the explicit formulas for
$L_z\rho_z$ and $L_za_z$ give
$\|L_z\rho_z\|_2\le K_1|z|\sH(F)$ and $\|L_za_z\|_2\le K_2|z|\sH(F)$.
Equation~\eqref{naclock:eq:d} gives
$\|d\|_2\le M_\theta t\|L_za_z\|_2/2$. Taking absolute values
\emph{after} the cancellation in \eqref{naclock:eq:exact} yields, uniformly
on the complex disk,
\begin{equation}\label{naclock:eq:complexA}
 |\Acal_t(1+zF)|
 \le \frac{b_+M_\theta t^2}{6}
       \|L_z\rho_z\|_2\|L_za_z\|_2
       +\frac{b_+^2M_\theta^2t^2}{2}\|L_za_z\|_2^2
 \le C_\theta |z|^2t^2N(F),
\end{equation}
where, for example,
$C_\theta=b_+M_\theta K_1K_2/6+b_+^2M_\theta^2K_2^2/2$.

Take $\varepsilon_\theta\le r_\theta/4$. Lemma~\ref{lem:cauchyjets} with
$\Phi(z)=\Acal_t(1+zF)$, $r=r_\theta$ and, by \eqref{naclock:eq:complexA},
$M=C_\theta r_\theta^2t^2N(F)$ gives, since $\varepsilon^2N(F)=N(f)$,
\[
 |\partial_s^j\Acal_t(1+sf)|
 \le j!\left(\frac2{r_\theta}\right)^j
       C_\theta r_\theta^2 t^2\varepsilon^jN(F)
 =C_{\theta,j}t^2\varepsilon^{j-2}N(f).
\]
\end{proof}

\section{The anisotropy term}
\label{naani:sec}

Norms in this section use Euclidean measure on $\mathbb T^2$; source
derivatives hold the target fixed; matrix norms are Frobenius.
Let $r_s(t,x,y)$ be the periodic-plane, fixed-Lebesgue transition density
for $L_s=\rho_s^{-1}\Delta$, $\rho_s=1+sf$. Put
\[
 J_s(t,x)=\int_{\mathbb R^2}
 \frac{\nabla_xr_s\otimes\nabla_xr_s}{r_s}\,dy,
 \qquad h_s=2tJ_s,
 \qquad
 D_t^\infty(sf)=\int_{\mathbb T^2}
 \left(\frac12\operatorname{tr}h_s-\sqrt{\det h_s}\right)dx.
\]
In dimension two the source conformal-density factors cancel in this
integrated anisotropy defect.

\begin{theorem}[Anisotropy fourth variation at the relative energy scale]
\label{naani:main}
Fix $\theta<\infty$. There are $\varepsilon_\theta>0$ and
$C_\theta<\infty$, independent of the Fourier cutoff and number of modes,
with the following property. Let $f$ be a real trigonometric polynomial
supported in $|k|\le\Lambda$, put
$\varepsilon=\|f\|_{\mathcal A}\le\varepsilon_\theta$, and suppose
$t\Lambda^2\le\theta$. Then
\begin{equation}\label{naani:bound}
 \sup_{|s|\le1}
 \left|\frac{d^4}{ds^4}\left(\frac{D_t^\infty(sf)}{t^2}\right)\right|
 \le C_\theta\|\nabla f\|_4^4
 \le10C_\theta\varepsilon^2\|\Delta f\|_2^2.
\end{equation}
The same bound holds with the third derivative in place of the fourth.
Both estimates hold at nonconstant backgrounds.
\end{theorem}

\paragraph{An exact residual that retains one spatial derivative.}
Write $w=y-x$, $c_s(x)=\rho_s(x)/(2t)$, and define
\begin{equation}\label{naani:residual}
 \mathcal N_s=\nabla_xr_s-c_s(x)w r_s,
 \qquad
 \Sigma_s(x)=\int_{\mathbb R^2}\frac{\mathcal N_s\otimes \mathcal N_s}{r_s}\,dy.
\end{equation}
The lifted diffusion has displacement covariance $\sigma_s(x)I$, and its
source score $S_s=\nabla_x\log r_s$ satisfies $\mathbb E_x[S_sw^T]=I$, by
differentiating the martingale identity $\mathbb E_x w=0$ (justified by the
Gaussian source bounds). Expanding the square in \eqref{naani:residual},
\begin{equation}\label{naani:tracefree}
 \Sigma_s=J_s+(c_s^2\sigma_s-2c_s)I,
 \qquad \Sigma_s^{\mathrm{tf}}=J_s^{\mathrm{tf}}=Q_s^{\mathrm{tf}},
\end{equation}
where $Q_s$ is the efficient displacement residual covariance using the
actual clock; the new residual has the same trace-free part although its
isotropic term uses the frozen local coefficient.
For a positive $2\times2$ matrix $h$,
$\operatorname{tr}h/2-\sqrt{\det h}
=|h^{\mathrm{tf}}|^2/(\operatorname{tr}h+2\sqrt{\det h})$, so (with the
Frobenius convention)
\begin{equation}\label{naani:identity}
 \frac{D_t^\infty(sf)}{t^2}
 =4\int_{\mathbb T^2} |\Sigma_s^{\mathrm{tf}}|^2 \mathfrak b_s\,dx,
 \qquad
 \mathfrak b_s=(\operatorname{tr}h_s+2\sqrt{\det h_s})^{-1}.
\end{equation}

\paragraph{A marked spatial-gradient estimate in the Levi series.}
We use the mixed norms of Section~\ref{subsec:levitools}; in a convolution
chain any one factor may be kept in $L^p$ and the others in $L^\infty$, by
Minkowski's inequality in the intermediate displacements and periodic
translation; the source supremum is never taken before a
gradient is retained in $L^4$.
Choose $0<\gamma<1/11$ and set $\kappa_\mp=(1\mp\gamma)/4$,
$\kappa_1=(1-\gamma/2)/4$. Choose a fixed $r>0$ sufficiently small, in
particular $r<\gamma/2$, within the complex-parametrix domain of
Section~\ref{sec:kernel}. For complex $z$ with $|z|\varepsilon\le r$, let
$Z_z,K_z$ be the frozen kernel and defect of
Section~\ref{subsec:levitools} for $\rho_z=1+zf$, $a_z=\rho_z^{-1}$. Since
$\Re\rho_z\ge1-r$ and $\kappa_1<(1-r)/4$, Lemma~\ref{lem:frozen} with
$\|\nabla a_z\|_p\le C|z|\|\nabla f\|_p$ gives, with $F_4=\sG(f)$ and
$F_\infty=\sL(f)$,
\begin{align}
 |Z_z|&\le CG_{\kappa_1},&
 |\nabla_x Z_z|&\le Cu^{-1/2}G_{\kappa_1},\notag\\
 \|K_z(u;w)\|_{[4]}&\le C|z|F_4u^{-1/2}G_{\kappa_1}(u,w),&
 \|K_z(u;w)\|_{[\infty]}&\le C|z|F_\infty u^{-1/2}G_{\kappa_1}(u,w),
 \label{naani:basic}
\end{align}
where the source derivative holds $y$ fixed and so does not differentiate
$\rho_z(y)$.
Write $L_4=|z|F_4$ and $L_\infty=|z|F_\infty$. Since
$\sqrt tF_\infty\le\sqrt t\Lambda\varepsilon\le\sqrt\theta\varepsilon$, we
have $L_\infty\sqrt t\le r\sqrt\theta$, and \eqref{naani:basic} supplies the
hypotheses of Lemma~\ref{lem:markedbeta} with $G=G_{\kappa_1}$ and $m=1$.
The normally convergent complex Levi series
$r_z=Z_z+\sum_{n\ge1}Z_z*K_z^{*n}$ is the transition kernel for real
positive coefficients and is analytic in $z$ without a complex lower bound.
For the correction $\Xi_z=r_z-Z_z$, Lemma~\ref{lem:markedbeta} gives
\begin{equation}\label{naani:correction}
 \|\Xi_z(t;w)\|_{[4]}\le C_\theta|z|F_4\sqrt tG_{\kappa_1}(t,w),
 \qquad
 \|\nabla_x\Xi_z(t;w)\|_{[4]}
 \le C_\theta|z|F_4G_{\kappa_1}(t,w).
\end{equation}
If $\nabla f=0$, all correction terms vanish.

\paragraph{Analytic residual numerators and real quotients.}
The leading Gaussian has residual
$\nabla_xZ_z-\rho_z(x)wZ_z/(2t)=[\rho_z(y)-\rho_z(x)]wZ_z/(2t)$, whose mixed
$L^4$ bound is $C|z|F_4(|w|^2/t)G_{\kappa_1}(t,w)$ by the segment formula
for $f$ (as in the proof of Lemma~\ref{lem:frozen}) and periodic
translation. By \eqref{naani:correction} the correction residual is at most
$C_\theta|z|F_4(1+|w|/\sqrt t)G_{\kappa_1}(t,w)$. The margin
$\kappa_1-\kappa_-$ absorbs both displacement polynomials, so
\begin{equation}\label{naani:Ncomplex}
 \|\mathcal N_z(t;w)\|_{[4]}
 \le C_\theta |z|F_4t^{-1}e^{-\kappa_-|w|^2/t}.
\end{equation}
For $\varepsilon\le r/4$ apply Lemma~\ref{lem:cauchyjets} in the variable
$\zeta=z\varepsilon$ to $\zeta\mapsto\mathcal N_{\zeta/\varepsilon}(t;\cdot)$
with values in $L^4$ of the source; by \eqref{naani:Ncomplex},
$M=C_\theta(3r/4)\varepsilon^{-1}F_4t^{-1}e^{-\kappa_-|w|^2/t}$, whence
\begin{equation}\label{naani:Njets}
 \|\partial_s^j\mathcal N_s(t;w)\|_{[4]}
 \le C_{\theta,j}\varepsilon^{j-1}F_4t^{-1}e^{-\kappa_-|w|^2/t},
 \qquad 1\le j\le4.
\end{equation}
For $j=0$ retain the sharper direct bound with factor $|s|F_4$.
The case $\varepsilon=0$ is exact.

For real $s$, Theorem~\ref{thm:neargaussian} gives
$r_s\ge ct^{-1}e^{-\kappa_+|w|^2/t}$ and
$|\partial_s^kr_s|\le C_k\varepsilon^kt^{-1}e^{-\kappa_-|w|^2/t}$, $0\le k\le4$.
A term in the $j$th derivative of $(\mathcal N_s\otimes \mathcal N_s)/r_s$ is a
constant multiple of
\begin{equation}\label{naani:quotient}
 r_s^{-m-1}(\partial_s^{\nu_1}\mathcal N_s)
 \otimes(\partial_s^{\nu_2}\mathcal N_s)
 \prod_{\ell=1}^m\partial_s^{k_\ell}r_s,
 \qquad
 \nu_1+\nu_2+\sum_{\ell=1}^mk_\ell=j,
 \quad k_\ell\ge1,\quad 0\le m\le j.
\end{equation}
Place the two residual factors in $L^4_x$ and the others in $L^\infty_x$.
The resulting $L^2_x$ envelope is a scalar amplitude factor times
$t^{-1}\exp(-[(m+2)\kappa_--(m+1)\kappa_+]|w|^2/t)$; for $j\le4$ the worst
exponent is $6\kappa_--5\kappa_+=(1-11\gamma)/4>0$, and integration in the
displacement cancels $t^{-1}$. Minkowski's inequality therefore yields
\begin{align}
 \|\Sigma_s\|_2&\le C s^2F_4^2,&
 \|\partial_s\Sigma_s\|_2&\le C|s|F_4^2,&
 \|\partial_s^2\Sigma_s\|_2&\le CF_4^2,\notag\\
 \|\partial_s^3\Sigma_s\|_2&\le C\varepsilon F_4^2,&
 \|\partial_s^4\Sigma_s\|_2&\le C\varepsilon^2F_4^2.
 \label{naani:Tjets}
\end{align}
Indeed, a term in \eqref{naani:quotient} carries $\varepsilon$ to the power
$j-\boldsymbol{1}_{\nu_1>0}-\boldsymbol{1}_{\nu_2>0}$, and an additional
$|s|$ for each undifferentiated residual. For $j\ge2$ the exponent is at
least $j-2$. For $j=1$ either a residual is differentiated, leaving $|s|$,
or a density is, leaving $\varepsilon s^2$; both give the second bound in
\eqref{naani:Tjets}. The trace-free projection preserves these bounds up to
a fixed factor.

\begin{proof}[Completion of the proof of Theorem~\ref{naani:main}]
The plane displacement regression bound gives $h_s\succeq(1-\varepsilon)I$,
and the near-Gaussian density and source jets give
$\|\partial_s^kh_s\|_\infty\le C_k\varepsilon^k$ for $k\le4$. Thus $h_s$
stays in a common compact subset of the positive matrices, and the finite
real chain rule for the denominator in \eqref{naani:identity} gives
$\|\partial_s^k\mathfrak b_s\|_\infty\le C_k\varepsilon^k$, $0\le k\le4$.
The fourth derivative of \eqref{naani:identity} is a finite sum of
integrals with derivative orders $i+j+k=4$ on the two $\Sigma_s^{\mathrm{tf}}$
factors and $\mathfrak b_s$. H\"older's inequality and
\eqref{naani:Tjets} bound every term by $C_\theta F_4^4$, since all
remaining powers of $\varepsilon$ are nonnegative and $|s|\le1$; the term
with orders $(2,2,0)$ has no additional small factor. This proves the first
bound in \eqref{naani:bound}. For three derivatives the same product rule
contains $\Sigma^{(3)}\Sigma$, $\Sigma^{(2)}\Sigma^{(1)}$ and lower-order
products with derivatives of $\mathfrak b_s$; bounding the zero- and
first-order $s$ factors and all nonnegative $\varepsilon$ powers by one
gives the same $C_\theta F_4^4$ bound.
Finally Lemma~\ref{lem:interp}(a) gives
$\|\nabla f\|_4^4\le10\|f\|_\infty^2\|\Delta f\|_2^2
\le10\varepsilon^2\|\Delta f\|_2^2$, with no lattice count; this is the
second bound in \eqref{naani:bound}.
\end{proof}
\section{Quadratic-score cancellation}
\label{namom:section}
The compensated scalar residual admits a further exact reduction, isolating
the leading spatial-gradient contribution; the following sections estimate
the remainder after spatial integration.

Let $\rho$ be smooth, positive, and periodic, fix $t>0$ and a source $x$,
and write $w=X_t-x$, $S=\nabla_x\log r_t(x,X_t)$, $g=\nabla\rho(x)$,
$c=\rho(x)/(2t)$; source derivatives hold the target fixed. Define the
frozen-clock residual numerator and covariance by
\[
 \mathcal N(t,x,y)=\nabla_xr_t(x,y)-c(y-x)r_t(x,y),\qquad
 \Sigma=\mathbb E_x[(S-cw)^{\otimes2}],
\]
($\Sigma$ uses the frozen clock, unlike the efficient residual covariance
$Q$), and
\begin{equation}\label{namom:Psi}
 \Psi(x,w)=\frac{(2ww^T-|w|^2I)g}{8t},\qquad
 M_3(x)=\mathbb E_x[|w|^2w],\qquad
 M_4(x)=\mathbb E_x|w|^4.
\end{equation}

\begin{proposition}[Exact quadratic-score cancellation and square completion]
\label{namom:identity}
In dimension two one has
\begin{equation}\label{namom:orthogonality}
 \mathbb E_x[S\cdot\Psi]=0.
\end{equation}
If $U=\mathcal N-r_t\Psi$, then, pointwise in the source,
\begin{align}\label{namom:square}
 \operatorname{tr}\Sigma
 ={}&\int_{\mathbb R^2}\frac{|U(t,x,y)|^2}{r_t(x,y)}\,dy
       -\frac{\rho\,g\cdot M_3}{8t^2}
    -\frac{|g|^2M_4}{64t^2}.
\end{align}
Writing $a=\rho^{-1}$, $m=\sigma_t/(2t)$, and $d=m-a$, the relation to
the efficient residual is exactly
\begin{equation}\label{namom:efficient}
 t(\operatorname{tr}\Sigma-\operatorname{tr}Q)
     =\frac{\rho^2d^2}{m}.
\end{equation}
\end{proposition}
\begin{proof}
Since $\mathbb E_xw=0$ and $\mathbb E_x[w_jw_k]=\sigma_t(x)\delta_{jk}$,
differentiating the covariance at fixed target gives
$\mathbb E_x[S_iw_jw_k]=\partial_i\sigma_t\,\delta_{jk}$ (the two
derivatives of $w$ have zero expectation since the mean vanishes). In
dimension $n$ the contraction with $(2ww^T-|w|^2I)g$ gives
$(2-n)\nabla\sigma_t\cdot g$, which vanishes for $n=2$, proving
\eqref{namom:orthogonality}. Gaussian source bounds and polynomial
displacement moments justify this differentiation and all integrals below.

The identities $(2ww^T-|w|^2I)^2=|w|^4I$ and
$w\cdot\Psi=|w|^2w\cdot g/(8t)$ give
\[
 \mathbb E_x|\Psi|^2=\frac{|g|^2M_4}{64t^2},\qquad
 \mathbb E_x[\Psi\cdot(S-cw)]
      =-\frac{\rho\,g\cdot M_3}{16t^2}.
\]
Expanding $\mathbb E_x|S-cw-\Psi|^2$ proves \eqref{namom:square}. Finally,
the exact displacement-regression identity with the frozen-clock predictor
$\lambda=2t/\rho$ gives
$\operatorname{tr}\Sigma-\operatorname{tr}Q=2(\sigma_t-\lambda)^2/(\sigma_t\lambda^2)$;
substituting $\sigma_t=2tm$, $\lambda=2ta$ yields \eqref{namom:efficient}.
\end{proof}

\section{A second-order score remainder}\label{sec:secondscore}
The estimate is proved by a target-frozen Levi expansion. Complex
analyticity is used only for numerator estimates; quotient estimates use
the real near-Gaussian lower bound.

Let $f$ be a real trigonometric polynomial with Fourier support
$|k|\le\Lambda$, and abbreviate $\eps=\sA(f)$, $H=\sH(f)=\norm{D^2 f}_2$,
$F_4=\sG(f)$, $F_\infty=\sL(f)$. We work on the periodic plane with
$t\Lambda^2\le\theta$, $\theta<\infty$ fixed; spatial norms are on the
source period cell. For complex $z$, let $\rho_z=1+zf$, $a_z=\rho_z^{-1}$,
and let $r_z$ be the complex Levi kernel (for real positive coefficients,
the transition density of $a_z\Delta$ relative to Lebesgue measure).
Source derivatives hold the target fixed. With $w=y-x$ define
\[
 \Psi_z(x,w)=\frac{(2ww^T-|w|^2I)\nabla\rho_z(x)}{8t},\quad
 \mathcal N_z=\nabla_x r_z-\frac{\rho_z(x)w}{2t}r_z,\quad
 U_z=\mathcal N_z-r_z\Psi_z(x,w),
\]
all evaluated at $(t,x,y)$.
Mixed norms $\norm{\cdot}_{[p]}$ are those of Section~\ref{sec:framework}.
Choose $0<\gamma<1/11$, put $\kappa_\mp=(1\mp\gamma)/4$ and
$G_{\kappa_-}(t,w)=(\kappa_-/\pi)t^{-1}e^{-\kappa_-|w|^2/t}$.
Constants may depend on $\theta,\gamma$ and the fixed jet order, but not
on the Fourier cutoff, mode count or rescaled cell size.

\begin{theorem}[Second-order residual numerator]
\label{score:thm:score}
There is a fixed $r>0$, chosen sufficiently small in terms of $\gamma$,
such that for $|z|\eps\le r$,
\begin{equation}\label{score:eq:complex}
 \norm{U_z(t;w)}_{[2]}
 \le C_\theta\sqrt t\bigl(|z|H+|z|^2F_4^2\bigr)
             G_{\kappa_-}(t,w).
\end{equation}
For $\eps$ sufficiently small and real $|s|\le1$ this implies
\begin{align}
 \norm{U_s(t;w)}_{[2]}
   &\le C_\theta |s|\sqrt t H G_{\kappa_-}(t,w),\label{score:eq:zerojet}\\
 \norm{\partial_s^jU_s(t;w)}_{[2]}
   &\le C_{\theta,j}\sqrt t\eps^{j-1}H G_{\kappa_-}(t,w),
                      \qquad 1\le j\le4.\label{score:eq:jets}
\end{align}
These estimates hold on the entire resolved Fourier ball, without a
positive lower spectral cutoff.
\end{theorem}

\subsection{Gaussian comparison and marked spatial norms}
Fix $\kappa_0$ strictly between $\kappa_-$ and $(1-r)/4$, reserving
further fixed exponent gaps between these numbers; extra displacement
polynomials below use one of these gaps, and the final exponent is
$\kappa_-$. Let $Z_z,K_z$ be the frozen kernel and defect of
Section~\ref{subsec:levitools} for $\rho_z$, so that
\begin{equation}\label{score:eq:levi}
 r_z=Z_z+Z_z*K_z+\sum_{n\ge2}Z_z*K_z^{*n}.
\end{equation}
Since $|z|\eps\le r<1$ gives $\Re\rho_z\ge1-r$, Lemma~\ref{lem:frozen} applies with $G=G_{\kappa_0}$:
\begin{align}
 |\nabla_v^\ell\mathcal G_\varrho(u,v)|&\le C u^{-\ell/2}G_{\kappa_0}(u,v),
                      &&\ell=0,1,2,\label{score:eq:gcomp}\\
 |\nabla_v^\ell(\mathcal G_{\varrho_1}-\mathcal G_{\varrho_2})(u,v)|
   &\le C|\varrho_1-\varrho_2|u^{-\ell/2}G_{\kappa_0}(u,v),
                      &&\ell=0,1.\label{score:eq:gdiff}
\end{align}

For fixed endpoint displacement, write every intermediate point as $x$
plus a relative displacement; marked coefficient evaluations then have the
form $h(x+v)$ with $v$ independent of $x$ (the shifts may be correlated or
identical). Translation and H\"older give
\begin{equation}\label{score:eq:marked}
 \norm{h(x+v)k(x+v')}_2\le\norm h_4\norm k_4,
 \qquad \norm{h(x+v)}_2=\norm h_2.
\end{equation}
Hence two gradient markers in $L^4$, or one Hessian marker in $L^2$, may
be retained through all positive majorant path integrals (by Tonelli and
Minkowski; no independence is assumed). The needed coefficient estimates are
\begin{align}
 \norm{\rho_z(x+v)-\rho_z(x+v')}_{L^4_x}
 &\le |z||v-v'|F_4,\label{score:eq:rho-diff}\\
 \norm{a_z(x+v)-a_z(x+v')}_{L^4_x}
 &\le C|z||v-v'|F_4,\label{score:eq:a-diff}\\
 \norm{D^2a_z}_2
 &\le C\bigl(|z|H+|z|^2F_4^2\bigr)=:C\mathcal H_z,
 \label{score:eq:hessa}
\end{align}
using $D^2a_z=-z\rho_z^{-2}D^2f+2z^2\rho_z^{-3}\nabla f\otimes\nabla f$.
Second-order Taylor remainders, with coefficients evaluated on their
segments and $y=x+w$ after the source derivative, satisfy
\begin{align}
 \norm{a_z(y+v)-a_z(y)-\nabla a_z(y)\cdot v}_{L^2_x}
 &\le C|v|^2\mathcal H_z,\label{score:eq:taylora}\\
 \norm{\rho_z(y)-\rho_z(x)-\nabla\rho_z(y)\cdot w}_{L^2_x}
 &\le C|z||w|^2H,\label{score:eq:taylorrho}\\
 \norm{\nabla\rho_z(y)-\nabla\rho_z(x)}_{L^2_x}
 &\le |z||w|H.\label{score:eq:gradient-diff}
\end{align}

\subsection{An exact first-gradient convolution}
Fix the target $y$ and abbreviate $\varrho=\rho_z(y)$, $a=\varrho^{-1}$,
$g=\nabla\rho_z(y)$, $b=\nabla a_z(y)=-\varrho^{-2}g$. Let
$\mathcal G=\mathcal G_\varrho$, the exact frozen kernel
$\mathcal G(u,v)=(4\pi au)^{-1}e^{-|v|^2/(4au)}$, and define
\[
 K_{\rm lin}(u,\xi,y)=[b\cdot(\xi-y)]\Delta_\xi\mathcal G(u,\xi-y),\qquad
 \mathcal B_z(t,x,y)=\int_0^t\!\!\int_{\R^2}
       \mathcal G(t-u,\xi-x)K_{\rm lin}(u,\xi,y)\,d\xi\,du.
\]

\begin{lemma}[First-gradient identity at a fixed target]
\label{score:lem:bridge}
In dimension two,
\begin{equation}\label{score:eq:D}
 \mathcal B_z(t,x,y)=\frac{(g\cdot w)|w|^2}{8t}Z_z(t,x,y),
\end{equation}
and its fixed-target source residual is
\begin{equation}\label{score:eq:Dgradient}
 \left(\nabla_x-\frac{\varrho w}{2t}\right)\mathcal B_z
 =-Z_z\left[\frac{g|w|^2}{8t}
                  +\frac{(g\cdot w)w}{4t}\right].
\end{equation}
\end{lemma}
\begin{proof}
First take $a>0$ real. Dividing the product of the two frozen Gaussians
by $\mathcal G(t,w)$ gives a Gaussian bridge: $v=\xi-y$ has mean
$\mu=-uw/t$ and coordinate variance $\sigma^2=2au(t-u)/t$. In two
dimensions,
\[
 \mathbb E[(b\cdot v)|v|^2]
   =(b\cdot\mu)(|\mu|^2+4\sigma^2),\qquad
 \Delta \mathcal G(u,v)/\mathcal G(u,v)
   =|v|^2/(4a^2u^2)-1/(au),
\]
so the bridge integral divided by $\mathcal G(t,w)$ is
$-\frac{u}{t}(b\cdot w)\bigl[\frac{|w|^2}{4a^2t^2}+\frac{t-2u}{aut}\bigr]$.
Integrating over $0<u<t$, the second term vanishes and the first gives
$-(b\cdot w)|w|^2/(8a^2t)$; since $b=-a^2g$, this is \eqref{score:eq:D}.
Both sides are holomorphic for $a$ near $1$, and Gaussian domination
extends the identity to the coefficient disk by analytic continuation.
Finally $y,\varrho,g$ are fixed during the source derivative, so
$\nabla_x=-\nabla_w$ on the polynomial and $\nabla_xZ_z=\varrho wZ_z/(2t)$,
which proves \eqref{score:eq:Dgradient} without differentiating a
source-frozen coefficient.
\end{proof}

\subsection{The second-spatial-order kernel error}
Set $\mathcal T_z=r_z-Z_z-\mathcal B_z$. We claim (with $\kappa_0$ allowed
to shrink within its reserved interval in later uses)
\begin{align}
 \norm{\mathcal T_z(t;w)}_{[2]}&\le C_\theta t\mathcal H_zG_{\kappa_0}(t,w),
                                                  \label{score:eq:E}\\
 \norm{\nabla_x\mathcal T_z(t;w)}_{[2]}
   &\le C_\theta\sqrt t\mathcal H_zG_{\kappa_0}(t,w).
                                                  \label{score:eq:gradE}
\end{align}

\paragraph{The one-defect error.}
With $\mathcal G=\mathcal G_{\rho_z(y)}$, split the difference between the
one-defect integral and $\mathcal B_z$ exactly into
\begin{align*}
 &\int_0^t\!\!\int
 \bigl[Z_z(t-u,x,\xi)-\mathcal G(t-u,\xi-x)\bigr]
                         K_z(u,\xi,y)\,d\xi\,du,\tag{I}\\
 &\int_0^t\!\!\int \mathcal G(t-u,\xi-x)
       \bigl[K_z(u,\xi,y)-K_{\rm lin}(u,\xi,y)\bigr]\,d\xi\,du.
                                                        \tag{II}
\end{align*}
In (I), the Gaussian difference contributes $|\rho_z(\xi)-\rho_z(y)|$ and
the defect $|a_z(\xi)-a_z(y)|\,u^{-1}$; by two segment formulas and
\eqref{score:eq:marked} their product has source $L^2$ norm at most
$C|z|^2|\xi-y|^2F_4^2u^{-1}$. The factor $|\xi-y|^2/u$ is absorbed in the
right-hand Gaussian, the spatial convolution is a Gaussian with no time
singularity, and integration in $u$ contributes $t$. A source derivative
falls only on the leftmost Gaussian difference (not on coefficients at
$\xi$ or $y$); by \eqref{score:eq:gdiff} it costs $(t-u)^{-1/2}$, whose
integral contributes $2\sqrt t$.

In (II), since the right frozen Gaussian is already $\mathcal G$, the
remaining coefficient difference is exactly the Taylor remainder in
\eqref{score:eq:taylora}, with $L^2$ bound $C|\xi-y|^2\mathcal H_z$. The
same absorption cancels the $u^{-1}$ from the Laplacian, and the same time
integrals give $t\mathcal H_z$ and $\sqrt t\mathcal H_z$. No third
derivative is used.

\paragraph{At least two defects.}
By \eqref{eq:frozenK}, $\|K_z(u;w)\|_{[4]}\le C|z|F_4u^{-1/2}G_{\kappa_0}$
and $\|K_z(u;w)\|_{[\infty]}\le C|z|F_\infty u^{-1/2}G_{\kappa_0}$, with
$|z|F_\infty\sqrt t\le r\sqrt\theta$. Lemma~\ref{lem:markedbeta} with
$m=2$ and $G=G_{\kappa_0}$ bounds the sums over $n\ge2$ of
$Z_z*K_z^{*n}$ and $(\nabla_xZ_z)*K_z^{*n}$ by
$C_\theta t|z|^2F_4^2G_{\kappa_0}$ and $C_\theta\sqrt t|z|^2F_4^2G_{\kappa_0}$.
With (I)--(II) this proves \eqref{score:eq:E}--\eqref{score:eq:gradE}
(if $\nabla f=0$ the defect series vanishes). The case $m=1$ of
Lemma~\ref{lem:markedbeta} gives
\begin{equation}\label{score:eq:C4}
 \norm{r_z(t;w)-Z_z(t;w)}_{[4]}
       \le C_\theta\sqrt t|z|F_4G_{\kappa_0}(t,w).
\end{equation}

\subsection{The exact first-gradient cancellation in the score}
For a vector $v$, write
$\Psi(v,w)=(2ww^T-|w|^2I)v/(8t)=(v\cdot w)w/(4t)-v|w|^2/(8t)$.
Insert $r_z=Z_z+\mathcal B_z+\mathcal T_z$ into
$U_z=(\nabla_x-\rho_z(x)w/(2t))r_z-r_z\Psi(\nabla\rho_z(x),w)$. Since
$\nabla_xZ_z=\rho_z(y)wZ_z/(2t)$, the $Z_z$ term contributes
$(\rho_z(y)-\rho_z(x))wZ_z/(2t)$. By \eqref{score:eq:Dgradient}, the
$\mathcal B_z$ term contributes
$-Z_z[g|w|^2/(8t)+(g\cdot w)w/(4t)]
=Z_z[\Psi(g,w)-(g\cdot w)w/(2t)]$ plus
$(\rho_z(y)-\rho_z(x))w\mathcal B_z/(2t)$. Collecting terms gives exactly
\begin{align}
 U_z={}&Z_z\left[
 \frac{(\rho_z(y)-\rho_z(x)-g\cdot w)w}{2t}
       +\Psi(g,w)-\Psi(\nabla\rho_z(x),w)\right]\notag\\
 &+\frac{(\rho_z(y)-\rho_z(x))w}{2t}\mathcal B_z
   +\left(\nabla_x-\frac{\rho_z(x)w}{2t}\right)\mathcal T_z
   -(r_z-Z_z)\Psi_z(x,w).
 \label{score:eq:exactcancel}
\end{align}
Thus the leading first-spatial-order terms cancel exactly, including the
term parallel to $g$, not only the trace-free part. Each term has the
desired mixed $L^2$ estimate:
\begin{itemize}
\item The first line is at most $C|z|H(|w|^3/t)G_{\kappa_0}$ by
\eqref{score:eq:taylorrho} and \eqref{score:eq:gradient-diff}; absorption
gives $C\sqrt t|z|H G_{\kappa_-}$.
\item In the second line's first term, \eqref{score:eq:D} and two $L^4$
markers give $C|z|^2F_4^2(|w|^5/t^2)G_{\kappa_0}$, hence
$C\sqrt t|z|^2F_4^2G_{\kappa_-}$.
\item By \eqref{score:eq:E}--\eqref{score:eq:gradE} the $\mathcal T_z$
residual is at most $C_\theta\sqrt t\mathcal H_z(1+|w|/\sqrt t)G_{\kappa_0}$.
\item By \eqref{score:eq:C4}, H\"older and
$\norm{\Psi_z(x,w)}_4\le C|z|F_4|w|^2/t$, the last term is at most
$C_\theta\sqrt t|z|^2F_4^2(|w|^2/t)G_{\kappa_0}$.
\end{itemize}
This proves \eqref{score:eq:complex}; every error retains a Hessian in
$L^2$ or two gradients in $L^4$, and none uses a mode count or a lower
spectral bound.

\subsection{Amplitude jets and the squared-score contribution}
Lemma~\ref{lem:interp}(a) gives
\begin{equation}\label{score:eq:GN}
 F_4^2\le\sqrt{10}\norm f_\infty\norm{\Delta f}_2\le\sqrt{10}\,\eps H.
\end{equation}
For $\eps\le r/4$ and $|z|\eps\le3r/4$, the right side of
\eqref{score:eq:complex} is therefore at most
$C_\theta\sqrt t\eps^{-1}H G_{\kappa_-}$. The normally convergent Levi
series, its fixed-target source derivative, and its explicit polynomial
correction are holomorphic with values in the fixed-displacement source
$L^2$ space, so Lemma~\ref{lem:cauchyjets} in $\zeta=z\eps$, with this
bound as $M$, yields \eqref{score:eq:jets}. At real $s$,
\eqref{score:eq:complex} and \eqref{score:eq:GN} give
\eqref{score:eq:zerojet}; the case $\eps=0$ is trivial. This completes the
proof of Theorem~\ref{score:thm:score}.

\begin{corollary}[The completed score square]
\label{score:cor:square}
Define for real $|s|\le1$
\[
 \mathcal S_t(s)=t\int_{\mathbb T^2}\int_{\R^2}
             \frac{|U_s(t,x,y)|^2}{r_s(t,x,y)}\,dy\,dx.
\]
Using the established real lower bound
$r_s\ge ct^{-1}e^{-\kappa_+|w|^2/t}$ and real upper density jets
$|\partial_s^kr_s|\le C_k\eps^kG_{\kappa_-}$, $k\le4$, one has
\begin{equation}\label{score:eq:square4}
 \sup_{|s|\le1}|\mathcal S_t^{(j)}(s)|
       \le C_{\theta,j}t^2\eps^{j-2}\norm{\Delta f}_2^2,
       \qquad j=2,3,4.
\end{equation}
\end{corollary}
\begin{proof}
For $j=2,3,4$, $\mathcal S_t^{(j)}$ is a sum of terms
$r_s^{-m-1}(\partial_s^{\nu_1}U_s)\cdot(\partial_s^{\nu_2}U_s)
\prod_{\ell=1}^m\partial_s^{k_\ell}r_s$ with
$\nu_1+\nu_2+\sum k_\ell=j$, $k_\ell\ge1$.
Put both $U$ jets in $L^2_x$ and density jets in sup norm. The total
amplitude power is $\eps^{j-\boldsymbol1_{\nu_1>0}-\boldsymbol1_{\nu_2>0}}$,
with bounded $|s|$ factors for zero-order residuals, hence at most
$C\eps^{j-2}$ for $\eps\le1$; the two $\sqrt t$ factors give $t$. The
Gaussian quotient has prefactor $t^{-1}$ and exponent
$(m+2)\kappa_--(m+1)\kappa_+\ge6\kappa_--5\kappa_+>0$, so its displacement
integral is constant, leaving $t\eps^{j-2}H^2$ before the exterior factor
$t$, which proves \eqref{score:eq:square4}. All quotients here are real; complex analyticity entered only through
numerator estimates.
\end{proof}

\section{Uniform cancellation of displacement moments}
\label{mom:section}
This section proves the full resolved-ball estimate for the compensated
displacement-moment term of the quadratic-score reduction, at nonconstant
backgrounds and with two spatial energy factors. It is independent of the
residual-square estimate; both combine with the quadratic clock bound in
Section~\ref{sec:closure}.

Let $F$ be a real trigonometric polynomial with $\|F\|_{\mathcal A}=1$ and
frequencies $|k|\le\Lambda$, and fix $t\Lambda^2\le\theta$. For complex
$z$ in a small disk, put
\[
 \rho=1+zF,\quad a=\rho^{-1},\quad L=a\Delta,\quad P_u=e^{uL},
 \quad g=\nabla\rho,\quad \mathfrak g=g\cdot g,\quad
 \mathcal I=\int_{\mathbb T^2} a^2\mathfrak g\,dx.
\]
Products in complex formulas are bilinear (no conjugation); estimates use
the complex absolute value. Write $\sH=\sH(F)$, $\sG=\sG(F)$.
Lemma~\ref{lem:interp}(a) gives
\begin{equation}\label{mom:GN}
 \sG^2\le\sqrt{10}\|F\|_\infty \sH\le\sqrt{10}\,\sH.
\end{equation}
Let $M_3=\mathbb E_x[|X_t-x|^2(X_t-x)]$ and $M_4=\mathbb E_x|X_t-x|^4$ (for
complex $z$, the same polynomial integrals against the complex lifted
kernel), and define
\begin{equation}\label{mom:defs}
 J_3(t,z)=\int_{\mathbb T^2}\rho g\cdot M_3\,dx,\qquad
 H_4(t,z)=\int_{\mathbb T^2}\mathfrak gM_4\,dx
\end{equation}
($H_4$ is a fourth displacement moment, not a fourth parameter
derivative), and the compensated moment functional
\begin{equation}\label{mom:Mdef}
 \mathfrak M_t(\rho)
 =-\frac{J_3(t,z)}{8t}-\frac{H_4(t,z)}{64t}-\frac t2\mathcal I.
\end{equation}

\begin{theorem}[Integrated moment cancellation on the full resolved ball]
\label{mom:main}
For every finite $\theta$ there are $r_\theta>0$ and $C_\theta<\infty$,
independent of the Fourier cutoff and mode count, such that
\begin{align}
 |J_3(t,z)+8t^2\mathcal I|&\le C_\theta t^3|z|^2\sH^2,\label{mom:J3bound}\\
 |H_4(t,z)-32t^2\mathcal I|&\le C_\theta t^3|z|^3\sH^2,\label{mom:H4bound}\\
 |\mathfrak M_t(1+zF)|&\le C_\theta t^2|z|^2\sH^2
                                  \label{mom:Mbound}
\end{align}
for $|z|\le r_\theta$ and $t\Lambda^2\le\theta$. The functionals are
holomorphic on a neighborhood of this closed disk after decreasing the
radius if necessary.

Consequently, for $f$ real, $\varepsilon=\|f\|_{\mathcal A}\le
r_\theta/4$, and $t\Lambda^2\le\theta$, each fixed integer $j\ge2$
satisfies
\begin{equation}\label{mom:derivatives}
 \sup_{|s|\le1}|\partial_s^j\mathfrak M_t(1+sf)|
 \le C_{\theta,j}t^2\varepsilon^{j-2}\|\Delta f\|_2^2.
\end{equation}
In particular the third derivative has bound
$C_{\theta,3}t^2\varepsilon\|\Delta f\|_2^2$, and the fourth derivative
has bound $C_{\theta,4}t^2\varepsilon^2\|\Delta f\|_2^2$.
\end{theorem}

\paragraph{Uniform semigroup input and its spatial consequences.}
Theorem~\ref{thm:neargaussian}(b), with $\gamma=1/2$ and
$r_\theta\le\bar r_\theta$, supplies holomorphy and, uniformly for
$0<u\le t$ and $|z|\le r_\theta$,
\begin{equation}\label{mom:kernelinput}
 |r_z(u,x,y)|\le C_\theta u^{-1}e^{-c|x-y|^2/u},\qquad
 |\nabla_xr_z(u,x,y)|\le C_\theta u^{-3/2}e^{-c|x-y|^2/u}.
\end{equation}
Periodization and the two Schur integrals give
\begin{equation}\label{mom:Lp}
 \|P_u h\|_p\le C_\theta\|h\|_p,\qquad 1\le p\le\infty.
\end{equation}
The identities $P_u1=1$ and the bilinear weighted symmetry
\begin{equation}\label{mom:transpose}
 \int_{\mathbb T^2}\rho h_1P_uh_2\,dx=\int_{\mathbb T^2}\rho(P_uh_1)h_2\,dx
\end{equation}
extend from real parameters by holomorphy. Kernel mass conservation and
\eqref{mom:kernelinput} give the gradient propagation estimate
\begin{equation}\label{mom:gradprop}
 \|\nabla P_uh\|_p\le C_\theta\|\nabla h\|_p,
 \qquad 1\le p\le\infty,
\end{equation}
for any smooth periodic scalar or vector $h$ (not necessarily with the
Fourier support of $F$). Indeed
$\nabla P_uh(x)=\int_{\mathbb R^2}\nabla_xr_z(u,x,x+w)[h(x+w)-h(x)]\,dw$
(source differentiation before setting $y=x+w$), and the segment formula,
Minkowski's inequality, and periodic translation bound its $L^p_x$ norm by
$C_\theta\|\nabla h\|_p\int_{\mathbb R^2}u^{-3/2}|w|e^{-c|w|^2/u}\,dw
\le C_\theta\|\nabla h\|_p$.
Commutation of $L$ with its semigroup and periodic Parseval (also for
complex functions) give
\begin{equation}\label{mom:H2}
 \Delta P_uh=\rho P_u(a\Delta h),\qquad
 \|D^2P_uh\|_2=\|\Delta P_uh\|_2
                         \le C_\theta\|\Delta h\|_2.
\end{equation}
These estimates extend to $u=0$. Uniformly on the complex disk, the
explicit derivatives of $a$ and \eqref{mom:GN} imply
\begin{align}
 \|a\|_\infty+\|\rho\|_\infty&\le C,&
 \|\nabla a\|_4&\le C|z|\sG,\notag\\
 \|\Delta a\|_2+\|La\|_2&\le C|z|\sH,&
 \|g\|_4&=|z|\sG,\label{mom:basic}\\
 \|Dg\|_2&=|z|\sH,&
 \|\mathfrak g\|_2&\le |z|^2\sG^2.\notag
\end{align}
In particular $\|\nabla P_ua\|_4\le C_\theta|z|\sG$ and
$\|D^2P_ua\|_2\le C_\theta|z|\sH$.

\paragraph{Exact commutator formulas for the lifted moments.}
For periodic $h$, define the source-periodic operators
$\mathscr X_i(u)h(x)=\mathbb E_x[(X_u^i-x_i)h(X_u)]$ and
$\mathscr R(u)h(x)=\mathbb E_x[|X_u-x|^2h(X_u)]$, both zero at $u=0$.
Expanding the generator on multiplication by a coordinate gives
\[
 \partial_u \mathscr X_i(u)h=L\mathscr X_i(u)h+2a\partial_iP_uh,\qquad
 \partial_u \mathscr R(u)h=L\mathscr R(u)h+
                   4a\textstyle\sum_i\partial_i\mathscr X_i(u)h+4aP_uh,
\]
where the coefficient $4$ in the last term uses dimension two. Duhamel's
formula gives, with all time variables nonnegative,
\begin{align*}
 \mathscr X_i(u)h&=2\int_{v_1+v_2=u}P_{v_1}[a\partial_iP_{v_2}h]\,dv_2,\\
 \mathscr R(u)h&=4\int_{v_1+v_2=u}P_{v_1}[aP_{v_2}h]\,dv_2\\
 &\quad+8\sum_i\int_{v_1+v_2+v_3=u}
       P_{v_1}[a\partial_iP_{v_2}(a\partial_iP_{v_3}h)]\,dv_2\,dv_3.
\end{align*}
Since $\Delta(|w|^2w_i)=8w_i$ and $\Delta|w|^4=16|w|^2$, Dynkin's formula
and one further time integration give the exact identities
\begin{align}
 M_3(t)&=16\int_{u+v\le t}P_u[a\nabla P_va]\,du\,dv,
                                                    \label{mom:M3exact}\\
 M_4(t)&=64\int_{u+v\le t}P_u[aP_va]\,du\,dv\notag\\
 &\quad+128\sum_i\int_{u+v+v'\le t}
       P_u[a\partial_iP_v(a\partial_iP_{v'}a)]\,du\,dv\,dv'.
                                                    \label{mom:M4exact}
\end{align}
For real parameters these are justified with smooth localizations of the
coordinate polynomials and Gaussian moment bounds; the complex versions
follow by holomorphy. They have no omitted remainder.

\paragraph{The integrated third moment.}
Weighted transpose symmetry in \eqref{mom:M3exact} yields
\[
 J_3(t,z)=16\int_{u+v\le t}\mathcal U(u,v)\,du\,dv,
 \qquad
 \mathcal U(u,v)=\int_{\mathbb T^2}(P_ug)\cdot\nabla P_va\,dx,
\]
with $\mathcal U(0,0)=\int g\cdot\nabla a=-\mathcal I$. Integration by parts
gives
$\partial_v\mathcal U=-\int_{\mathbb T^2}\operatorname{div}(P_ug)\,P_vLa\,dx$,
bounded by $C_\theta|z|^2\sH^2$ via \eqref{mom:gradprop} and
\eqref{mom:basic}. For $u$, use weighted transpose before integrating the
third derivative by parts:
\[
 \partial_u\mathcal U
 =\int_{\mathbb T^2}(P_uLg)\cdot\nabla P_va\,dx
 =\int_{\mathbb T^2}\nabla\Delta\rho\cdot P_u[a\nabla P_va]\,dx
 =-\int_{\mathbb T^2}\Delta\rho\,
                    \operatorname{div}P_u[a\nabla P_va]\,dx.
\]
The product rule, \eqref{mom:gradprop}, and \eqref{mom:H2} give
\begin{align*}
 \|\nabla P_u[a\nabla P_va]\|_2
 &\le C_\theta\|\nabla(a\nabla P_va)\|_2
 \le C_\theta\bigl(\|\nabla a\|_4\|\nabla P_va\|_4
                       +\|a\|_\infty\|D^2P_va\|_2\bigr)\\
 &\le C_\theta(|z|^2\sG^2+|z|\sH)
 \le C_\theta|z|\sH,
\end{align*}
so $|\partial_u\mathcal U|\le C_\theta|z|^2\sH^2$ as well. These
differentiations are identities for each fixed smooth polynomial, and the
bounds use no derivatives beyond the second energy norm. Consequently
$|\mathcal U(u,v)+\mathcal I|\le C_\theta(u+v)|z|^2\sH^2$. The triangle has
area $t^2/2$ and $\int_{u+v\le t}(u+v)\,du\,dv=t^3/3$, proving
\eqref{mom:J3bound}, including its sign and coefficient $8$.

\paragraph{The integrated fourth moment.}
Write the integrands in \eqref{mom:M4exact} as
\[
 \mathcal Y(u,v)=\int_{\mathbb T^2}\mathfrak gP_u[aP_va]\,dx,\qquad
 \mathcal K_i(u,v,v')=
      \int_{\mathbb T^2}\mathfrak gP_u[a\partial_iP_v(a\partial_iP_{v'}a)]\,dx.
\]
One has $\mathcal Y(0,0)=\int_{\mathbb T^2}\mathfrak ga^2=\mathcal I$, and
\[
 |\partial_v\mathcal Y|
   =\left|\int_{\mathbb T^2}\mathfrak gP_u[aP_vLa]\,dx\right|
   \le C_\theta\|\mathfrak g\|_2\|La\|_2
   \le C_\theta|z|^3\sG^2\sH
   \le C_\theta|z|^3\sH^2.
\]
For $u$, commute $L$ through $P_u$ and use
$L(aP_va)=a\bigl[(\Delta a)P_va+2\nabla a\cdot\nabla P_va+a\Delta P_va\bigr]$,
whose $L^2$ norm is at most $C_\theta|z|\sH$ by
\eqref{mom:Lp}--\eqref{mom:basic}. Thus
$|\partial_u\mathcal Y|\le C_\theta|z|^3\sH^2$ and
\begin{equation}\label{mom:Jdifference}
 |\mathcal Y(u,v)-\mathcal I|
            \le C_\theta(u+v)|z|^3\sH^2.
\end{equation}
For the triple term, transpose $P_u$ with weight $\rho$ and integrate by
parts once,
$\mathcal K_i
  =-\int\partial_iP_u(a\mathfrak g)\,P_v(a\partial_iP_{v'}a)\,dx$;
use H\"older with exponents $4/3$ and $4$. Gradient propagation and the
product rule give
\[
 \|\nabla P_u(a\mathfrak g)\|_{4/3}
 \le C_\theta\|\nabla(a\mathfrak g)\|_{4/3}
 \le C_\theta\bigl(\|\nabla a\|_4\|g\|_4^2
                  +\|a\|_\infty\|Dg\|_2\|g\|_4\bigr)
 \le C_\theta(|z|^3\sG^3+|z|^2\sH\sG),
\]
whereas $\|P_v(a\partial_iP_{v'}a)\|_4\le C_\theta|z|\sG$. By \eqref{mom:GN},
\begin{equation}\label{mom:Kbound}
 |\mathcal K_i(u,v,v')|
 \le C_\theta(|z|^4\sG^4+|z|^3\sH\sG^2)
 \le C_\theta|z|^3\sH^2.
\end{equation}
The triple simplex has volume $t^3/6$. Combining
\eqref{mom:Jdifference}--\eqref{mom:Kbound} with \eqref{mom:M4exact} proves
\eqref{mom:H4bound}, with leading coefficient $64\cdot t^2/2=32t^2$.

\paragraph{Compensation and parameter derivatives.}
The signs in \eqref{mom:Mdef} give the exact cancellation
$\mathfrak M_t=-(J_3+8t^2\mathcal I)/(8t)-(H_4-32t^2\mathcal I)/(64t)$,
so \eqref{mom:J3bound}--\eqref{mom:H4bound} imply \eqref{mom:Mbound}. All
integrals are holomorphic in $z$, with locally uniform absolute domination
from the complex Gaussian and the estimates above. For
$\varepsilon\le r_\theta/4$ and $F=f/\varepsilon$, Lemma~\ref{lem:cauchyjets}
with $\Phi(z)=\mathfrak M_t(1+zF)$ and, by \eqref{mom:Mbound},
$M=C_\theta r_\theta^2t^2N(F)$ gives
\[
 |\partial_s^j\mathfrak M_t(1+sf)|
 \le j!\left(\frac2{r_\theta}\right)^j
       C_\theta r_\theta^2t^2\varepsilon^jN(F)
 =C_{\theta,j}t^2\varepsilon^{j-2}N(f),
\]
which proves \eqref{mom:derivatives}; the zero perturbation is exact.

\section{Closure of the scalar estimate}\label{sec:closure}

\begin{theorem}[Full-ball compensated trace estimate]\label{thm:scalar}
Fix $\theta<\infty$. There are $\varepsilon_\theta>0$ and $C_\theta$
independent of the Fourier cutoff and mode count such that the following
holds. Let $f$ be a real trigonometric polynomial supported in
$|k|\le\Lambda$, with $\varepsilon=\|f\|_{\mathcal A}\le\varepsilon_\theta$
and $t\Lambda^2\le\theta$. Then
\begin{equation}\label{eq:scalarclosed}
 \sup_{|s|\le1}|\partial_s^j\mathfrak B_t(1+sf)|
 \le C_\theta t^2\varepsilon^{j-2}N(f),\qquad j=3,4.
\end{equation}
In particular,
\[
 \left|\int_0^1(1-s)^3\partial_s^4\mathfrak B_t(1+sf)\,ds\right|
 \le\frac{C_\theta}{4}t^2\varepsilon^2N(f).
\]
\end{theorem}
\begin{proof}
Proposition~\ref{namom:identity} gives the exact decomposition
\begin{equation}\label{eq:Bcloseddecomp}
 \mathfrak B_t(\rho)
 =\underbrace{t\int_{\mathbb T^2}\int_{\mathbb R^2}\frac{|U|^2}{r_t}\,dy\,dx}_{\mathcal S_t}
   +\mathfrak M_t(\rho)
   -\underbrace{\int_{\mathbb T^2}\frac{\rho^2(m_t-\rho^{-1})^2}{m_t}\,dx}_{\mathcal J_t}.
\end{equation}
Corollary~\ref{score:cor:square} and Theorem~\ref{mom:main} bound the third
and fourth derivatives of $\mathcal S_t$ and $\mathfrak M_t$. On the
normalized complex disk $\rho=1+zF$, the clock proof gives
$|\mathcal J_t|\le C_\theta t^2|z|^2\|\Delta F\|_2^2$ by \eqref{naclock:eq:d},
the uniform bound for $m_t^{-1}$, and the second term in
\eqref{naclock:eq:complexA}; the same Cauchy circle gives its $j$th bound
$C_{\theta,j}t^2\varepsilon^{j-2}N(f)$. Summing, on the common interval given by the minimum of the
finitely many cutoff-independent amplitude radii, proves
\eqref{eq:scalarclosed}; integrating $(1-s)^3$ proves the last assertion.
\end{proof}

\begin{lemma}[Geometric derivatives]\label{lem:geometric}
For $j=3,4$ and $\|f\|_{\mathcal A}\le1/4$,
\[
 \sup_{|s|\le1}|\partial_s^j\mathcal C(sf)|
 \le C_j\varepsilon^{j-2}N(f).
\]
\end{lemma}
\begin{proof}
The exact formula is
\[
 \mathcal C(sf)=\frac14\int_{\mathbb T^2}
 \left[\frac{s^2(\Delta f)^2}{\rho_s^3}
       -\frac{2s^3\Delta f|\nabla f|^2}{\rho_s^4}
       +\frac{s^4|\nabla f|^4}{\rho_s^5}\right]dx.
\]
Differentiate the bounded rational coefficients. The first term has two
$\Delta f$ factors and the required remaining powers of $\varepsilon$; the
others satisfy $\int|\Delta f||\nabla f|^2\le\sqrt{10}\,\varepsilon N(f)$ and
$\int|\nabla f|^4\le10\varepsilon^2N(f)$ by \eqref{eq:GNbasic}.
\end{proof}

\begin{corollary}[Plane relative variations]\label{cor:planevariations}
For mean-zero $f$ in the same class,
\begin{equation}\label{eq:planevariations}
 \sup_{|s|\le1}\bigl(|D^3E_t^\infty(sf)[f^3]|+
                         |D^3R_t^\infty(sf)[f^3]|\bigr)
 \le C_\theta\varepsilon N(f),
\end{equation}
whereas both fourth variations are at most
$C_\theta\varepsilon^2N(f)$. The constants do not depend on any separation
between low and high Fourier frequencies.
\end{corollary}
\begin{proof}
Use \eqref{eq:master}, Theorem~\ref{naclock:thm:clockderivative},
Theorem~\ref{thm:scalar}, and the third- and fourth-order anisotropy
bounds in Theorem~\ref{naani:main}, whose common bound
$C\|\nabla f\|_4^4\le10C\varepsilon^2N(f)$ is at most $10C\varepsilon N(f)$
since $\varepsilon\le1$. For $R_t^\infty$, subtract the geometric bounds of
Lemma~\ref{lem:geometric}.
\end{proof}

\paragraph{Why low--high interactions cause no loss here.}
All estimates are made in physical space before Fourier summation; every
surviving term retains two Hessian factors in $L^2$ or gradient factors
converted into such a pair by \eqref{eq:GNbasic}, with coefficients bounded
on a common Wiener disk. No division by an individual frequency occurs, so
no lower spectral cutoff is needed.

\section{The quadratic response and nonlinear stability}\label{sec:response}
The quadratic response depends on the scale $\tau=t|k|^2$ through the
scalar function (not to be confused with the compensated trace $\mathfrak B_t$)
\begin{equation}\label{eq:responseB}
 \mathscr B(\tau)=\int_0^1\int_0^1
 e^{-\tau(u+v-2uv)}(1+\tau uv)^2\,du\,dv.
\end{equation}

\begin{proposition}[Exact quadratic multiplier]\label{prop:response}
For real mean-zero trigonometric $f$,
$\frac12D^2E_t^\infty(0)[f,f]
 =A\sum_{k\ne0}|\widehat f_k|^2\frac9{2t^2}(t|k|^2)\mathscr B'(t|k|^2)$.
Equivalently the multiplier is $|k|^4q(t|k|^2)$, where
\begin{equation}\label{eq:qdef}
 q(\tau)=\frac{9\mathscr B'(\tau)}{2\tau}\quad(\tau>0),\qquad q(0)=\frac14,
\end{equation}
the value at $0$ being the continuous extension
(Subsection~\ref{subsec:closedform}).
\end{proposition}
\begin{proof}
\emph{Step 1: reduction to one mode.}
Source translation commutes with the construction, so
$\frac12D^2E_t^\infty(0)$ is a translation-invariant quadratic form,
diagonal in the Fourier basis: $\frac12D^2E_t^\infty(0)[f,f]=A\sum_{k\ne0}\mathsf m_t(k)
|\widehat f_k|^2$ with $\mathsf m_t(k)=\mathsf m_t(-k)$. Take
$f=\eta\cos(k\cdot x)$. The second-order part of $\sqrt{\det h}$ is a
function of $k\cdot x$ alone, so its average over $\mathbb T^2$ is its
average over one period of $k\cdot x$; this is the only place the period
of $\mathbb T^2$ enters the plane construction on $\mathbb R^2$. The plane
construction is rotation invariant, and $h$ is unchanged under
$x\mapsto\lambda x$, $t\mapsto\lambda^2t$. Hence the averaged second-order density for
$\eta\cos(k\cdot x)$ at time $t$ equals the one for
$\eta\cos x_1$ at time $\tau=t|k|^2$. It suffices to show that, for
$\rho=1+\eta\cos x_1$,
\begin{equation}\label{eq:onemodeV}
 \frac{V_\tau^\infty}{A}=1-\frac{\eta^2}{4}\tau\mathscr B'(\tau)+O(\eta^3),
\end{equation}
since then $E_t^\infty=9A\eta^2\tau\mathscr B'(\tau)/(4t^2)+O(\eta^3)$ and
$\sum_k|\widehat f_k|^2=\eta^2/2$.

\emph{Step 2: the first kernel variation.}
Write $\varphi_\tau(w)=(4\pi\tau)^{-1}e^{-|w|^2/(4\tau)}$ and $r'=\partial_\eta r|_{\eta=0}$.
Since $\partial_\eta(\rho^{-1}\Delta)|_{\eta=0}=-\cos x_1\,\Delta$,
Duhamel's formula in the source variable gives
\[
 r'(\tau,x,y)=-\int_0^\tau\!\!\int_{\mathbb R^2}
 \varphi_{\tau-s}(\xi-x)\cos\xi_1\,\Delta_\xi \varphi_s(y-\xi)\,d\xi\,ds .
\]
Put $w=y-x$ and $s=\tau(1-u)$. Divided by $\varphi_\tau(w)$, the product of the
two Gaussians is the law of a Brownian bridge point
$\xi=x+uw+\lambda Z$ with $Z$ standard normal and
$\lambda^2=2\tau u(1-u)$. Thus $y-\xi=(1-u)w-\lambda Z$ and
$|y-\xi|^2=(1-u)^2|w|^2-2\lambda(1-u)w\cdot Z+\lambda^2|Z|^2$. Using
$\Delta \varphi_s/\varphi_s=|y-\xi|^2/(4s^2)-1/s$ and the Gaussian identities
$\mathbb E e^{i\lambda Z_1}=e^{-\lambda^2/2}$,
$\mathbb E[Z_1e^{i\lambda Z_1}]=i\lambda e^{-\lambda^2/2}$,
$\mathbb E[|Z|^2e^{i\lambda Z_1}]=(2-\lambda^2)e^{-\lambda^2/2}$, one finds
\[
 \mathbb E\Bigl[e^{i\xi_1}\frac{\Delta \varphi_s(y-\xi)}{\varphi_s(y-\xi)}\Bigr]
 =e^{ix_1+iuw_1-\lambda^2/2}
 \Bigl(\frac{(1-u)^2|w|^2-2i\lambda^2(1-u)w_1+\lambda^2(2-\lambda^2)}{4s^2}
  -\frac1s\Bigr).
\]
With $s=\tau(1-u)$ and $\lambda^2/2=\tau u(1-u)$ this simplifies to
\[
 \mathbb E\Bigl[e^{i\xi_1}\frac{\Delta \varphi_s(y-\xi)}{\varphi_s(y-\xi)}\Bigr]
 =e^{ix_1+iuw_1-\tau u(1-u)}
 \Bigl(\frac{|w|^2}{4\tau^2}-\frac{iuw_1}{\tau}-\frac1\tau-u^2\Bigr).
\]
Since $ds=\tau\,du$, this yields
$r'(x,x+w)=\varphi_\tau(w)\operatorname{Re}(e^{ix_1}\mathscr R_\tau(w))$ with
\[
 \mathscr R_\tau(w)=\int_0^1e^{iuw_1-\tau u(1-u)}
 \left(1-\frac{|w|^2}{4\tau}+iuw_1+\tau u^2\right)du.
\]

\emph{Step 3: the Gaussian square.}
Under $\varphi_\tau$, $w$ is centered Gaussian with covariance $2\tau I$.
Write $\mathscr R_\tau=\int_0^1e^{iuw_1-\tau u(1-u)}\mathsf p_u(w)\,du$ with
$\mathsf p_u(w)=\mathsf c+\tau u^2+iuw_1$ and $\mathsf c=1-|w|^2/(4\tau)$. In
$\mathbb E_{\varphi_\tau}|\mathscr R_\tau|^2$ the parameters $u,v$ produce the
factor $e^{i(u-v)w_1}\mathsf p_u(w)\overline{\mathsf p_v(w)}$. For a polynomial $Y$,
$\mathbb E[e^{i\mu w_1}Y(w)]=e^{-\tau\mu^2}\mathbb E[Y(w+2i\tau\mu e_1)]$
under the covariance $2\tau I$. With $\mu=u-v$, the shift
$w_1\mapsto w_1+2i\tau(u-v)$ turns $\mathsf p_u$ into $\mathsf c+\tau v^2+ivw_1$ and
$\overline{\mathsf p_v}$ into $\mathsf c+\tau u^2-iuw_1$. Since
$\mathbb E\mathsf c=0$, $\mathbb E\mathsf c^2=1$ (using $\mathbb E|w|^2=4\tau$ and
$\mathbb E|w|^4=32\tau^2$), $\mathbb E[\mathsf cw_1]=0$ and
$\mathbb Ew_1^2=2\tau$, and odd moments vanish, the shifted expectation is
$1+\tau^2u^2v^2+2\tau uv=(1+\tau uv)^2$. The prefactors combine as
$e^{-\tau(u-v)^2-\tau u(1-u)-\tau v(1-v)}=e^{-\tau(u+v-2uv)}$. Thus
$\mathbb E_{\varphi_\tau}|\mathscr R_\tau|^2=\mathscr B(\tau)$.

\emph{Step 4: an entropy identity.}
Here $\tau$ is the heat-flow time for $\rho=1+\eta\cos x_1$ (the scale, by
Step~1); we differentiate in it at fixed $\rho$, with
$\partial_\tau r=L_xr$, $L_x=\rho(x)^{-1}\Delta_x$. Put
\[
 \operatorname{Ent}(\tau)=\int_{\mathbb T^2}\rho(x)\int_{\mathbb R^2}
 r\log\frac{r}{\rho(y)}\,dy\,dx .
\]
Mass conservation removes the derivative of the logarithm, and
integration by parts in $x$ (with respect to $dx$, since $\rho(x)$ cancels
the $\rho^{-1}$ of the generator) gives
$-\operatorname{Ent}'(\tau)=\int_{\mathbb T^2}\int|\nabla_xr|^2/r\,dy\,dx
=(2\tau)^{-1}\int_{\mathbb T^2}\operatorname{tr}h\,dx$. The integrand is
invariant under $(x,y)\mapsto(x+\ell,y+\ell)$ for every period $\ell$, so
$\int_{\mathbb T^2}dx\int_{\mathbb R^2}dy=\int_{\mathbb R^2}dx\int_{\mathbb T^2}dy$;
we integrate by parts over $x\in\mathbb R^2$ at fixed $y$ (Gaussian bounds
remove boundary terms) and fold back. Hence
\begin{equation}\label{eq:VEnt}
 V_\tau^\infty=-\tau\operatorname{Ent}'(\tau)-D_\tau .
\end{equation}
Decompose $\log(r/\rho(y))=\log(r/\varphi_\tau)+\log \varphi_\tau-\log\rho(y)$.
Invariance of $\rho\,dx$ gives
$\int_{\mathbb T^2}\rho(x)\int r\log\rho(y)=\int_{\mathbb T^2}\rho\log\rho$. The displacement
second moment $\mathbb E_x|w|^2=4\tau m_\tau(x)$ and weighted symmetry give
$\int_{\mathbb T^2}\rho m_\tau=\tau^{-1}\int_0^\tau\!\int_{\mathbb T^2}\rho P_u(\rho^{-1})=A$.
Therefore
\[
 \frac{\operatorname{Ent}}A=-\log(4\pi\tau)-1+\operatorname{Rel}
 -\frac1A\int_{\mathbb T^2}\rho\log\rho,\qquad
 \operatorname{Rel}=\frac1A\int_{\mathbb T^2}\rho(x)\int r\log\frac r{\varphi_\tau}\,dy\,dx .
\]
Since $\int r\,dy=\int \varphi_\tau=1$, the relative entropy equals
$\frac12\int(r-\varphi_\tau)^2/\varphi_\tau+O(\eta^3)$; averaging
$\operatorname{Re}(e^{ix_1}\mathscr R_\tau)^2$ over $x_1$ gives
$\frac12|\mathscr R_\tau|^2$, so $\operatorname{Rel}=\eta^2\mathscr B/4+O(\eta^3)$.
Also $A^{-1}\int_{\mathbb T^2}\rho\log\rho=\eta^2/4+O(\eta^3)$. Consequently
$\operatorname{Ent}/A=-\log(4\pi\tau)-1+\eta^2(\mathscr B(\tau)-1)/4+O(\eta^3)$,
and this expansion may be differentiated in $\tau$: $\operatorname{Ent}/A$
is jointly smooth in $(\eta,\tau)$ for small $\eta$, $\tau>0$, since
Theorem~\ref{thm:neargaussian} bounds the $\eta$-jets of
$r$ and of $\tau\partial_\tau r$ by Gaussians. Quantitatively, the cubic
Taylor remainder of $x\log x$ is bounded by a multiple of
$|r-\varphi_\tau|^3/\min(r,\varphi_\tau)^2$; by
Theorem~\ref{thm:neargaussian}, its numerator carries
$e^{-3\kappa_-|w|^2/\tau}$ and its denominator at most
$e^{2\kappa_+|w|^2/\tau}$, and $3\kappa_--2\kappa_+=(1-5\gamma)/4>0$, so
the remainder is integrable in $w$ uniformly in the amplitude; the same
bounds hold for its $\tau\partial_\tau$ derivative.

\emph{Step 5: the anisotropy is quartic.}
By Lemma~\ref{lem:regression}, $h=\alpha I+2\tau Q$, and $Q$ is quadratic in
a residual that vanishes at $\eta=0$. Thus $h^{\rm tf}=O(\eta^2)$ and
$D_\tau=\int|h^{\rm tf}|^2/(\operatorname{tr}h+2\sqrt{\det h})=O(\eta^4)$.
Inserting Steps 4 and 5 into \eqref{eq:VEnt} proves \eqref{eq:onemodeV};
the kernel bounds of Section~\ref{sec:kernel} justify the density
quotients, expansions and time derivative.
\end{proof}

\subsection{A closed form for the response}
\label{subsec:closedform}
Write $\mathscr F(\tau)=\mathscr B'(\tau)/\tau$, so that
$q=\frac92\mathscr F$ by \eqref{eq:qdef} and $\mathscr F(0)=\frac1{18}$;
the values at $0$ follow from
$\mathscr B'(0)=\int\!\!\int(2b-a)=\frac12-\frac12=0$ and
$\mathscr B''(0)=\int\!\!\int[(a-2b)^2-2b^2]=5/18-2/9$, where
$a=u+v-2uv$ and $b=uv$. All numerical constants used below follow from
$2.7182<e<2.7183$ and finitely many terms of the exponential series;
for $0<x\le1$ the tail of $e^x$ after the term $x^n/n!$ is at most
$2x^{n+1}/(n+1)!$, which gives the upper bounds for $e^x$ used below
(for example $e^{7/10}<2.014$ and $e^{1/2}<1.6488$). Throughout this
section $z=\tau/2$, $\operatorname{Shi}$ is the hyperbolic sine integral
\cite{DLMF}, $\mathscr S(z)=e^{-z}\operatorname{Shi}(z)$,
\[
 \operatorname{Shi}(z)=\int_0^z\frac{\sinh s}{s}\,ds
 =\sum_{k\ge0}\frac{z^{2k+1}}{(2k+1)\,(2k+1)!},\qquad
 p(z)=z^3-z^2+2z+2,
\]
$\mathsf a(z)=2z^4+2z^3+2z^2+4z+1$ and $\mathsf b(z)=2z^3-4z^2+2z+1$.
For $z>0$ one has $p(z)\ge2$, $\mathsf a(z)>0$ and $\mathsf b(z)=2z(z-1)^2+1>0$.

\begin{lemma}[Closed form]\label{lem:closedform}
For $z>0$,
\begin{equation}\label{eq:closedform}
 16z^3\mathscr F(2z)=\Upsilon(z):=\frac{z^2+3}2+e^{-2z}\frac{(z-3)(z+1)}2-p(z)\mathscr S(z).
\end{equation}
\end{lemma}
\begin{proof}
\emph{Densities.} Let $g_j$ be the density on $[0,1]$ of the push-forward
of $b^j\,du\,dv$ under $(u,v)\mapsto a$. For fixed $u\ne\frac12$, $a$ is
affine in $v$ with slope $1-2u$; for $0<\sigma<\frac12$ the constraint
$0\le v\le1$ leaves $u\in[0,\sigma]\cup[1-\sigma,1]$, and for
$\frac12<\sigma<1$ it leaves $u\in[0,1-\sigma]\cup[\sigma,1]$. On the
fibre $\{a=\sigma\}$ one has $v=(\sigma-u)/(1-2u)$, $b=u(\sigma-u)/(1-2u)$
and $dv=d\sigma/|1-2u|$. Substitute $s=1-2u$ and put $x=1-2\sigma$ and
$\ell=\log|x|$. Then $b=(1-s)(s-x)/(4s)$, the admissible $u$ correspond
exactly to $|x|\le|s|\le1$ in both cases, and
\[
 g_j(\sigma)=2^{-2j-1}\int_{|x|\le|s|\le1}
       \Bigl(\frac{(1-s)(s-x)}{s}\Bigr)^j\frac{ds}{|s|},\qquad
 \frac{(1-s)(s-x)}{s}=(1+x)-s-\frac xs .
\]
The integrand is a Laurent polynomial in $s$ divided by $|s|$. Odd powers
integrate to zero over the symmetric set, while
$\int_{|x|\le|s|\le1}ds/|s|=-2\ell$ and
$\int_{|x|\le|s|\le1}s^k\,ds/|s|=2(1-x^k)/k$ for even $k\ne0$. Expanding
the square for $j=2$ and keeping the powers $s^2,s^0,s^{-2}$ gives
\[
 g_0=-\ell,\quad g_1=-\frac{(1+x)\ell}4=-\frac{(1-\sigma)\ell}2,\quad
 g_2=\frac{(1-x^2)-(x^2+4x+1)\ell}{16}
 =-\frac{(2\sigma^2-6\sigma+3)\ell}8+\frac{\sigma(1-\sigma)}4 .
\]

\emph{Laplace form.} Differentiating \eqref{eq:responseB} and dividing by
$\tau$ gives
\[
 \mathscr F(\tau)=\int_0^1\!\!\int_0^1e^{-\tau a}\Bigl[\frac{2b-a}\tau+2b^2-2ab
 -\tau ab^2\Bigr]du\,dv
 =\int_0^1e^{-\tau\sigma}\Bigl[\frac{h_{-1}}\tau+h_0+\tau h_1\Bigr]d\sigma
\]
with $h_{-1}=2g_1-\sigma g_0$, $h_0=2g_2-2\sigma g_1$, $h_1=-\sigma g_2$.
In the variable $x=1-2\sigma\in[-1,1]$,
\[
 h_{-1}=-x\ell,\qquad
 h_0=\frac{(1-x^2)-(3x^2+4x-1)\ell}8,\qquad
 h_1=\frac{(1-x)\bigl[(x^2-1)+(x^2+4x+1)\ell\bigr]}{32},
\]
and $e^{-\tau\sigma}d\sigma=\frac12e^{-z}e^{zx}dx$.

\emph{Integration.} Put $c_k(z)=\int_{-1}^1e^{zx}x^k\,dx$ and
$\mathscr L_k(z)=\int_{-1}^1e^{zx}x^k\log|x|\,dx$. By the Laplace form,
$2e^z\mathscr F(2z)=\int_{-1}^1e^{zx}[h_{-1}/(2z)+h_0+2zh_1]\,dx$. Here
$2zh_1=\frac z{16}[(1-x)(x^2-1)+(1-x)(x^2+4x+1)\ell]$ with
$(1-x)(x^2-1)=-x^3+x^2+x-1$ and $(1-x)(x^2+4x+1)=-x^3-3x^2+3x+1$, so
termwise integration gives
\begin{align*}
 2e^z\mathscr F(2z)={}&\frac{c_0-c_2+\mathscr L_0-4\mathscr L_1-3\mathscr L_2}8-\frac{\mathscr L_1}{2z}\\
 &+\frac z{16}\bigl(-c_0+c_1+c_2-c_3+\mathscr L_0+3\mathscr L_1-3\mathscr L_2-\mathscr L_3\bigr).
\end{align*}
The $c_k$ are elementary, and $\mathscr L_k=(d/dz)^k\mathscr L_0$ with
$\mathscr L_0(z)=2\int_0^1\cosh(zx)\log x\,dx=-2\operatorname{Shi}(z)/z$
by one integration by parts. Differentiation gives
\begin{align*}
 \mathscr L_1&=\frac{2(\operatorname{Shi}z-\sinh z)}{z^2},\qquad
 \mathscr L_2=\frac{6\sinh z-2z\cosh z-4\operatorname{Shi}z}{z^3},\\
 \mathscr L_3&=\frac{12\operatorname{Shi}z-22\sinh z+10z\cosh z-2z^2\sinh z}{z^4},
\end{align*}
and $c_0=2\sinh z/z$, $c_1=2(z\cosh z-\sinh z)/z^2$,
$c_2=2((z^2+2)\sinh z-2z\cosh z)/z^3$,
$c_3=2((z^3+6z)\cosh z-(3z^2+6)\sinh z)/z^4$.
Substituting these expressions and collecting the coefficients of
$\operatorname{Shi}z$, $\cosh z$ and $\sinh z$ (group by group in
Table~\ref{tab:closedform} of Appendix~\ref{app:closedform}) gives
\[
 2e^z\mathscr F(2z)=-\frac{p(z)\operatorname{Shi}z}{8z^3}
   +\frac{(z^2-z)\cosh z+(z+3)\sinh z}{8z^3}.
\]
Multiplying by $8z^3e^{-z}$ and using
$e^{-z}\cosh z=\frac12(1+e^{-2z})$, $e^{-z}\sinh z=\frac12(1-e^{-2z})$
gives \eqref{eq:closedform}. (Check, not used below: the right side
begins $\frac89z^3-\frac43z^4$, matching $\mathscr F(0)=\frac1{18}$ and $\mathscr F'(0)=-\frac1{24}$.)
\end{proof}

\begin{lemma}[A first-order identity]\label{lem:omega}
Put $\mathscr V(z)=\mathsf a(z)-e^{2z}\mathsf b(z)$ and $\Omega(z)=e^z\Upsilon(z)/p(z)$. Then
$\Omega(0^+)=0$ and
\begin{equation}\label{eq:omegaprime}
 \Omega'(z)=\frac{2e^{-z}\mathscr V(z)}{z\,p(z)^2}\qquad(z>0).
\end{equation}
The Taylor coefficients of $\mathscr V$ are $0,0,0,\frac83,\frac83$ in degrees
$0,\dots,4$ and
\begin{equation}\label{eq:vcoeff}
 -v_n,\qquad v_n=\frac{2^{n-3}}{n!}\bigl[2n(n-1)(n-6)+8n+8\bigr]>0,
 \qquad n\ge5 .
\end{equation}
Consequently $\mathscr V$ has exactly one positive zero $z_1$, with
$\frac{27}{20}<z_1<\frac32$; $\mathscr V>0$ on $(0,z_1)$ and $\mathscr V<0$ on $(z_1,\infty)$.
Moreover $\mathscr V(z)\le\frac83z^3(1+z)$ for all $z\ge0$, and
$\mathscr V(z)\ge\frac83z^3$ for $0\le z\le1$.
\end{lemma}
\begin{proof}
Since $\mathscr S'=-\mathscr S+(1-e^{-2z})/(2z)$, differentiation of
\eqref{eq:closedform} gives
\[
 \Upsilon'+\Upsilon=\frac{3z^2+z-2+e^{-2z}(3z^2+3z+2)}{2z}-p'(z)\mathscr S(z).
\]
In $p(\Upsilon'+\Upsilon)-p'\Upsilon$ the terms $\mp pp'\mathscr S$ cancel, and
the polynomial identities
\[
 p(3z^2+z-2)-zp'(z^2+3)=-4\mathsf b(z),\qquad
 p(3z^2+3z+2)-zp'(z-3)(z+1)=4\mathsf a(z)
\]
give $p(\Upsilon'+\Upsilon)-p'\Upsilon=2e^{-2z}\mathscr V/z$. Since
$\Omega'=e^z[p(\Upsilon'+\Upsilon)-p'\Upsilon]/p^2$, this is
\eqref{eq:omegaprime}.
$\Omega(0^+)=0$ because $\Upsilon(z)=16z^3\mathscr F(2z)\to0$.
The coefficients of $\mathsf a$ are $1,4,2,2,2$; those of $e^{2z}\mathsf b$ are
$1,4,2,-\frac23,-\frac23$ in degrees $0,\dots,4$, and for $n\ge5$ the
coefficient is the displayed $v_n$, as one sees from
$\frac{2^n}{n!}+\frac{2^n}{(n-1)!}-\frac{2^n}{(n-2)!}+\frac{2^{n-2}}{(n-3)!}$.
For $n=5$ the bracket in \eqref{eq:vcoeff} equals $8$; for $n\ge6$ every
term is nonnegative. Hence $\mathscr V(z)/z^4=\frac8{3z}+\frac83-\sum_{n\ge5}v_nz^{n-4}$
is strictly decreasing on $(0,\infty)$; it tends to $+\infty$ at $0$ and
to $-\infty$ at $\infty$, so $\mathscr V$ has exactly one positive zero. The same
expansion gives the two polynomial bounds, the second because
$(\mathscr V-\frac83z^3)/z^4$ is decreasing and
$\mathscr V(1)-\frac83=\frac{25}3-e^2>0$. Finally
$\mathscr V(\frac{27}{20})>0$ because $e^{27/10}=e^2e^{7/10}<7.39\cdot2.014<14.9$
while, with $\mathsf a(\frac{27}{20})=\frac{1728701}{80000}$ and
$\mathsf b(\frac{27}{20})=\frac{5323}{4000}$,
$\mathsf a(\frac{27}{20})/\mathsf b(\frac{27}{20})=\frac{1728701}{106460}>16$, and
$\mathscr V(\frac32)=\frac{227}8-\frac74e^3<0$ because $e^3>16.3$.
\end{proof}

\begin{proposition}[Global sign and size of the response]
\label{prop:responseglobal}
\begin{enumerate}
\item[(i)] $q$ has exactly one positive zero $\tau_*$, and
$4<\tau_*<5$. Moreover $q>0$ on $[0,\tau_*)$ and $q<0$ on $(\tau_*,\infty)$.
\item[(ii)] $-\frac1{40}\le q(\tau)\le\frac14$ for every $\tau\ge0$.
\item[(iii)] For $0\le\tau\le2$,
$q(\tau)\ge e^{-\tau}/(2p(\tau/2))$. In particular $q\ge\frac1{16}$ on
$[0,1]$ and $q\ge\frac1{60}$ on $[0,2]$.
\item[(iv)] $\tau^4q(\tau)\to-18$ as $\tau\to\infty$.
\end{enumerate}
\end{proposition}
\begin{proof}
By \eqref{eq:closedform},
\begin{equation}\label{eq:qomega}
 q(2z)=\frac{9\Upsilon(z)}{32z^3}=\frac{9e^{-z}p(z)\,\Omega(z)}{32z^3},
\end{equation}
so $q(2z)$ has the sign of $\Omega(z)$. By Lemma~\ref{lem:omega},
$\Omega$ increases strictly on $(0,z_1]$ from $\Omega(0^+)=0$ and decreases
strictly on $[z_1,\infty)$.

\emph{(i).} We evaluate $\Upsilon$ at two points. The series for
$\operatorname{Shi}$ has positive terms, so its partial sums are lower
bounds; at $z=2$ the ratio of consecutive terms after the sixth is below
$\frac1{40}$, so the remainder after six terms is at most $\frac{40}{39}$
times the seventh term $2^{13}/(13\cdot13!)<1.1\cdot10^{-7}$. At $z=\frac52$,
$p=\frac{131}8$ and $\Upsilon(\frac52)=\frac{37}8-\frac78e^{-5}-\frac{131}8\mathscr S(\frac52)$;
since $\operatorname{Shi}(\frac52)\ge\frac52+\frac{125}{144}+\frac{125}{768}
+\frac{15625}{903168}>3.548$ and $e^{5/2}<12.19$, we get
$\frac{131}8\mathscr S(\frac52)>4.76>\frac{37}8$, so $\Upsilon(\frac52)<0$. At $z=2$,
$p=10$ and $\Upsilon(2)=\frac72-\frac32e^{-4}-10\mathscr S(2)$; since $\operatorname{Shi}(2)<2+\tfrac49+\tfrac4{75}+\tfrac8{2205}+\tfrac4{25515}
 +\tfrac8{1715175}+2\cdot10^{-7}<2.5016$ and $e^{-2}<0.1354$, we get $10\mathscr S(2)<3.388$, while
$\frac72-\frac32e^{-4}>3.47$, so $\Upsilon(2)>0$. Since $\Omega>0$ on
$(0,z_1]$ and $\Omega$ is strictly decreasing afterwards, it has exactly
one zero $z_*$, which lies in $(2,\frac52)$. Put $\tau_*=2z_*$.

\emph{Upper bound in (ii).} By Lemma~\ref{lem:omega}, $\mathscr V(s)\le\frac83s^3(1+s)$;
with $e^{-s}(1+s)\le1$ and $p(s)\ge2$, \eqref{eq:omegaprime} gives
$\Omega(z)\le\frac{16}3\int_0^z\frac{s^2}4\,ds=\frac49z^3$. Hence
$q(2z)\le e^{-z}p(z)/8\le\frac14$, because
$\frac d{dz}(e^{-z}p(z))=-z(z-2)^2e^{-z}\le0$ and $p(0)=2$.

\emph{(iii).} For $0\le s\le z\le1$, Lemma~\ref{lem:omega} gives
$\mathscr V(s)\ge\frac83s^3$, while $e^{-s}\ge e^{-z}$ and $p(s)\le p(z)$, since $p$
is increasing. Thus $\Omega(z)\ge\frac{16}9e^{-z}z^3/p(z)^2$ and
$q(2z)\ge e^{-2z}/(2p(z))$. The right side is decreasing; at $\tau=1$ it
equals $4/(23e)>\frac1{16}$, and at $\tau=2$ it equals $e^{-2}/8>\frac1{60}$.

\emph{Lower bound in (ii).} Only $z\ge z_*>2$ matters. Put
$\beta(s)=\mathsf b(s)/(s\,p(s)^2)$. Since $\Omega(z_1)>0$ and
$\mathscr V\ge-e^{2s}\mathsf b$,
\[
 \Omega(z)\ge-2\int_{z_1}^ze^s\beta(s)\,ds\ge-2\beta(z_1)e^z
\]
provided $\beta$ is decreasing on $[z_1,\infty)$. Then \eqref{eq:qomega}
gives $q(2z)\ge-\frac9{16}\,\frac{p(z)}{z^3}\,\beta(z_1)$. For $z\ge2$,
$p(z)\le\frac54z^3$, since $\frac14z^3+z^2-2z-2$ vanishes at $2$ and is
increasing. Also $z_1>\frac{27}{20}$, and since
$p(\frac{27}{20})=\frac{42703}{8000}$ and
$\mathsf b(\frac{27}{20})=\frac{5323}{4000}$,
$\beta(\frac{27}{20})=\frac{1703360000}{49235747643}<\frac8{225}$. Hence
$q\ge-\frac{45}{64}\cdot\frac8{225}=-\frac1{40}$. It remains to check that
$\beta$ decreases on $[\frac{27}{20},\infty)$, i.e.\ that
\[
 \Pi(s)=s\,p\mathsf b'-p\mathsf b-2s\,p'\mathsf b=-8s^3(s-1)^3+9s^3-11s^2-6s-2<0 .
\]
Write $\Pi=s^3[-8(s-1)^3+r(s)]$ with $r(s)=9-11/s-6/s^2-2/s^3$, which is
increasing. For $\frac{27}{20}\le s\le\frac{17}{10}$, $r(s)\le r(\frac{17}{10})<0.3$
and $8(s-1)^3\ge8(\frac7{20})^3>0.34$. For $\frac{17}{10}\le s\le\frac{41}{20}$,
$r(s)\le r(\frac{41}{20})<2$ and $8(s-1)^3\ge8(\frac7{10})^3>2.7$. For
$s\ge\frac{41}{20}$, $8(s-1)^3\ge8(\frac{21}{20})^3>9>r(s)$.

\emph{(iv).} By \eqref{eq:omegaprime} and l'H\^opital's rule applied to
$\Omega(z)/(e^zz^{-4})$, whose denominator has derivative
$e^zz^{-4}(1-4/z)$,
$\Omega(z)z^4e^{-z}\to\lim2\mathscr V(z)e^{-2z}z^3/p(z)^2=-4$, since
$\mathscr V\sim-2z^3e^{2z}$ and $p\sim z^3$. Then \eqref{eq:qomega} gives
$z^4q(2z)\to-\frac98$, that is, $\tau^4q(\tau)\to-18$.
\end{proof}

\begin{proposition}[An explicit positive band up to four]
\label{prop:explicitfour}
The response satisfies
\begin{equation}\label{eq:explicitfour}
 q(\tau)\ge\frac3{1000}\qquad(0\le\tau\le4).
\end{equation}
The bound is analytic, not a numerical minimization of $q$.
\end{proposition}
\begin{proof}
For $0\le\tau\le2$, Proposition~\ref{prop:responseglobal}(iii) gives
$q(\tau)\ge1/60>3/1000$. For $\tau=2z$ with $1\le z\le2$: $\Omega$
increases up to its unique maximum at $z_1\in(27/20,3/2)$ and decreases
thereafter, so its minimum on $[1,2]$ is at an endpoint, where
\[
 \Omega(1)=\frac e2-\frac{e^{-1}}2-\operatorname{Shi}(1),\qquad
 \Omega(2)=\frac{7e^2}{20}-\frac{3e^{-2}}{20}
                                      -\operatorname{Shi}(2).
\]
The bounds $e>2.7182$, $e^{-1}<0.368$, $\operatorname{Shi}(1)<1.058$
(first four terms plus a geometric tail), $e^2<7.39$,
$\operatorname{Shi}(2)>2.5015$ give $\Omega(1)>0.1171>1/10$, whereas
$\Omega(2)<0.085<1/10$. Hence $\Omega(z)\ge\Omega(2)$ on $[1,2]$, and
$\Omega(2)>0$ because $\Upsilon(2)>0$, as shown below.
The factor $e^{-z}p(z)/z^3$ is decreasing (its positive numerator has derivative
$-z(z-2)^2e^{-z}\le0$), so \eqref{eq:qomega} yields
$q(2z)\ge q(4)=\frac9{256}\Upsilon(2)$. Finally, $e^{-2}<677/5000$ and
$\operatorname{Shi}(2)<3127/1250$ imply the exact rational bound
$\Upsilon(2)=\frac72-\frac32e^{-4}-10e^{-2}\operatorname{Shi}(2)
 >\frac72-\frac32(\frac{677}{5000})^2-10\frac{677}{5000}\frac{3127}{1250}$.
Here $(677/5000)^2=458329/25000000$, so the middle term is
$1374987/50000000$, and the last term is $2116979/625000=169358320/50000000$.
Hence the right side equals
\[
 \frac{175000000-1374987-169358320}{50000000}
 =\frac{4266693}{50000000}>\frac{4266666.\overline6}{50000000}
 =\frac{32}{375}.
\]
Multiplication by $9/256$ proves \eqref{eq:explicitfour}.
\end{proof}

\begin{proposition}[A simple resolution zero and its coercivity margin]
\label{prop:simplezero}
The zero $\tau_*$ is simple and $q'(\tau_*)<0$. If
$m_\theta=\min_{0\le\tau\le\theta}q(\tau)$, there are
$\delta_0,c_0,C_0>0$ such that
\begin{equation}\label{eq:linearmargin}
 m_{\tau_*-\delta}=q(\tau_*-\delta),\qquad
 c_0\delta\le m_{\tau_*-\delta}\le C_0\delta,
 \qquad 0<\delta\le\delta_0.
\end{equation}
\end{proposition}
\begin{proof}
Write $P(z)=9e^{-z}p(z)/(32z^3)>0$, so $q(2z)=P(z)\Omega(z)$.
The zero $z_* =\tau_*/2$ lies beyond $z_1$, so
$\Omega'(z_*)<0$ by Lemma~\ref{lem:omega}, and since $\Omega(z_*)=0$,
$q'(\tau_*)=\tfrac12P(z_*)\Omega'(z_*)<0$.
By continuity $q'$ lies between two negative constants on an interval
left of $\tau_*$; $q$ has a positive minimum on the remaining compact
interval, and decreasing $\delta_0$ until $q(\tau_*-\delta_0)$ is below it
puts the minimum on $[0,\tau_*-\delta]$ at the right endpoint.
Integrating the derivative bounds from $\tau_*-\delta$ to $\tau_*$
proves \eqref{eq:linearmargin}.
\end{proof}

\begin{proposition}[$\tau_*$ is the scale of maximal entropy response]
\label{prop:entropyresponse}
The function $\mathscr B$ of \eqref{eq:responseB} satisfies $\mathscr B(0)=1$,
increases strictly on $[0,\tau_*]$, decreases strictly on $[\tau_*,\infty)$,
and $\mathscr B(\tau)\to1$ as $\tau\to\infty$. Hence $\mathscr B>1$ on
$(0,\infty)$ and $\tau_*$ is its unique maximum. For real mean-zero
trigonometric $f$ the fixed-time plane entropy
$\operatorname{Ent}_t=\operatorname{Ent}(t)$ of Step~4 of the proof of
Proposition~\ref{prop:response} with $\rho=1+f$ (see also Section~\ref{sec:entropy}) satisfies
$\tfrac12D^2\operatorname{Ent}_t(0)[f,f]
   =\frac A2\sum_{k\ne0}\bigl(\mathscr B(t|k|^2)-1\bigr)|\widehat f_k|^2>0$
for $f\ne0$, while $\frac12D^2E_t^\infty(0)$ has multiplier
$\frac9{2t^2}\,\tau\mathscr B'(\tau)$ at $\tau=t|k|^2$.
\end{proposition}
\begin{proof}
Since $q=9\mathscr B'/(2\tau)$, the monotonicity is
Proposition~\ref{prop:responseglobal}(i). For the limit write, by the
push-forward densities of the proof of Lemma~\ref{lem:closedform},
$\mathscr B(\tau)=\int_0^1e^{-\tau\sigma}(g_0+2\tau g_1+\tau^2g_2)\,d\sigma$.
For $0\le\sigma\le\frac14$ one has $0\le-\ell\le4\sigma$, hence
$0\le g_0\le4\sigma$ and $0\le g_1\le2\sigma$, and the expansion
$\ell=-2\sigma-2\sigma^2+O(\sigma^3)$ gives $|g_2-\sigma|\le C\sigma^2$. On
$[\frac14,1]$ the densities are bounded by $C(1+|\ell|)$, which is
integrable, and $e^{-\tau\sigma}\le e^{-\tau/4}$. Therefore
$\mathscr B(\tau)=\tau^2\int_0^{1/4}\sigma e^{-\tau\sigma}d\sigma+O(\tau^{-1})
=1+O(\tau^{-1})$. The entropy formula is the expansion
$\operatorname{Ent}/A=-\log(4\pi\tau)-1+\eta^2(\mathscr B(\tau)-1)/4+O(\eta^3)$
from Step~4 of the proof of Proposition~\ref{prop:response}, with
$\sum_k|\widehat f_k|^2=\eta^2/2$ and translation invariance; the last
assertion restates Proposition~\ref{prop:response}.
\end{proof}
Thus $\tau_*$ is where the $\chi^2$-size
$\mathscr B=\mathbb E|\mathscr R_\tau|^2$ of the first kernel variation stops
growing.

% Insert after Proposition entropyresponse, before the numerical remark.
\subsection{The observation scale as a statistical design parameter}
\label{subsec:statisticaldesign}
There are two distinct Fisher informations in this paper. The tensor
$J_t(x)$ estimates the unknown \emph{source location} for a known metric.
The following scalar information instead estimates one unknown
\emph{metric amplitude}, when both the source and the endpoint are observed.

\begin{proposition}[Information in one resolved curvature mode]
\label{prop:amplitudeinformation}
Fix a nonzero dual frequency $k$ and the family
$g_\eta=(1+\eta\cos(k\cdot x))g_0$. Choose $X$ uniformly in a fundamental
cell, independently of $\eta$, and, given $X=x$, observe the lifted endpoint
$Y\in\R^2$ with transition density $r_\eta(t,x,y)$. At $\eta=0$ the Fisher
information in the observed pair $(X,Y)$ for the parameter $\eta$ is
\begin{equation}\label{eq:amplitudeinformation}
 I_\eta(t;k):=\frac1A\int_{\mathbb T^2}\int_{\R^2}
 \left(\left.\partial_\eta\log r_\eta(t,x,y)\right|_{\eta=0}\right)^2
 r_0(t,x,y)\,dy\,dx
 =\frac12\mathscr B(t|k|^2).
\end{equation}
Consequently, for this local, one-mode, per-observation experiment,
$t=\tau_*/|k|^2$ is the unique information-maximizing observation time.
With $I(\tau)=\mathscr B(\tau)/2$,
\begin{equation}\label{eq:qinformation}
 I(0)=I(\infty)=\tfrac12,\qquad
 q(\tau)=\frac9\tau I'(\tau).
\end{equation}
If instead the parameter is the linearized curvature amplitude
$\kappa=\eta|k|^2/2$, its information is $I_\kappa(t;k)=2\mathscr B(t|k|^2)/|k|^4$.
\end{proposition}
\begin{proof}
After rotation and parabolic rescaling, Step~2 of
Proposition~\ref{prop:response} gives the amplitude score
$\operatorname{Re}(e^{ix_1}\mathscr R_\tau(w))$ under the unperturbed
Gaussian transition law. Averaging its square over the uniform phase
$x_1$ gives $\tfrac12|\mathscr R_\tau(w)|^2$. The Gaussian square in Step~3
is exactly $\mathscr B(\tau)$, which proves
\eqref{eq:amplitudeinformation}. The score has conditional mean zero
because $\int r_\eta(t,x,y)\,dy=1$. Proposition~\ref{prop:entropyresponse}
therefore proves the asserted unique maximum and endpoint limits;
\eqref{eq:qinformation} is \eqref{eq:qdef}. Finally
$K_{g_\eta}=\eta |k|^2\cos(k\cdot x)/2+O(\eta^2)$, and information
transforms by $(d\eta/d\kappa)^2=4/|k|^4$.
\end{proof}

More generally, if the metric path is $g_\eta=(1+\eta f)g_0$ with
real mean-zero trigonometric $f$, translation invariance diagonalizes its
amplitude information:
\[
 I_f(t)=\sum_{k\ne0}|\widehat f_k|^2\mathscr B(t|k|^2),\qquad
 \frac12D^2E_t^\infty(0)[f,f]=\frac{9A}{2t}\frac{d}{dt}I_f(t).
\]
This identity follows from the one-mode calculation, polarization and
Proposition~\ref{prop:response}. It does not select a common optimal time
for a mixture of different frequencies.

Thus the sign of curvature coercivity has an observational meaning:
the quadratic heat-Fisher deficit is positive precisely while the
information about this metric mode is increasing with observation time.
At $\tau_*$ that increase stops; the information itself is still strictly
positive. Numerically $I(\tau_*)\simeq0.53290255$, about $6.58\%$ above its
endpoint value. These numbers illustrate the exact identities and are
not used in their proof. The two endpoint regimes retain different
information: at small time, local diffusivity; at long time, the
stationary phase density of the periodic medium. In particular,
$\tau_*$ is not an identifiability boundary.

\begin{figure}[tbp]
\centering
\includegraphics[width=\textwidth]{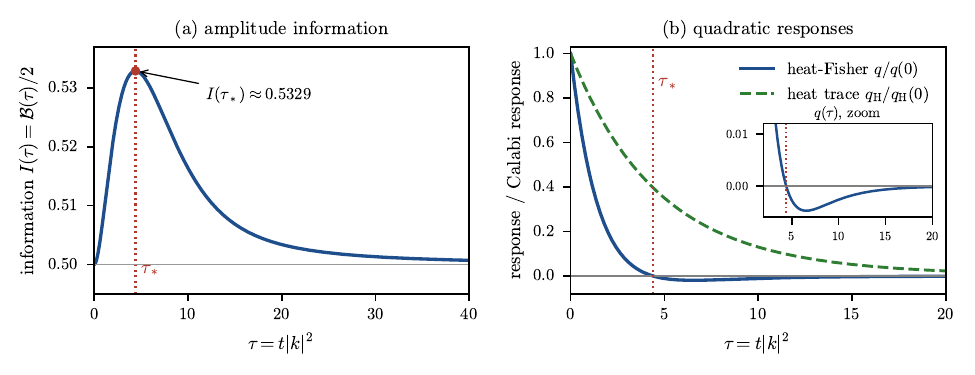}
\caption{Why $\tau_*$ is a statistical threshold. (a) The information
$I(\tau)=\mathscr B(\tau)/2$ about one metric mode in the lifted
experiment of Proposition~\ref{prop:amplitudeinformation}: it equals
$1/2$ at both ends and has a single interior maximum at $\tau_*$.
(b) The heat-Fisher quadratic response $q=9I'/\tau$, normalized by its
Calabi value $q(0)=1/4$, changes sign exactly at that maximum (inset:
the small negative lobe), whereas the heat-trace response $q_H/q_H(0)$
of Proposition~\ref{prop:heattrace} stays positive at every scale.
Curves are quadrature evaluations of proved formulas and are not
used in any proof.}
\label{fig:threeresponses}\label{fig:response}
\end{figure}

\begin{remark}[What is observed and what is optimized]
\label{rem:observationdesign}
The source design in Proposition~\ref{prop:amplitudeinformation} is fixed.
If $X$ is instead drawn from the unknown stationary density
$(1+\eta\cos(k\cdot x))/A$, its own score contributes $\cos(k\cdot X)$;
the joint information becomes $(1+\mathscr B(\tau))/2$, because the
cross term vanishes by conditional centering. If only the wrapped
endpoint $Y\bmod\mathcal L$ is recorded, its score is the conditional
expectation of the lifted score given the wrapped observation, so its
information is at most \eqref{eq:amplitudeinformation}; equality and the
same exact maximizing time are not asserted for that experiment.
Nor do we optimize information per unit running time, simultaneous
estimation of several modes, or nonzero-amplitude information. The
exact design statement is for the plane-lift experiment at the flat
metric. In the short-time, high-frequency regime the compact comparison
explains why this is the relevant local model.
\end{remark}

\begin{remark}[Numerical values]\label{rem:numerics}
$\tau_*\approx4.40822468$ and $\min q\approx-0.004741$ near
$\tau\approx6.55$ (not used in any proof).
\end{remark}

\begin{corollary}[Quantitative quadratic positivity]
\label{cor:largerband}
Let $f$ be a real mean-zero trigonometric polynomial supported in
$|k|\le\Lambda$. If $t\Lambda^2\le\Theta\le2$, then
$\frac{e^{-\Theta}}{2p(\Theta/2)}N(f)\le\frac12D^2E_t^\infty(0)[f,f]
\le\frac14N(f)$. In particular the lower constant is $\frac1{16}$ for $\Theta=1$ and
$\frac1{60}$ for $\Theta=2$. For $t\Lambda^2\le4$, one also has the explicit
lower constant $3/1000$ by Proposition~\ref{prop:explicitfour}. For every fixed $0<\Theta<\tau_*$ the same
holds with the positive lower constant $m_\Theta=\min_{[0,\Theta]}q$.
\end{corollary}
\begin{proof}
Use the exact multiplier, Parseval and
Propositions~\ref{prop:responseglobal} and~\ref{prop:explicitfour};
$m_\Theta>0$ by Proposition~\ref{prop:responseglobal}(i) and continuity.
\end{proof}

\begin{corollary}[A sharp boundary for uniform resolved-band positivity]
\label{cor:responseobstruction}
For every $\Theta>\tau_*$ there are $t_M\downarrow0$ and mean-zero Fourier
modes $\phi_M(x)=\cos(Mk_0\cdot x)$, $0\ne k_0\in\mathcal L^\vee$,
with $t_MM^2|k_0|^2<\Theta$ for which
$D^2E_{t_M}^T(0)[\phi_M,\phi_M]<0$ for all sufficiently large $M$. At those
same times a fixed low-frequency mode $\psi$ has positive second variation,
so the flat metric is a saddle point of $E_{t_M}^T$ restricted to
$\operatorname{span}\{\psi,\phi_M\}$. One explicit choice, admissible
whenever $\Theta>5$, is $t_M=5/(M^2|k_0|^2)$.
Consequently a uniform positive quadratic coercivity
constant on the full band cannot extend to $\Theta\ge\tau_*$.
\end{corollary}
\begin{proof}
Choose $\tau_0\in(\tau_*,\Theta)$ and put
$t_M=\tau_0/(M^2|k_0|^2)$. The plane multiplier identity and
Proposition~\ref{prop:responseglobal}(i) give
$\tfrac12D^2E_{t_M}^\infty(0)[\phi_M,\phi_M]/N(\phi_M)=q(\tau_0)<0$.
The $j=2$ part of Lemma~\ref{lem:compact} changes this ratio by at most
$C_{\tau_0}t_M^{-3}e^{-c\ell_0^2/t_M}\to0$ (apply it to $\delta\phi_M$ with a small
fixed $\delta>0$, then divide by $\delta^2$, so its small-amplitude
hypothesis is harmless). Thus the compact second variation is negative for large $M$.
For $\psi(x)=\cos(k_0\cdot x)$ the same normalized compact second variation
tends to $q(0)=1/4>0$. For $M\ge2$, $\psi,\phi_M$ have disjoint spectra, so
$D^2E_{t_M}^\infty(0)[\psi,\phi_M]=0$ and, by Lemma~\ref{lem:compact}, the compact
cross term is $O(t_M^{-3}e^{-c\ell_0^2/t_M})$. Source-translation invariance makes the flat metric
stationary against these mean-zero conformal directions, and smooth dependence
on each finite amplitude gives both increasing and decreasing nearby
metrics in $\operatorname{span}\{\psi,\phi_M\}$.

The value $5$ exceeds $\tau_*$ by
Proposition~\ref{prop:responseglobal}(i). At $\Theta=\tau_*$ take
$t_M=\tau_*/(M^2|k_0|^2)$: the plane normalized second variation is exactly zero
and its compact counterpart tends to zero, so no positive lower constant
uniform in $M$ holds. A nonlinear quadratic lower bound would imply
such a Hessian bound as the amplitude tends to zero, so it is ruled out as well.
\end{proof}

\begin{remark}[Small observation times and the normalization]
\label{rem:normalization}
The normalization $9/t^2$ makes $|k|^4q(t|k|^2)\to|k|^4/4$ as $t\to0$, the
multiplier of $\frac12D^2\mathcal C(0)$ (since $K_{g_f}=-\frac12\Delta f+O(f^2)$,
$dA_{g_f}=(1+f)dx$). The expected expansion
$V_t=A+ta_1\int K\,dA+t^2a_2\int K^2\,dA+o(t^2)$ (Gauss--Bonnet removes
the $O(t)$ term on the torus) is then forced to have $a_2=-1/9$.
Theorem~\ref{thm:fixedlimit} proves it, with $a_1=-1$ and $a_2=-1/9$, using
\cite{MPR}; Corollary~\ref{cor:Eclose} gives the uniform version. The
stability bounds of this paper do not use it.
\end{remark}

\subsection{Comparison with the fixed-time heat trace}
\label{subsec:heattrace}
On a closed surface $\operatorname{Tr}e^{t\Delta_g}=(4\pi t)^{-1}\bigl(A_g+\frac t3\int K\,dA
+\frac{t^2}{15}\int K^2\,dA+O(t^3)\bigr)$ \cite{MP,McKeanSinger,Gilkey}.
Conformal perturbations of heat kernels and traces via the Duhamel series
go back to Morpurgo \cite{Morpurgo,Morpurgo2002}, and heat-trace critical
metrics, including under conformal variations, were characterized in
\cite{ElSoufiIlias}; the flat metric is one, since the first variation
below vanishes for mean-zero $f$. We evaluate the second-order Duhamel term
there at fixed $t>0$ for every mode. Put
$\operatorname{Tr}_t(f)=\int_{\mathbb T^2}r_f(t,x,x)\,dx$, the local heat trace of the plane lift. The torus trace
differs from it by $\int_{\mathbb T^2}\sum_{0\ne n\in\mathcal L}r_f(t,x,x+n)\,dx$, which for
$t\Lambda^2\le\theta$ with any fixed $\theta<\infty$ (not requiring $\theta<\tau_*$) and small Wiener amplitude
is $O(t^{-1}e^{-c\ell_0^2/t})$ with its parameter jets, by \eqref{eq:nonzerodensityjets}
of Theorem~\ref{thm:neargaussian}, with constants depending on $\theta$.

\begin{proposition}[The heat-trace multiplier has no sign change]
\label{prop:heattrace}
For real mean-zero trigonometric $f$ and $t>0$,
\[
 \frac{60\pi}t\cdot\frac12D^2\operatorname{Tr}_t(0)[f,f]
   =A\sum_{k\ne0}|k|^4q_{\rm H}(t|k|^2)|\widehat f_k|^2,\qquad
 q_{\rm H}(\tau)=\frac{15}2\int_0^1s^2(1-s)^2e^{-\tau s(1-s)}\,ds .
\]
Thus $q_{\rm H}(0)=\frac14=q(0)$, and $q_{\rm H}$ is positive and strictly
decreasing on $[0,\infty)$ with $\tau^3q_{\rm H}(\tau)\to30$. The normalized
fixed-time heat trace has no resolution threshold, while its quadratic
coercivity relative to $N(f)$ degenerates like $(t\Lambda^2)^{-3}$.
\end{proposition}
\begin{proof}
By translation invariance the form is diagonal, so take
$f=\eta\cos(k\cdot x)$. Write $L=(1+f)^{-1}\Delta=\Delta+V_1+V_2+O(\eta^3)$
with $V_1=-f\Delta$ and $V_2=f^2\Delta$. The first-order Duhamel term has
diagonal average $t\,\overline f\cdot(4\pi t^2)^{-1}=0$. To second order,
the per-area trace density is the sum of
$t\,\overline{f^2}\int(-|\xi|^2)e^{-t|\xi|^2}\frac{d\xi}{4\pi^2}
=-\eta^2/(8\pi t)$ and, by cyclicity,
\[
 \frac{t^2}2\int_0^1\!\!\int\frac{\eta^2}{4}\sum_\pm|\xi|^2|\xi\pm k|^2
 e^{-t[(1-s)|\xi|^2+s|\xi\pm k|^2]}\frac{d\xi}{4\pi^2}\,ds .
\]
Here $V_1$ maps $e^{i\xi\cdot x}$ to $\frac\eta2|\xi|^2(e^{i(\xi+k)\cdot x}
+e^{i(\xi-k)\cdot x})$. With $\zeta=\xi+sk$ one has
$(1-s)|\xi|^2+s|\xi+k|^2=|\zeta|^2+s(1-s)|k|^2$, and the moments
$\mathbb E\zeta_i^2=1/(2t)$, $\mathbb E|\zeta|^4=2/t^2$ of the weight
$e^{-t|\zeta|^2}$ give
$\mathbb E\,|\zeta-sk|^2|\zeta+(1-s)k|^2
 =(\tau^2\sigma^2+\tau(1-2s)^2+2)/t^2$ with $\sigma=s(1-s)$, $\tau=t|k|^2$.
Since $\int e^{-t|\zeta|^2}d\zeta/(4\pi^2)=(4\pi t)^{-1}$, the second-order
density equals $\eta^2(16\pi t)^{-1}T(\tau)$ with
$T(\tau)=-2+\int_0^1e^{-\tau\sigma}(\tau^2\sigma^2+\tau(1-2s)^2+2)\,ds$.
Integration by parts, using $\frac d{ds}e^{-\tau\sigma}=-\tau(1-2s)e^{-\tau\sigma}$,
gives $\tau\int_0^1(1-2s)^2e^{-\tau\sigma}ds=2-2\int_0^1e^{-\tau\sigma}ds$,
hence $T(\tau)=\tau^2\int_0^1\sigma^2e^{-\tau\sigma}ds$. Multiplying by
$A$ and $60\pi/t$ and using $\sum|\widehat f_k|^2=\eta^2/2$ gives the
formula. Positivity and monotonicity are immediate, $\int_0^1\sigma^2ds=1/30$,
and near the endpoints $\sigma\sim s$ gives
$\tau^3\int_0^1\sigma^2e^{-\tau\sigma}ds\to4$. The terms of order $\eta^3$
do not contribute to the second variation. The expansion is justified by
Theorem~\ref{thm:neargaussian}(b): for $f=\eta\cos(k\cdot x)$ the plane
kernel is holomorphic in $\eta$ with Gaussian bounds uniform up to the
diagonal, so $\eta\mapsto \operatorname{Tr}_t(\eta\cos(k\cdot x))$ is analytic and its
$\eta^2$ coefficient is the second-order Duhamel term computed above.
\end{proof}
As $t\to0$ the formula reproduces the coefficient $t^2/15$ above. Unlike
$q$, the fixed-time heat trace and (Proposition~\ref{prop:entropyresponse})
entropy respond positively at every scale; the sign change at $\tau_*$
comes from differentiating the entropy in the observation scale
(Figure~\ref{fig:threeresponses}).

\section{Periodization and the compact theorem}\label{sec:compact}\label{subsec:flatlattices}
Periodization identifies endpoints differing by a lattice vector; the
resulting loss and its parameter jets are exponentially small here.

\begin{lemma}[Differentiated compact--plane comparison]\label{lem:compact}
Fix $\theta$ and $0<\gamma<1/11$, and let $\varepsilon\le\varepsilon_0(\theta,\gamma)$,
the amplitude radius of Theorem~\ref{thm:neargaussian}. There are
$t_0,C,c>0$ such that, for $0<t<t_0$, $t\Lambda^2\le\theta$, and
$0\le j\le4$,
\begin{align}
 \sup_{x,|s|\le1}|\partial_s^j(h_s^T-h_s)|
       &\le C\varepsilon^jt^{-1}e^{-c\ell_0^2/t},\label{eq:compacth}\\
 \sup_{|s|\le1}|\partial_s^j(V_t^T-V_t^\infty)|
       &\le CA\varepsilon^jt^{-1}e^{-c\ell_0^2/t}.
                                                        \label{eq:compactV}
\end{align}
For mean-zero $f$ and $j=2,3,4$, the resulting energy bound is
\begin{equation}\label{eq:compactE}
 |\partial_s^j(E_t^T-E_t^\infty)(sf)|
 \le C_\theta t^{-3}e^{-c\ell_0^2/t}\varepsilon^{j-2}N(f).
\end{equation}
\end{lemma}
\begin{proof}
Write $\kappa_\pm=(1\pm\gamma)/4$, so
$\kappa_\sharp=6\kappa_--5\kappa_+=(1-11\gamma)/4>0$.
For $n\in\mathcal L$ put $r_n=r_s(t,x,x+w+n)$ and
$r_{\rm per}=\sum_n r_n$.
Source derivatives are taken at fixed target $y$, before
substituting $y=x+w$. The compact Fisher quotient is $(\nabla_xr_{\rm per})^{\otimes2}/r_{\rm per}$ (the
factor $\rho_s(y)$ cancels), and Gaussian
bounds justify periodizing all parameter and source jets.
A $j$th quotient derivative is a finite sum of terms
$r^{1-\ell}\prod_{i=1}^\ell U_i$, $2\le \ell\le j+2\le6$, where exactly two
$U_i$ carry a source derivative and the parameter orders $\nu_i$
sum to $j$. As jets of $r_{\rm per}$ are sums of jets of $r_n$,
the quotient is a product of averages
$\sum_n\omega_nU_{i,n}/r_n$ with probability weights $\omega_n=r_n/r_{\rm per}$, and
H\"older gives
$|r_{\rm per}\prod_i\sum_n\omega_n U_{i,n}/r_n|
 \le\prod_i(\sum_n r_n|U_{i,n}/r_n|^\ell)^{1/\ell}$.
If $\alpha_i\in\{0,1\}$ records a source derivative, the $i$th
sum is at most
$C\varepsilon^{\ell\nu_i}t^{-1-\ell\alpha_i/2}
      \sum_ne^{-[\ell\kappa_--(\ell-1)\kappa_+]|w+n|^2/t}$.
Thus compact and folded-plane quotient jets are at most
$C\varepsilon^jt^{-2}\sum_ne^{-\kappa_\sharp|w+n|^2/t}$, and
on the centered Voronoi cell outside $|w|<\ell_0/8$,
$|w+n|\ge|w|$, so integration and unfolding give
$C\varepsilon^jt^{-1}e^{-c\ell_0^2/t}$ (before normalization by $2t$).

Inside $|w|<\ell_0/8$, $|w+n|\ge7|n|/8$ for $n\ne0$.
Packing disjoint disks of radius $\ell_0/2$ bounds the lattice sums
uniformly for $t\le\ell_0^2$, so the nonprincipal numerator jets sum to at most
$C\varepsilon^\nu t^{-1-\alpha/2}e^{-49\kappa_-\ell_0^2/(64t)}$, whereas
$r_0\ge ct^{-1}e^{-\kappa_+\ell_0^2/(64t)}$.
Telescoping the numerator factors, using
$|r_{\rm per}^{-k}-r_0^{-k}|\le k(r_{\rm per}-r_0)r_0^{-k-1}$, $k=\ell-1$ (as $r_{\rm per}\ge r_0>0$),
and noting the largest inverse power is $r_0^{-\ell}$, $\ell\le6$, bounds the compact
quotient minus its principal-plane term by
$C\varepsilon^jt^{-2}e^{-(49\kappa_--6\kappa_+)\ell_0^2/(64t)}$.
The remaining nonprincipal folded-plane terms on the inner ball are
at most $C\varepsilon^jt^{-2}\sum_{n\ne0}e^{-\kappa_\sharp|w+n|^2/t}$ by the earlier estimate,
with exponentially small integral. This proves \eqref{eq:compacth}, with constants uniform in the
source; no source-dependent cell was differentiated.
The plane tensor is uniformly positive (Lemma~\ref{lem:regression})
with jets bounded by $C_j\varepsilon^j$ by the same estimates, so after reducing $t_0$ both tensors lie in a
fixed compact set of positive matrices, where the determinant has bounded
derivatives through order five; the chain rule, telescoping and source integration
give \eqref{eq:compactV}.
Multiplying by $9/t^2$ and using $\varepsilon^2\le C_W^2N(f)$ for mean-zero $f$
(Lemma~\ref{lem:interp}(c), cutoff independent) proves \eqref{eq:compactE}; $f=0$ is immediate.
\end{proof}

\begin{theorem}[Uniform nonlinear stability on the full resolved ball]
\label{thm:stability}\label{thm:flatlattices}
Fix $A,\ell_0>0$ and $0<\theta<\tau_*$, where $\tau_*$ is the unique positive response zero
of Proposition~\ref{prop:responseglobal}. There are $\varepsilon_*,t_*,c_*,C_*>0$
such that every real mean-zero trigonometric polynomial with
$\operatorname{supp}\widehat f\subset\{|k|\le\Lambda\}$,
$\|f\|_{\mathcal A}\le\varepsilon_*$, and $t\Lambda^2\le\theta$
satisfies
\begin{align}
 c_*\|\Delta f\|_2^2&\le E_t^\infty(f)\le C_*\|\Delta f\|_2^2,
                                                        \label{eq:planestability}\\
 c_*\|\Delta f\|_2^2&\le E_t^T(f)-E_t^T(0)
                           \le C_*\|\Delta f\|_2^2,
                  \qquad 0<t<t_* .                    \label{eq:torusstability}
\end{align}
All constants are uniform over lattices of covolume $A$ and systole at
least $\ell_0$, independent of the cutoff and number of modes, and
require no lower spectral cutoff; the band
$t\Lambda^2\le4$ is admissible. The constants may deteriorate as
$\theta\uparrow\tau_*$; no uniform positive lower constant extends to
$\theta\ge\tau_*$ by Corollary~\ref{cor:responseobstruction}. The radial third and fourth
relative remainder bounds also hold on the torus. One may take
$c_*=m_\theta/2$ and $C_*=1/2$, where $m_\theta=\min_{[0,\theta]}q>0$.
\end{theorem}
\begin{proof}
At the flat metric the first plane and compact variations vanish for
mean-zero $f$ by source-translation invariance, and $E_t^\infty(0)=0$.
Corollaries~\ref{cor:largerband} and~\ref{cor:planevariations}
and Taylor's formula give
$|E_t^\infty(f)-\tfrac12D^2E_t^\infty(0)[f^2]|\le\tfrac16C_\theta\varepsilon N(f)$.
With $m_\theta$ as in Corollary~\ref{cor:largerband}, choose $\varepsilon_*$ so the right side is at most
$m_\theta N(f)/4$; this gives the plane assertion with constants
$\frac34m_\theta$ and $\frac14+\frac14m_\theta\le\frac12$.
For the torus, $\mathcal D(s)=(E_t^T-E_t^\infty)(sf)$ has
$\mathcal D'(0)=0$ by translation invariance, and
\eqref{eq:compactE} with $j=2$ holds uniformly in $|s|\le1$, so
$|E_t^T(f)-E_t^T(0)-E_t^\infty(f)|=|\mathcal D(1)-\mathcal D(0)|
\le\frac12C_\theta t^{-3}e^{-c\ell_0^2/t}N(f)$. Choose $t_*$ so this is at most
$m_\theta N(f)/4$. Since $m_\theta\le q(0)=\frac14$, the compact
comparison holds with constants $c_*=\frac12m_\theta$ and
$\frac14+\frac12m_\theta\le\frac12=C_*$. The $j=3,4$ transfer bounds and the fact that
both remainders subtract the same $\mathcal C$ prove the last assertion.
\end{proof}

\section{Spectral tails without a hard Fourier cutoff}
\label{sec:soft}
The cutoff above controls the diffusion-scale Lipschitz norm and keeps the
quadratic response positive; we separate these roles.
For smooth, real, mean-zero periodic $f$ (with $\varepsilon=\sA(f)$,
$N(f)=\sH(f)^2$; the zero function is included in every class below), fix
$L_*\in(0,\infty)$ and impose
$\sqrt t\,\|\nabla f\|_\infty\le L_*\varepsilon$ of \eqref{eq:HL}, which
allows unboundedly many frequencies.

\begin{lemma}[Uniform estimates from a diffusion-scale Lipschitz bound]
\label{lem:softjets}
For fixed $L_*$ there is a common small amplitude radius such that
\eqref{eq:HL} implies
\begin{equation}\label{eq:softjets}
 \sup_{|s|\le1}\bigl(|D^jE_t^\infty(sf)[f^j]|
                  +|D^jR_t^\infty(sf)[f^j]|\bigr)
 \le C_{L_*}\varepsilon^{j-2}N(f),\qquad j=3,4.
\end{equation}
The compensated scalar estimate of Theorem~\ref{thm:scalar} holds with
$C_{L_*}$ in place of $C_\theta$. For sufficiently small $t$, the same energy
and remainder estimates hold on the compact torus. All these assertions
apply to smooth periodic functions without a finite Fourier support.
\end{lemma}
\begin{proof}
For $f\ne0$ and $F=f/\varepsilon$, $\|F\|_\infty\le1$,
$\sqrt t\|\nabla F\|_\infty\le L_*$, and
$|\zeta|\sqrt t\|\nabla F\|_\infty\le rL_*$ for $\rho_\zeta=1+\zeta F$, $|\zeta|\le r$. By the paragraph containing
\eqref{eq:HL}, these are the only inputs of the resolution
hypothesis in Sections~\ref{sec:kernel}--\ref{sec:closure}, now with $L_*$ for $\sqrt\theta$:
in Theorem~\ref{thm:neargaussian}(a),(b), $L_\zeta=|\zeta|L_*/(1-r_0)^2$ in \eqref{eq:paramlip}
(the Levi series converges normally for any finite Lipschitz constant; chaining does
not involve it); Section~\ref{sec:compclock} uses
Lemma~\ref{naclock:lem:complexinput}, which follows from these kernel bounds,
and Lemma~\ref{lem:interp}(a); Sections~\ref{naani:sec} and~\ref{sec:secondscore} use the kernel bounds,
interpolation, periodic translation, and $L_\infty\sqrt t\le rL_*$; and
Section~\ref{mom:section} uses the semigroup bounds,
$\Delta P_uh=\rho P_u(\rho^{-1}\Delta h)$ and Lemma~\ref{lem:interp}(a).
Hence Theorems~\ref{thm:scalar} and~\ref{naani:main} and
Corollary~\ref{cor:planevariations} (that is, \eqref{eq:softjets}) hold under
\eqref{eq:HL} with constants $C_{L_*}$; higher norms of the smooth $f$ do not enter.
The compact--plane comparison uses the same Gaussian
jets, giving energy jets at most
$CA\varepsilon^jt^{-3}e^{-c\ell_0^2/t}$; with
$\varepsilon^2\le C_W^2N(f)$ (Lemma~\ref{lem:interp}(c)), for $j=2,3,4$,
\begin{equation}\label{eq:softtransfer}
 \sup_{|s|\le1}|D^j(E_t^T-E_t^\infty)(sf)[f^j]|
 \le C_{L_*}t^{-3}e^{-c\ell_0^2/t}\varepsilon^{j-2}N(f),
\end{equation}
which proves the compact assertion, since energy and geometric remainder
subtract the same curvature functional.
\end{proof}

With $q$ as in \eqref{eq:qdef}, the relative quadratic margin
\begin{equation}\label{eq:spectralmargin}
 \mathcal Q_t(f)=A\sum_{k\ne0}|k|^4q(t|k|^2)|\widehat f_k|^2
\end{equation}
converges absolutely for $f\in H^2$ by Proposition~\ref{prop:responseglobal}.

\begin{theorem}[Nonlinear stability with a spectral energy budget]
\label{thm:softstability}
Fix $L_*<\infty$ and $0<\mu\le\frac14$; for $\mu>\frac14$ only $f=0$
qualifies, since $q\le\frac14$ gives $\mathcal Q_t(f)\le N(f)/4$. There are
$\varepsilon_*,t_*>0$, depending only
on $L_*,\mu$ (and, for the compact statement, on $A,\ell_0$), such that, with $c_*=\mu/2$ and $C_*=1/2$,
every smooth real mean-zero periodic $f$ satisfying
$\|f\|_{\mathcal A}\le\varepsilon_*$,
$\sqrt t\|\nabla f\|_\infty\le L_*\|f\|_{\mathcal A}$ and
$\mathcal Q_t(f)\ge\mu N(f)$ obeys
\begin{align}
 c_*N(f)&\le E_t^\infty(f)\le C_*N(f),\label{eq:softplane}\\
 c_*N(f)&\le E_t^T(f)-E_t^T(0)\le C_*N(f),
                         \qquad 0<t<t_* .\label{eq:softcompact}
\end{align}
In particular, if
$q(\tau)\ge m>0$ on $[0,\Theta]$ and $\delta_*=1/40$, it suffices that
\begin{equation}\label{eq:tailbudget}
 A\sum_{t|k|^2>\Theta}|k|^4|\widehat f_k|^2\le\vartheta N(f),
 \qquad 0\le\vartheta<\frac{m}{m+\delta_*},
\end{equation}
and then one can take $\mu=m(1-\vartheta)-\delta_*\vartheta>0$.
\end{theorem}
\begin{proof}
The identity $\tfrac12D^2E_t^\infty(0)[f,f]=\mathcal Q_t(f)$ extends to smooth $f$ via Fourier partial sums,
which converge in every fixed smooth norm and in $\mathcal A$ with
normalized Lipschitz bounds eventually at most $L_*+1$;
normal convergence of the Levi series and a common Gaussian quotient bound
for jets through order two permit the limit. (Alternatively, the first kernel variation is a linear Fourier series, and
source orthogonality and Parseval give the quadratic form.)
The first variation vanishes on mean-zero functions and
$E_t^\infty(0)=0$, so Lemma~\ref{lem:softjets} and Taylor's formula give
$|E_t^\infty(f)-\mathcal Q_t(f)|\le(C_{L_*}/6)\varepsilon N(f)\le\mu N(f)/4$ for suitable $\varepsilon_*$. Hence
$E_t^\infty(f)\ge\frac34\mu N(f)$, and $q\le1/4$
(Proposition~\ref{prop:responseglobal}) gives
$E_t^\infty(f)\le(\frac14+\frac\mu4)N(f)$.
Integrating the second derivative along $sf$, using
the vanishing first variation, \eqref{eq:softtransfer} at $j=2$ bounds the
difference of flat-subtracted energies by
$C_{L_*}t^{-3}e^{-c\ell_0^2/t}N(f)/2\le\mu N(f)/4$ for suitable $t_*$; the compact difference then lies between $\frac12\mu N(f)$
and $(\frac14+\frac12\mu)N(f)\le\frac38N(f)$, proving both inequalities with
$c_*=\mu/2$ and $C_*=1/2$.
Finally, splitting \eqref{eq:spectralmargin} at $t|k|^2=\Theta$ and using
$q\ge-\delta_*$ (Proposition~\ref{prop:responseglobal}) gives
$\mathcal Q_t(f)\ge[m(1-\vartheta)-\delta_*\vartheta]N(f)$.
\end{proof}

\begin{remark}[An explicit family excluded by every fixed cutoff]
On the standard square torus $\R^2/(2\pi\mathbb Z)^2$, for $f_M=a\cos x_1+(ac/M^2)\cos(Mx_1)$, $M\ge2$,
$\|f_M\|_{\mathcal A}=|a|(1+|c|/M^2)$ and
$\|\nabla f_M\|_\infty\le|a|(1+|c|/M)$, so one $L_*$ serves all
$M$ at bounded times; when only the first mode is resolved, the unresolved
fraction of $N(f_M)$ is $c^2/(1+c^2)$. For $\Theta=1$, $m=1/16$, \eqref{eq:tailbudget} allows $c^2<5/2$, i.e.\ any
$\vartheta<5/7$, with
$\mu=(5-7\vartheta)/80$, so one amplitude interval controls the family
while $tM^2\to\infty$. Proposition~\ref{prop:explicitfour} also permits
$\Theta=4$, $m=3/1000$, giving $\vartheta<3/28$ with $\mu=(3-28\vartheta)/1000$. The two fractions
refer to different resolved bands.
\end{remark}

\begin{corollary}[A sufficient amplitude scale near the sharp boundary]
\label{cor:boundaryamplitude}
There are positive constants $\delta_0,c_1,c_2,c_3,C,t_0$, independent
of the Fourier cutoff, such that the following holds. Let
$0<\delta\le\delta_0$, $\theta=\tau_*-\delta$, and let $f$ be real,
mean zero, and supported in $|k|\le\Lambda$. If
$t\Lambda^2\le\theta$ and $\mathsf A(f)\le c_1\delta$, then
$c_2\delta N(f)\le E_t^\infty(f)\le C N(f)$.
For the compact functional the same bounds hold for
$E_t^T(f)-E_t^T(0)$ provided also
$0<t<t_0$ and $t^{-3}e^{-c_3\ell_0^2/t}\le c_1\delta$,
after reducing $c_1,c_2$ if necessary. The linear amplitude restriction is
a sufficient choice from the relative third-variation estimate; its
optimality is not asserted.
\end{corollary}
\begin{proof}
Take all analytic estimates at the resolution
bound $t\Lambda^2\le\tau_*$, so their amplitude radius and constants are
fixed as $\delta\downarrow0$; positivity is used only in the
quadratic response. Taylor's formula gives
$|E_t^\infty(f)-\mathcal Q_t(f)|\le C_1\mathsf A(f)N(f)$
with $C_1$ independent of $\delta$, and
Proposition~\ref{prop:simplezero} gives
$c_0\delta N(f)\le\mathcal Q_t(f)\le N(f)/4$. Choosing $c_1$ with $C_1c_1\le c_0/4$ and
$c_1\delta_0$ in the analytic amplitude interval
proves the plane assertion. The first variation
of $E_t^T-E_t^\infty$ vanishes at the flat metric, so integrating
\eqref{eq:compactE} with $j=2$ twice bounds the
flat-subtracted compact difference by $C_2t^{-3}e^{-c_3\ell_0^2/t}N(f)$;
reducing $c_1$ so $C_2c_1\le c_0/4$ proves the compact assertion.
All upper bounds are uniform in $0<\delta\le\delta_0$.
\end{proof}

\section{The entropy-production deficit and the role of the area}
\label{sec:entropy}
The determinant in $V_t$ does not enter the quadratic analysis; we
identify the part of $E_t$ that does with the
entropy production of the heat flow.
For positive smooth periodic $\rho$ let
$\operatorname{Ent}_t(\rho)=\int_{\mathbb T^2}\rho(x)\int_{\mathbb R^2}
     r_t(x,y)\log(r_t(x,y)/\rho(y))\,dy\,dx$,
the source-averaged relative entropy of the heat distribution with respect
to $\rho\,dy$ ($\operatorname{Ent}_t^T$: periodic kernel).

\begin{lemma}[A de Bruijn identity]\label{lem:debruijn}
$-\frac d{dt}\operatorname{Ent}_t=\frac1{2t}\int_{\mathbb T^2}\operatorname{tr}h_t\,dx$,
and the same holds on the torus. Consequently, with
$D_t=\int_{\mathbb T^2}(\frac12\operatorname{tr}h_t-\sqrt{\det h_t})\,dx\ge0$,
$A-V_t=(A+t\tfrac d{dt}\operatorname{Ent}_t)+D_t$.
\end{lemma}
\begin{proof}
As in \eqref{eq:VEnt}, which used only
positivity and smoothness of $\rho$, mass conservation removes the
derivative of the logarithm, and integration by parts with weight
$\rho(x)$, after unfolding, gives $-\frac d{dt}\operatorname{Ent}_t=\int_{\mathbb T^2}\int|\nabla_xr_t|^2/r_t
\,dy\,dx=\int_{\mathbb T^2}\operatorname{tr}J_t\,dx$. Also $D_t\ge0$ as
$\frac12\operatorname{tr}h\ge\sqrt{\det h}$ for positive $2\times2$ matrices.
\end{proof}

Define the \emph{entropy-production deficit}
\[
 \mathcal W_t(f)=\frac9{t^2}\Bigl(A-\frac12\int_{\mathbb T^2}\operatorname{tr}h_t\,dx\Bigr)
 =\frac9{t^2}\Bigl(A+t\frac d{dt}\operatorname{Ent}_t\Bigr),
 \qquad E_t=\mathcal W_t+\frac{9D_t}{t^2},
\]
on the plane and torus; it compares averaged heat Fisher information with its flat value.
Put
$\widetilde F(f)=F(f)-F(0)$ (as in Theorem~\ref{thm:radialselection}) and
$\beta_{\mathcal L}(t)=C_{A,\ell_0,\theta}t^{-3}e^{-c\ell_0^2/t}$ on the family of
Theorem~\ref{thm:flatlattices}.

\begin{theorem}[The results hold for entropy production]
\label{thm:entropy}
Let $f$ be real and mean zero. Say that $f$ satisfies the \emph{resolved
hypotheses} if it satisfies those of Theorem~\ref{thm:flatlattices}
(with $\varepsilon=\mathsf A_{\mathcal L}(f)$), and the
\emph{soft hypotheses} if it satisfies those of
Theorem~\ref{thm:softstability} . Under either, \eqref{eq:EminusW}--\eqref{eq:EminusWT} below hold, with constants
depending on $\theta$, resp.\ $L_*$, and on $A,\ell_0$.
On the plane the second variations of $\mathcal W_t^\infty$ and
$E_t^\infty$ at the flat metric coincide, with response $q$ and threshold
$\tau_*$, for every mean-zero trigonometric $f$. Moreover,
\begin{equation}\label{eq:EminusW}
 0\le E_t^\infty(f)-\mathcal W_t^\infty(f)
       \le C_\theta\varepsilon^2N_{\mathcal L}(f),
\end{equation}
and the relative third and fourth variation bounds of
Corollary~\ref{cor:planevariations} hold for $\mathcal W_t^\infty$.
On every compact flat torus of the family, for
$0<t<t_0(A,\ell_0,\theta)$,
\begin{align}
 \left|D^2(E_t^{T_{\mathcal L}}-\mathcal W_t^{T_{\mathcal L}})(0)[f,f]\right|
       &\le\beta_{\mathcal L}(t)N_{\mathcal L}(f),\label{eq:EminusWtorus}\\
 -\beta_{\mathcal L}(t)N_{\mathcal L}(f)
 \le\widetilde E_t^{T_{\mathcal L}}(f)
       -\widetilde{\mathcal W}_t^{T_{\mathcal L}}(f)
 &\le\bigl(C_\theta\varepsilon^2+\beta_{\mathcal L}(t)\bigr)N_{\mathcal L}(f).\label{eq:EminusWT}
\end{align}
Always $E_t^{T_{\mathcal L}}-\mathcal W_t^{T_{\mathcal L}}=9D_t^{T_{\mathcal L}}/t^2\ge0$. If the flat reference Fisher tensor is isotropic, as for
the square and hexagonal lattices (but not, for example, a rectangular one), then additionally
$D^2(E_t^{T_{\mathcal L}}-\mathcal W_t^{T_{\mathcal L}})(0)[f,f]\ge0$ and
$\widetilde E_t^{T_{\mathcal L}}(f)-\widetilde{\mathcal W}_t^{T_{\mathcal L}}(f)\ge0$,
with the same displayed upper bounds.
Consequently the stability, spectral-margin, radial, mixed-Hessian,
and finite-mode-flow results hold with $E_t$ replaced by $\mathcal W_t$,
with the same response $q$ and the same sharp uniform-band threshold
$\tau_*$, compact value comparisons using the
flat-subtracted energies; constants are uniform on the fixed-area
noncollapsing families of Theorem~\ref{thm:flatlattices}.
\end{theorem}
\begin{proof}
On the plane, along $\rho_s=1+sf$, the efficient residual covariance
$Q$ is quadratic in a residual vanishing at $s=0$, so
$h^{\rm tf}=2tQ^{\rm tf}=O(s^2)$ and
$D_t=\int|h^{\rm tf}|^2/(\operatorname{tr}h+2\sqrt{\det h})=O(s^4)$;
the first three parameter derivatives of $D_t^\infty(sf)$ vanish at zero.
The fourth-variation estimate of Theorem~\ref{naani:main}, Taylor's
formula, and interpolation give
$0\le D_t^\infty(f)/t^2
 \le\frac{C_\theta}{24}\|\nabla f\|_4^4
 \le\frac{10C_\theta}{24}\varepsilon^2N_{\mathcal L}(f)$,
and the same theorem supplies the third and fourth relative variations
(under the soft hypotheses via its $L_*$-version from
Lemma~\ref{lem:softjets}, with $C_{L_*}$ for $C_\theta$).
Subtracting them from those of $E_t^\infty$ proves the plane assertions.

On a compact flat torus, $h_t^{T_{\mathcal L}}(0)$ is constant
but possibly anisotropic, so $D_t^{T_{\mathcal L}}(0)$ need not
vanish. The trace and determinant parts of the compact comparison give
\begin{equation}\label{eq:Dlatticejets}
 \sup_{|s|\le1}\left|\partial_s^j
       (D_t^{T_{\mathcal L}}-D_t^\infty)(sf)\right|
 \le CA\varepsilon^jt^{-1}e^{-c\ell_0^2/t},\qquad 0\le j\le4.
\end{equation}
Both first variations vanish on mean-zero
functions by translation invariance, without isotropy, so
integrating \eqref{eq:Dlatticejets} at $j=2$ twice yields
$\widetilde D_t^{T_{\mathcal L}}(f)
       =D_t^\infty(f)+O(A\varepsilon^2t^{-1}e^{-c\ell_0^2/t})$.
Multiplying by $9/t^2$ and using Lemma~\ref{lem:latticeWiener} proves
\eqref{eq:EminusWT}; the $j=2$ estimate at zero with
$D^2D_t^\infty(0)=0$ proves \eqref{eq:EminusWtorus}.
If the flat compact tensor is isotropic, $D_t^{T_{\mathcal L}}(0)=0$ and
nonnegativity gives the one-sided statements.
The absolute exponential error suffices in the compact stability and
sharpness proofs after shrinking the small-time interval;
the radial proofs use only the relative third variation
and flat stationarity; the mixed proof estimates the
scalar and anisotropy terms separately, so applies to
$\mathcal W_t$, with compact transfer via the mixed
versions of \eqref{eq:Dlatticejets}; the finite-mode flow follows from
the Hessian bounds; and the soft argument is
unchanged. None needs a sign for the compact
Hessian difference in \eqref{eq:EminusWtorus}.
\end{proof}

\paragraph{What the area adds.}
The area deficit adds to entropy production the nonnegative term $9D_t/t^2$, quartic on the plane.
We keep $V_t$ primary because it needs no reference metric, unlike the trace;
because $D_t$ vanishes exactly when $h_t$ is pointwise conformal to $g$;
and because the distance of $h_t$ from a multiple of $g$ involves $V_t$ and $D_t$ jointly.
Conversely, the theorems control $\mathcal W_t$ uniformly at fixed time, complementing
the small-time entropy asymptotics of \cite{MPR}.
\section{Curvature energy and statistical uniformization}
\label{sec:radialconsequences}
The third-variation estimate controls the Hessian in the base direction,
which gives more than value coercivity (but no mixed-direction positivity).

\begin{proposition}[Comparison with exact curvature energy]
\label{prop:curvatureequivalence}
Let $f$ be real and smooth, and put $\delta=\|f\|_\infty\le1/5$.
For $g=(1+f)g_0$ on any flat two-torus, with $N(f)=\|\Delta f\|_2^2$,
\[
 \frac{(1-(1+\sqrt{10})\delta)^2}{4(1-\delta)^2(1+\delta)^3}N(f)
 \le\mathcal C(f)\le\frac{(1+(\sqrt{10}-1)\delta)^2}{4(1-\delta)^5}N(f).
\]
In particular, $\delta\le1/8$ implies
$N(f)/32\le\mathcal C(f)\le N(f)$. Thus the plane and compact comparisons
of Theorem~\ref{thm:stability} are also comparisons with the exact
curvature energy, with cutoff-independent constants.
\end{proposition}
\begin{proof}
With $\rho=1+f$ and $b=|\nabla f|^2$ the conformal curvature formula gives
$4\mathcal C(f)=\|\rho^{-3/2}(\Delta f-\rho^{-1}b)\|_2^2$, and
\eqref{eq:GNbasic} gives $\|b\|_2\le\sqrt{10}\,\delta\|\Delta f\|_2$. By the
triangle and reverse triangle inequalities,
\[
 (1+\delta)^{-3/2}\Bigl(1-\frac{\sqrt{10}\,\delta}{1-\delta}\Bigr)\|\Delta f\|_2
 \le2\sqrt{\mathcal C(f)}\le(1-\delta)^{-3/2}\Bigl(1+\frac{\sqrt{10}\,\delta}{1-\delta}\Bigr)\|\Delta f\|_2,
\]
and the lower coefficient is positive for $\delta<1/(1+\sqrt{10})$. Squaring
proves the assertion; the coarse constants follow for $\delta\le1/8$.
\end{proof}

\begin{corollary}[Quantitative closeness to the curvature energy]
\label{cor:Eclose}
Under the hypotheses of Theorem~\ref{thm:stability},
$|E_t^\infty(f)-\mathcal C(f)|\le C_\theta(t\Lambda^2+\varepsilon)N(f)$,
and the same bound holds for the entropy-production deficit
$\mathcal W_t^\infty$ of Section~\ref{sec:entropy}. In particular
$E_t^\infty\to\mathcal C$ relative to $N(f)$ as $t\Lambda^2\to0$ and
$\varepsilon\to0$, uniformly in the number of modes. On the fixed-area noncollapsing flat
families of Theorem~\ref{thm:flatlattices}, the compact version is
\begin{equation}\label{eq:compactCalabi}
 |E_t^{T_{\mathcal L}}(f)-E_t^{T_{\mathcal L}}(0)-\mathcal C(f)|
 \le C_{A,\ell_0,\theta}
   \bigl(t\Lambda^2+\varepsilon+t^{-3}e^{-c\ell_0^2/t}\bigr)N_{\mathcal L}(f),
\end{equation}
and the same holds for the flat-subtracted $\mathcal W_t^{T_{\mathcal L}}$.
(Convergence at a fixed metric is Theorem~\ref{thm:fixedlimit} below.)
\end{corollary}
\begin{proof}
Put $R=E_t^\infty-\mathcal C$. Both functionals vanish with their first
variations at $f=0$, and by Proposition~\ref{prop:response} and
Remark~\ref{rem:normalization},
$\frac12D^2R(0)[f,f]=A\sum_{k\ne0}|k|^4(q(t|k|^2)-\frac14)|\widehat f_k|^2$.
Since $q$ is smooth on $[0,\theta]$ with $q(0)=\frac14$,
$|q(\tau)-\frac14|\le C_\theta\tau\le C_\theta t\Lambda^2$ there. Taylor's
formula with the third-variation bound for $R$ in
Corollary~\ref{cor:planevariations} adds at most $C_\theta\varepsilon N(f)$;
for $\mathcal W_t^\infty$ use also \eqref{eq:EminusW}.
Integrating the $j=2$ estimate \eqref{eq:compactE} twice and using
the vanishing first variation proves \eqref{eq:compactCalabi}; the entropy
version follows from Theorem~\ref{thm:entropy}.
\end{proof}

\subsection{Fixed metrics: the small-time limit}\label{subsec:fixedlimit}
For a single fixed metric the amplitude error of Corollary~\ref{cor:Eclose}
disappears as $t\to0$. Let $(\mathcal M,g)$ be a closed smooth surface of area $A$ and Euler
characteristic $\chi$, and let $p_t(x,y)$ be its heat kernel
($\partial_t=\Delta_g$, density with respect to $dA$). For $x\in\mathcal M$ and
$v\in T_x\mathcal M$ put
$J_t(x)(v,v)=\int_{\mathcal M}(\partial_v\log p_t(x,y))^2p_t(x,y)\,dA(y)$,
$h_t=2tJ_t$, and define $V_t=\int_{\mathcal M}\sqrt{\det\nolimits_gh_t}\,dA$,
$\operatorname{Ent}_t=\int_{\mathcal M}\int_{\mathcal M}p_t\log p_t\,dA\,dA$,
$D_t=\int_{\mathcal M}(\frac12\operatorname{tr}_gh_t-\sqrt{\det\nolimits_gh_t})\,dA$,
$E_t(g)=9t^{-2}(A-V_t)$ and
$\mathcal W_t(g)=9t^{-2}(A-\frac12\int_{\mathcal M}\operatorname{tr}_gh_t\,dA)$.
For $g=(1+f)g_0$ on a flat torus these are the functionals used above
($r_t(x,y)=p_t(x,y)\rho(y)$, the Fisher form does not depend on the reference
measure in $y$, and $\sqrt{\det h}\,dx=\sqrt{\det_gh}\,dA$,
$\operatorname{tr}h\,dx=\operatorname{tr}_gh\,dA$). The proof
of Lemma~\ref{lem:debruijn} works on every closed surface (mass
conservation and integration by parts in the target give
$\frac d{dt}\int p_t\log p_t\,dA(y)=-\int|\nabla_yp_t|^2/p_t\,dA(y)$, and
$p_t(x,y)=p_t(y,x)$ turns the integral over $x$ into the source
Fisher information), so
\begin{equation}\label{eq:debruijnM}
 -\frac d{dt}\operatorname{Ent}_t=\frac1{2t}\int_{\mathcal M}\operatorname{tr}_gh_t\,dA,
 \qquad E_t=\mathcal W_t+\frac{9D_t}{t^2},\qquad D_t\ge0 .
\end{equation}

\begin{lemma}[Small-time expansion of the heat-Fisher tensor]
\label{lem:fisherexpansion}
There are smooth symmetric $2$-tensors $H_1,H_2$ on $\mathcal M$ and $C<\infty$ with
$\sup_{x\in\mathcal M}|h_t-g-tH_1-t^2H_2|_g\le Ct^3$ for $0<t\le1$. Moreover $H_1=c_{\rm H}Kg$ for a universal constant $c_{\rm H}$, and $\operatorname{tr}_gH_j$
is a universal polynomial in the curvature and its covariant derivatives,
of weight $2j$.
\end{lemma}
\begin{proof}
Fix $\delta$ below a quarter of the injectivity radius and $m=30$. For
$t\in[t_0,1]$ with a fixed small $t_0$ the claim is trivial, because $h_t$
is bounded there by the gradient bound below; so let $0<t\le t_0$, with
$t_0$ so small that $mt_0\log(1/t_0)\le\delta^2$. For $x\in\mathcal M$ use
normal coordinates $y=\exp_x(v)$, $|v|=d(x,y)$, $dA(y)=\theta_x(v)\,dv$, with
$\theta_x$ smooth and $\theta_x(0)=1$, and put
$\Omega_t(x)=\{y:d(x,y)^2\le mt\log(1/t)\}\subset B_\delta(x)$.

\emph{Step 1: localization.} On a compact surface
$p_t(x,y)\le Ct^{-1}e^{-d(x,y)^2/(5t)}$ and
$|\nabla_x\log p_t(x,y)|\le C(d(x,y)/t+t^{-1/2})$ for $0<t\le1$
\cite{Davies,Hsu}. Off $\Omega_t(x)$ one has
$e^{-d^2/(5t)}\le t^{m/10}e^{-d^2/(10t)}$, hence
\[
 \int_{\mathcal M\setminus\Omega_t(x)}|\nabla_x\log p_t|^2p_t\,dA
 \le Ct^{m/10}\int_{\mathcal M}\Bigl(\frac{d^2}{t^2}+\frac1t\Bigr)
       \frac{e^{-d^2/(10t)}}t\,dA\le Ct^{m/10-1}.
\]
After multiplication by $2t$ this part of $h_t(x)$ is $O(t^{m/10})=O(t^3)$,
uniformly in $x$.

\emph{Step 2: parametrix.} By \cite[Theorem~2.30]{BGV},
$p_t=k_t+R_N$ with $k_t(x,y)=(4\pi t)^{-1}e^{-d(x,y)^2/(4t)}\Phi(t,x,y)$,
$\Phi=\sum_{j=0}^Nt^j\Phi_j$, where the $\Phi_j$ are smooth on
$\{d<2\delta\}$, $\Phi_0(x,x)=1$, and $\|R_N\|_{C^1(\mathcal M\times\mathcal M)}\le
C_Nt^{N-3/2}$ (BGV's cutoff equals $1$ on $\{d\le\delta\}$). Shrinking $\delta$
and $t_0$, $\frac12\le\Phi\le2$ and $|\nabla_x\log\Phi|\le C$ on
$\{d\le\delta\}$. On $\Omega_t(x)$ therefore $k_t\ge ct^{m/4-1}$ and $|\nabla_x\log k_t|\le C(d/t+1)\le
Ct^{-1/2}\log(1/t)^{1/2}$. Take $N=14$, so that $t^{N-3/2}\le t^6\,t^{m/4-1}$.
Then $|R_N|+|\nabla_xR_N|\le Ct^5k_t$ on $\Omega_t(x)$, so $p_t=k_t(1+O(t^5))$ and
$\nabla_x\log p_t=\nabla_x\log k_t+O(t^4)$ there. Since
$\int_{\Omega_t(x)}k_t\,dA\le C$ and $\int|\nabla_x\log k_t|^ik_t\,dA\le Ct^{-i/2}$ for
$i=1,2$, we obtain, uniformly in $x$,
\[
 h_t(x)=2t\int_{\Omega_t(x)}(\nabla_x\log k_t\otimes\nabla_x\log k_t)\,
          k_t\,dA(y)+O(t^3).
\]

\emph{Step 3: Laplace's method.} By the Gauss lemma
$\nabla_xd(x,y)^2=-2v$ exactly, so
$\nabla_x\log k_t=v/(2t)+B(t,x,v)$ with $B=\nabla_x\log\Phi(t,x,\exp_xv)$
smooth in $(t,x,v)$, using a local orthonormal frame to identify $T_x\mathcal M$
with $\mathbb R^2$ (the $H_j$ below are frame independent by $O(2)$-equivariance). Substitute $v=\sqrt t\,w$. Since
$(4\pi t)^{-1}e^{-|v|^2/(4t)}dv=\gamma(w)\,dw$ with
$\gamma(w)=(4\pi)^{-1}e^{-|w|^2/4}$,
\[
 h_t(x)=\int_{|w|^2\le m\log(1/t)}\Bigl(\tfrac12w\otimes w
   +\sqrt t\,(w\otimes B+B\otimes w)+2t\,B\otimes B\Bigr)
   \Phi\,\theta_x(\sqrt tw)\,\gamma(w)\,dw+O(t^3),
\]
with $B$ and $\Phi$ evaluated at $(t,x,\exp_x(\sqrt tw))$. Expand
$B$, $\Phi$ and $\theta_x$ in Taylor series in $(t,v)$ at $(0,0)$ to total
order five. The coefficients are smooth in $x$, and by compactness the
Taylor remainder is at most $C(t+|v|)^6=Ct^3(\sqrt t+|w|)^6$ on
$[0,t_0]\times\mathcal M\times\{|v|\le\delta\}$, so the total remainder of the
integrand is at most $Ct^3(1+|w|)^{20}$, uniformly in $x$, which is
$\gamma$-integrable. The
expanded integrand is a finite sum of terms $t^{a/2}P_a(x,w)$, where
$P_a$ is a polynomial in $w$ with smooth coefficients. A monomial
$t^\ell v^\alpha$ becomes $t^{\ell+|\alpha|/2}w^\alpha$, whose $t^{1/2}$-exponent
$2\ell+|\alpha|$ has the parity of its $w$-degree $|\alpha|$; the three
prefactors have $t^{1/2}$-exponent and $w$-degree $(0,2)$, $(1,1)$ and $(2,0)$.
Hence $P_a$ has the parity of $a$. Extending the integrals to
$\mathbb R^2$ costs $O(t^{m/4-1})=O(t^3)$, and odd polynomials integrate to
zero against $\gamma$. Therefore
\[
 h_t(x)=H_0(x)+tH_1(x)+t^2H_2(x)+O(t^3),\qquad
 H_j(x)=\int_{\mathbb R^2}P_{2j}(x,w)\,\gamma(w)\,dw,
\]
uniformly in $x$, with $H_j$ smooth. The leading term is
$H_0=\frac12\int w\otimes w\,\gamma\,dw=g$, since $\int w\otimes w\,\gamma=2I$;
this is the flat computation of Section~\ref{sec:framework}.

\emph{Step 4: the coefficients.} The Taylor coefficients of $d^2$, of
$\theta_x$ and of the $\Phi_j$ in normal coordinates are universal
polynomials in the curvature and its covariant derivatives
\cite{BGV,Gilkey}; since $\operatorname{tr}_gH_j$ is $O(2)$-invariant, Weyl's
invariant theory \cite{Gilkey} makes it a contraction polynomial in the
covariant derivatives of the curvature. Moreover $H_j(x)$ is a universal polynomial in the Taylor
coefficients of $g$ at $x$ in normal coordinates, in which $\partial^kg$
has weight $k$ and the first-order coefficients vanish. Under
$g\mapsto\lambda^2g$, $t\mapsto\lambda^2t$ the heat kernel scales by
$\lambda^{-2}$ and $h\mapsto\lambda^2h$, so the coefficient of $t^j$ has
weight $2j$. Thus $H_1(x)$ depends only on the $2$-jet of the metric at
$x$ in normal coordinates, which on a surface is
$g_{ij}(v)=\delta_{ij}-\frac{K(x)}3(|v|^2\delta_{ij}-v_iv_j)+O(|v|^3)$,
rotation invariant and linear in $K(x)$. Normal coordinates are defined up
to $O(2)$, so $H_1(x)$ is $O(2)$-equivariant, and a rotation-invariant
symmetric $2\times2$ form is a multiple of the identity. Hence
$H_1=c_{\rm H}Kg$. The proof of Theorem~\ref{thm:fixedlimit} gives
$c_{\rm H}\int K\,dA=-\int K\,dA$; since $c_{\rm H}$ is universal, a round sphere
shows $c_{\rm H}=-1$.
\end{proof}

\begin{theorem}[The fixed-metric limit]\label{thm:fixedlimit}
For every smooth metric $g$ on a closed surface,
\[
 \mathcal W_t(g)=\frac{18\pi\chi}t+\int_{\mathcal M}K_g^2\,dA_g+O(t),\qquad
 E_t(g)=\mathcal W_t(g)+O(t^2)\qquad(t\to0).
\]
In particular, on a torus $E_t(g)\to\int K_g^2\,dA_g=\operatorname{Ca}(g)$ and
$\mathcal W_t(g)\to\operatorname{Ca}(g)$ for every fixed smooth metric, with
error $O(t)$.
\end{theorem}
\begin{proof}
Put $I(t)=-\frac d{dt}\operatorname{Ent}_t$. By \eqref{eq:debruijnM} and
Lemma~\ref{lem:fisherexpansion},
$tI(t)=\frac12\int\operatorname{tr}_gh_t\,dA=A+b_1t+b_2t^2+O(t^3)$ with
$b_j=\frac12\int\operatorname{tr}_gH_j\,dA$. Integrating $I$ from $t$ to $1$,
\begin{equation}\label{eq:Entexp1}
 \operatorname{Ent}_t=-A\log t+c_{\rm Ent}-b_1t-\tfrac{b_2}2t^2+O(t^3)
\end{equation}
for a constant $c_{\rm Ent}$. On the other hand, Menkovski, Portegies and
Ravelonanosy \cite[Theorems~3.1, 4.4 and~4.5]{MPR} expand the relative
entropy of the heat kernel at each point; their generator is
$\frac12\Delta$, so their time is $s=2t$, and their reference measure is
normalized. In dimension two their expansion reads
$\int_{\mathcal M}p_t(z,y)\log p_t(z,y)\,dA(y)=-\log(4\pi t)-1+2\mathsf c_1(z)t+4\mathsf c_2(z)t^2+o(t^2)$
for each $z$, with $\mathsf c_1=\frac14\operatorname{Sc}$ and (MPR's formula at dimension two)
\[
 \mathsf c_2=\tfrac{16}{480}\operatorname{Sc}_{;ii}-\tfrac1{48}\operatorname{Ric}_{ij}\operatorname{Ric}_{ij}
 +\tfrac{8}{80}\operatorname{Ric}_{ij;ij}
 +\tfrac{12}{1440}R_{iujv}R_{iujv}+\tfrac{8}{720}R_{iujv}R_{juiv}.
\]
This expansion is pointwise in $z$. The
localization, parametrix and Laplace argument of the proof of
Lemma~\ref{lem:fisherexpansion}, applied to $\int p_t(z,y)\log p_t(z,y)\,dA(y)$
(off $\Omega_t(z)$ use $p_t|\log p_t|\le C(p_t\log(1/t)+\sqrt{p_t})$, which
follows from $p_t\le Ct^{-1}$, together with the Gaussian upper bound and a
larger fixed $m$), gives
an expansion $-\log(4\pi t)-1+e_1(z)t+e_2(z)t^2+O(t^3)$ uniform in
$z$, with $e_1,e_2$ continuous; uniqueness of asymptotic expansions at
each $z$ gives $e_1=2\mathsf c_1$ and $e_2=4\mathsf c_2$. Thus \cite{MPR} is used only
for the values of the coefficients, and integration over $z$ gives
\begin{equation}\label{eq:Entexp2}
 \operatorname{Ent}_t=-A\log(4\pi t)-A+2\mathsf C_1t+4\mathsf C_2t^2+o(t^2),\qquad
 \mathsf C_j=\int_{\mathcal M}\mathsf c_j\,dA .
\end{equation}
Comparing the coefficients of $t$ and $t^2$ in \eqref{eq:Entexp1} and
\eqref{eq:Entexp2} gives $b_1=-2\mathsf C_1$ and $b_2=-8\mathsf C_2$.

On a surface $R_{ijkl}=K(g_{ik}g_{jl}-g_{il}g_{jk})$, so $\operatorname{Ric}=Kg$,
$\operatorname{Sc}=2K$, $\operatorname{Ric}_{ij}\operatorname{Ric}_{ij}=2K^2$,
$R_{iujv}R_{iujv}=4K^2$ and $R_{iujv}R_{juiv}=K^2(4-2-2+2)=2K^2$. The terms
$\operatorname{Sc}_{;ii}$ and $\operatorname{Ric}_{ij;ij}=\frac12\Delta\operatorname{Sc}$
integrate to zero. Hence, by Gauss--Bonnet,
$\mathsf C_1=\frac12\int_{\mathcal M}K\,dA=\pi\chi$ and
\[
 \mathsf C_2=\Bigl(-\frac2{48}+\frac{12\cdot4}{1440}+\frac{8\cdot2}{720}\Bigr)\int_{\mathcal M}K^2\,dA
 =\Bigl(-\frac1{24}+\frac1{30}+\frac1{45}\Bigr)\int_{\mathcal M}K^2\,dA
 =\frac1{72}\int_{\mathcal M}K^2\,dA .
\]
Therefore
$A-tI(t)=-b_1t-b_2t^2+O(t^3)=2\pi\chi\,t+\frac{t^2}9\int_{\mathcal M}K^2\,dA+O(t^3)$, and
$\mathcal W_t=9(A-tI(t))/t^2$ has the stated expansion.

For $E_t$, Lemma~\ref{lem:fisherexpansion} gives $h_t=(1+c_{\rm H}tK)g+O(t^2)$, so
$|h_t^{\rm tf}|_g\le Ct^2$ and $h_t\ge\frac12g$
for small $t$. Since
$\frac12\operatorname{tr}h-\sqrt{\det h}=|h^{\rm tf}|^2/(\operatorname{tr}h+2\sqrt{\det h})$,
we get $0\le D_t\le CAt^4$, and \eqref{eq:debruijnM} gives
$E_t-\mathcal W_t=9D_t/t^2=O(t^2)$.
\end{proof}

\begin{remark}[Two independent determinations of the normalization]
\label{rem:twonormalizations}
The constant $1/72$ above comes from the entropy coefficients of
\cite{MPR}; the factor $9/t^2$ was fixed independently in
Remark~\ref{rem:normalization} by the flat quadratic response
$q(0)=\frac14$ (Proposition~\ref{prop:response}). They agree:
$9/t^2$ times $-b_2=8\mathsf C_2=\frac8{72}\int K^2\,dA$ gives
coefficient $9\cdot8\cdot\frac1{72}=1$ on $\int K^2\,dA$, and
$\int K_{g_f}^2\,dA_{g_f}=\frac14\|\Delta f\|_2^2+O(f^3)$ has the flat
quadratic multiplier $|k|^4/4=|k|^4q(0)$. This consistency check is not an input (it presumes that $t\to0$ commutes with the
second variation). A numerical evaluation on the unit sphere ($K=1$)
independently reproduces the per-point coefficients
$\frac12\operatorname{tr}h_t=1-tK-\frac{t^2}9K^2+O(t^3)$. Theorem~\ref{thm:fixedlimit} also identifies the
coefficients of Remark~\ref{rem:normalization}:
$V_t=A-t\int K\,dA-\frac{t^2}9\int K^2\,dA+O(t^3)$, so $a_1=-1$ and
$a_2=-\frac19$. With Corollary~\ref{cor:Eclose}, $E_t$
converges to the Calabi energy for each fixed metric, uniformly relative
to $N(f)$ on resolved bands of small amplitude.
\end{remark}

\subsection{Uniform statistical uniformization on bounded geometry classes}
\label{subsec:uniformstatistical}
The fixed-metric limit becomes a variational statement only after its
remainder is uniform on the admissible metrics.  The following formulation
makes the compactness requirement explicit.  Fix a closed connected oriented smooth surface
$M$ and a smooth reference metric $g_*$, and let $\mathscr G$ be a family
of smooth metrics satisfying
\begin{equation}\label{eq:boundedgeometryclass}
 L^{-1}g_*\le g\le Lg_*,\qquad
 \|g\|_{C^{80}(g_*)}\le L,\qquad
 \operatorname{inj}(M,g)\ge\iota>0.
\end{equation}
Here $C^{80}$ can be replaced by a sufficiently large finite regularity
order; the displayed order is a convenient sufficient choice, and is not
asserted to be optimal.  Dependence on the fixed reference atlas and metric
is understood.  Write
\[
 \widehat E_t(g)=E_t(g)-\frac{18\pi\chi(M)}t,
 \qquad \mathcal K_2(g)=\int_M K_g^2\,dA_g.
\]

\begin{theorem}[Uniform Calabi limit]\label{thm:uniformCalabi}
There are $C,t_0>0$, depending only on $M,g_*,L,\iota$, such that
\begin{equation}\label{eq:uniformCalabilimit}
 \sup_{g\in\mathscr G}
 \left|\widehat E_t(g)-\mathcal K_2(g)\right|\le Ct
 \qquad(0<t\le t_0).
\end{equation}
The same conclusion holds with $\mathcal W_t$ in place of $E_t$; moreover
$0\le E_t(g)-\mathcal W_t(g)\le Ct^2$ uniformly on $\mathscr G$.
\end{theorem}
\begin{proof}
We track the constants in Lemma~\ref{lem:fisherexpansion}.  The bounds
\eqref{eq:boundedgeometryclass} give a common normal-coordinate radius
$\delta<\iota/4$, uniform ellipticity, and uniform bounds for the metric,
volume density, and their derivatives on these coordinate balls.  The
transport recursion for the coefficients $\Phi_j$ and the cutoff
parametrix therefore have uniform bounds through order $N=14$, including
the one source derivative and the six Taylor derivatives used in that
lemma.  Indeed each transport step applies a differential operator of
order two, while the final differentiated defect and Taylor estimates
require only finitely many further derivatives; fewer than $80$
coordinate derivatives suffice.  The constants in the Gaussian upper
bound and the logarithmic gradient bound are uniform as well: their
constants depend only on lower bounds for the Ricci curvature and the
injectivity radius, upper bounds for finitely many covariant derivatives of
the curvature, the area and the diameter
\cite{LiYau,StroockTuretsky,Hsu}; their local estimates use these same ellipticity and coefficient bounds, and
the fixed finite atlas controls the global patching constants.  Thus the
choice of $m=30$, the localization region
$d_g(x,y)^2\le30t\log(1/t)$, and a single $t_0$ are valid for the entire
family.  The localized Gaussian moments have uniform integrable
majorants.  This proves the strengthened conclusion
\begin{equation}\label{eq:uniformtensor}
 \sup_{g\in\mathscr G}\sup_{x\in M}
 \bigl|h_t-g+tK_gg-t^2H_2(g)\bigr|_g\le Ct^3,
 \qquad \sup_{g\in\mathscr G}\|H_2(g)\|_\infty\le C.
\end{equation}
The coefficient identities in Theorem~\ref{thm:fixedlimit} are exact for
each metric, so, without differentiating an asymptotic remainder,
\[
 \frac12\int_M\operatorname{tr}_g h_t\,dA_g
 =A_g-2\pi\chi(M)t-\frac{t^2}{9}\mathcal K_2(g)+O(t^3).
\]
The remainder is now uniform because \eqref{eq:uniformtensor} is uniform
and $A_g\le L A_{g_*}$.  Multiplication by $9/t^2$ proves the assertion
for $\mathcal W_t$.  Finally \eqref{eq:uniformtensor} gives
$\|h_t^{\rm tf}\|_\infty\le Ct^2$ and $h_t\ge g/2$ uniformly.  The exact
trace--determinant identity in Theorem~\ref{thm:fixedlimit} yields
$D_t\le Ct^4$, proving both remaining assertions.
\end{proof}

For clarity, the geometric input and its role can be recorded as follows:
\begin{center}
\begin{tabular}{p{.34\textwidth}p{.55\textwidth}}
\hline
Uniform input & Use in the proof\\ \hline
Fixed $M,g_*$ and metric ellipticity & A finite common atlas, volume and
 global patching bounds\\
Injectivity radius $\iota$ & A common normal-coordinate and cutoff radius\\
Finite high derivative bound & Transport coefficients, differentiated
 parametrix remainder and Taylor remainder\\
Two-dimensionality & $H_1=-Kg$, hence the trace-free term starts at $t^2$\\
Fixed area and conformal class & Identification of the minimizing
 constant-curvature metric\\
Precompactness of admissible metrics & Passage from a curvature-defect
 estimate to convergence of metrics\\ \hline
\end{tabular}
\end{center}

\begin{corollary}[Approximate minimizers recover uniformization]
\label{cor:statisticaluniformization}
Assume that every $g\in\mathscr G$ is conformal to one fixed smooth
metric and has the same area $A$.  Suppose that $\mathscr G$ contains an
area-$A$ constant-curvature metric $g_c$.  If $g_t\in\mathscr G$ satisfies
\[
 \widehat E_t(g_t)\le\inf_{g\in\mathscr G}\widehat E_t(g)+\epsilon_t,
 \qquad \epsilon_t\ge0,\qquad\epsilon_t\longrightarrow0,
\]
then, with $\bar K=2\pi\chi(M)/A$,
\begin{equation}\label{eq:uniformizationdefect}
 \int_M(K_{g_t}-\bar K)^2\,dA_{g_t}\le\epsilon_t+2Ct.
\end{equation}
Every $C^2$ cluster point is a constant-curvature metric of area $A$ in
the prescribed conformal class.  Such cluster points exist under
\eqref{eq:boundedgeometryclass}.  If $\chi(M)\le0$, then $g_t\to g_c$ in
$C^2$.  If $M$ is a sphere, the conclusion is convergence in $C^2$ to the
set of constant-curvature metrics; convergence to one representative
requires a normalization of the conformal automorphism freedom.
\end{corollary}
\begin{proof}
Gauss--Bonnet and the area constraint give the exact identity
\[
 \mathcal K_2(g)-\frac{(2\pi\chi(M))^2}{A}
 =\int_M(K_g-\bar K)^2\,dA_g\ge0.
\]
Apply \eqref{eq:uniformCalabilimit} twice, using $g_c$ as a competitor,
to obtain \eqref{eq:uniformizationdefect}.  The common ellipticity and
$C^{80}$ bounds make the family precompact in $C^2$ by
Arzel\`a--Ascoli.  Area, conformality and curvature pass to a $C^2$
limit, which therefore has $K\equiv\bar K$.  For $\chi\le0$ the
area-normalized constant-curvature representative in a fixed conformal
class is unique.  Compactness then gives convergence of the full
sequence: any sequence staying a positive $C^2$ distance away would
have a cluster point with that same positive distance, a contradiction.
On the sphere the same argument gives distance to the set of such
representatives tending to zero.  If the admissible family is also
$C^2$-closed, fixed-$t$ continuity gives exact minimizers; existence of
exact minimizers is not needed for the stated approximate-minimizer
conclusion.
\end{proof}

\begin{remark}[What this variational statement does not claim]
The geometry bounds are an assumption on the admissible class, not a
consequence of a bound on $\widehat E_t$.  The theorem does not establish
equicoercivity on the unrestricted space of metrics, exclude bubbling
or conformal degeneration there, or construct a new proof of classical
uniformization.  Its content is that finite-time diffusion statistics
select the classical uniformizer in each bounded, precompact conformal
class that contains it.  A single scalar objective need not determine a
metric uniquely away from its minimizing set.
\end{remark}

\begin{corollary}[A torus rate]\label{cor:uniformizationtorusrate}
Let $g_c=g_0$ be a fixed flat metric and write
$g_t=(1+f_t)g_0$, $\int_M f_t\,dA_{g_0}=0$.  Under the hypotheses of
Corollary~\ref{cor:statisticaluniformization}, if
$\|f_t\|_\infty\le1/8$, then
\[
 \|f_t\|_{H^2(g_0)}\le C_{g_0}\bigl(\epsilon_t+t\bigr)^{1/2}.
\]
If in addition $0<t<t_*$ and each $f_t$ is a trigonometric polynomial with
spectrum in $|k|\le\Lambda_t$, $t\Lambda_t^2\le\theta<\tau_*$ and
$\sA(f_t)\le\varepsilon_*$ (the hypotheses of
Theorem~\ref{thm:flatlattices}), its flat-subtracted coercivity improves
this to $\|f_t\|_{H^2(g_0)}\le C\epsilon_t^{1/2}$, with $C$ depending on
$A,\ell_0,\theta$, whenever the approximate
minimality is relative to an admissible class containing $g_0$.
\end{corollary}
\begin{proof}
For $\chi=0$, \eqref{eq:uniformizationdefect} bounds
$\mathcal K_2(g_t)$.  Proposition~\ref{prop:curvatureequivalence}
gives $\|\Delta f_t\|_2^2\le32\mathcal K_2(g_t)$; the mean-zero
elliptic estimate on the fixed flat torus proves the first bound.
For the second bound, use the strictly positive, cutoff-independent
lower bound for $E_t(g_t)-E_t(g_0)$ in
Theorem~\ref{thm:flatlattices} and the competitor inequality
$E_t(g_t)-E_t(g_0)\le\epsilon_t$.  In this resolved class the flat metric
is the unique minimizer, so no $O(t)$ error is necessary.
\end{proof}

\subsection{A curved finite-amplitude consistency theorem}
\label{subsec:fixedmodecurved}
The uniform value estimate already gives a rigorous fixed-mode test on
curved backgrounds, without interchanging a small-time limit and a
second variation.  Let $K_{g_c}=K_c$ and
$g_\eta=(1+\eta f)g_c$, where $f$ is a fixed smooth function with
$\int f\,dA_{g_c}=0$.  For sufficiently small fixed $\eta_0>0$, the
family $|\eta|\le\eta_0$ satisfies \eqref{eq:boundedgeometryclass}.
Define its central finite-amplitude response by
\[
 Q_{t,\eta}(f)=\frac{\widehat E_t(g_\eta)+\widehat E_t(g_{-\eta})
                         -2\widehat E_t(g_c)}{2\eta^2},\qquad\eta\ne0.
\]

\begin{proposition}[Curved finite-amplitude response]
\label{prop:curvedfiniteamplitude}
There is $C_f<\infty$ such that
\begin{equation}\label{eq:curvedfiniteamplitude}
 \left|Q_{t,\eta}(f)
 -\frac14\int_M(\Delta_{g_c}f+2K_cf)^2\,dA_{g_c}\right|
 \le C_f\left(\eta^2+\frac{t}{\eta^2}\right)
\end{equation}
for $0<|\eta|\le\eta_0$ and $0<t\le t_0$.
In particular $\eta=t^{1/4}$ gives convergence with error $O_f(t^{1/2})$.
For $-\Delta_{g_c}f=\lambda f$, the limiting response is
$(\lambda-2K_c)^2\|f\|_2^2/4$.
\end{proposition}
\begin{proof}
Area is exactly fixed along this path.  Write $\rho_\eta$ for the conformal
factor and $F=K^2\rho$. The conformal curvature formula gives
$\dot K=-\frac12(\Delta_{g_c}+2K_c)f$ and
$\ddot F=2\dot K^2+2K_c\ddot K+4K_c\dot Kf$ at $\eta=0$. Differentiating the
Gauss--Bonnet constraint $\int K\rho\,dA_{g_c}=2\pi\chi$ twice gives
$\int\ddot K=-2\int\dot Kf$, which cancels the last two terms after
integration. Hence
\[
 \frac12\frac{d^2}{d\eta^2}\bigg|_0\mathcal K_2(g_\eta)
 =\int_M\dot K^2\,dA_{g_c}
 =\frac14\int_M(\Delta_{g_c}f+2K_cf)^2\,dA_{g_c}.
\]
The exact function $\eta\mapsto\mathcal K_2(g_\eta)$ is smooth on
$[-\eta_0,\eta_0]$.  Its fourth derivative is bounded, so its central
second difference divided by $2\eta^2$ differs from the displayed
quadratic coefficient by at most $C_f\eta^2$.  Theorem~\ref{thm:uniformCalabi}
changes the numerator of that second difference by at most $4Ct$,
and hence its quotient by at most $2Ct/\eta^2$.  This proves the claim.
\end{proof}

The sphere's $\lambda=2K_c$ first spherical harmonics are the classical
conformal automorphism null directions.  At the next eigenvalue
$\lambda=6K_c$, division by $\lambda^2\|f\|_2^2$ gives the limiting
finite-amplitude coefficient $1/9$, rather than the flat value $q(0)=1/4$.
The curvature-corrected low-frequency operator must therefore be retained.
This is a statement about a specified joint limit $t/\eta^2\to0$,
not a proof of convergence of the actual Hessian at $\eta=0$.

\begin{remark}[The remaining differentiated estimate]
If one additionally proves
$\|\widehat E_t(g_s)-\mathcal K_2(g_s)\|_{C^2_s}\le Ct$
for a compact smooth metric family, then the same algebra yields
\[
 \tfrac12D^2\widehat E_t(g_c)[f,f]
 =\tfrac14\int_M(\Delta_{g_c}f+2K_cf)^2\,dA_{g_c}+O_f(t).
\]
We do not infer this differentiated estimate from uniform convergence
of values.  In particular, differentiating the source-Fisher quotient
twice introduces global terms with denominators $p_t^2$ and $p_t^3$;
these require parameter-jet Gaussian bounds and a compatible positive
lower bound on $p_t$, beyond the undifferentiated localization argument.
Even such a fixed-mode Hessian estimate would not establish the
semiclassical regime $\lambda\asymp t^{-1}$ or the universality of
$\tau_*$ on curved surfaces.
\end{remark}

\begin{theorem}[Radial strict convexity and absence of nonflat critical points]
\label{thm:radialselection}
Shrink the amplitude and compact-time constants in
Theorem~\ref{thm:stability}, if necessary. There are $a_*,b_*>0$,
independent of the Fourier cutoff, such that the following holds for either
$\mathcal E=E_t^\infty$ or $\mathcal E=E_t^T$ (with $0<t<t_*$ in the compact
case). Every $f$ satisfying the hypotheses of Theorem~\ref{thm:stability}
satisfies \eqref{eq:radialhessian} and consequently, with
$\widetilde{\mathcal E}(f)=\mathcal E(f)-\mathcal E(0)$,
\eqref{eq:radialvalue}--\eqref{eq:radialgradient}:
\begin{align}
 a_*N(f)&\le D^2\mathcal E(sf)[f,f]\le b_*N(f),\qquad |s|\le1,\label{eq:radialhessian}\\
 \tfrac{a_*}{2}N(f)&\le\widetilde{\mathcal E}(f)\le\tfrac{b_*}{2}N(f),\label{eq:radialvalue}\\
 a_*N(f)&\le D\mathcal E(f)[f]\le b_*N(f).\label{eq:radialgradient}
\end{align}
Every nonconstant ray has a strictly increasing energy for positive radial
parameter. The flat metric is the only point of this class at which the
first variation vanishes in all permitted mean-zero Fourier directions $|k|\le\Lambda$.
\end{theorem}
\begin{proof}
The quadratic response and compact comparison give common $a_0,b_0>0$ with
$a_0N(f)\le D^2\mathcal E(0)[f,f]\le b_0N(f)$. The third-variation
estimates of Corollary~\ref{cor:planevariations} and Lemma~\ref{lem:compact}
give $|D^2\mathcal E(sf)[f,f]-D^2\mathcal E(0)[f,f]|\le M\varepsilon|s|N(f)$
for $|s|\le1$, with a common cutoff-independent constant $M$; choose
$\varepsilon_*$ with $M\varepsilon_*\le a_0/2$.
Then \eqref{eq:radialhessian} holds with $a_*=a_0/2$ and
$b_*=b_0+a_0/2$. Integrate once and twice in $s$, using
$D\mathcal E(0)[f]=0$, to obtain \eqref{eq:radialvalue}--\eqref{eq:radialgradient}.
If all permitted first variations vanish, the direction $f$
itself gives $0\ge a_*N(f)$, and mean zero forces $f=0$.
\end{proof}

\begin{corollary}[A dual-gradient inequality and stability of forced critical points]
\label{cor:radialPL}
Let $X_\Lambda$ be the real space of mean-zero
trigonometric polynomials supported in $|k|\le\Lambda$, with
$\|h\|_{\dot H^2}=\|\Delta h\|_2$, and for linear $\ell$ on $X_\Lambda$ let
$\|\ell\|_{-2,\Lambda}=\sup_{h\in X_\Lambda\setminus\{0\}}|\ell[h]|/\|\Delta h\|_2$.
Under the hypotheses of Theorem~\ref{thm:radialselection},
\begin{equation}\label{eq:radialPL}
 \|D\mathcal E(f)\|_{-2,\Lambda}^2
   \ge\frac{2a_*^2}{b_*}\,\widetilde{\mathcal E}(f).
\end{equation}
If $D\mathcal E(f)=\ell$ on $X_\Lambda$, then
$\|\Delta f\|_2\le\|\ell\|_{-2,\Lambda}/a_*$ and
$\widetilde{\mathcal E}(f)\le\frac{b_*}{2a_*^2}\|\ell\|_{-2,\Lambda}^2$;
for amplitudes at most $1/8$, also
$\mathcal C(f)\le a_*^{-2}\|\ell\|_{-2,\Lambda}^2$. The same
bounds hold with $\|\ell\|_{-2,\Lambda}$ replaced by any upper bound for
$\|D\mathcal E(f)\|_{-2,\Lambda}$.
\end{corollary}
\begin{proof}
Testing the dual norm against $h=f$, \eqref{eq:radialgradient}
gives $\|D\mathcal E(f)\|_{-2,\Lambda}\ge a_*\sqrt{N(f)}$; with the upper
bound in \eqref{eq:radialvalue} this gives \eqref{eq:radialPL}. The forced equation gives
$a_*N(f)\le|\ell[f]|\le\|\ell\|_{-2,\Lambda}\sqrt{N(f)}$;
the rest follows from the value and curvature comparisons.
\end{proof}
Under the hypotheses of Theorem~\ref{thm:softstability} (preserved
along $sf$) the same proofs apply, using the margin $\mu$ and
Lemma~\ref{lem:softjets}, after decreasing the amplitude and compact-time
constants, with $X_\Lambda$ replaced by any finite-dimensional space of
smooth mean-zero functions containing $f$ (no mixed-direction bound is asserted).

\section{A tame Hessian estimate at nonconstant backgrounds}
\label{sec:mixedhessian}
We now control a variation $u$ at an independent background $f$, in terms of the norms
$\mathsf A,\mathsf H,\mathsf G$ of Section~\ref{subsec:notation}. All Fourier supports lie in one resolved ball
$|k|\le\Lambda$, $t\Lambda^2\le\theta$. For intermediate formulas we extend the plane energy by
$E_t(g)=9t^{-2}(\int_{\mathbb T^2}(1+g)-V_t(g))$; this equals the previous functional
on the mean-zero slice, and the linear area term does not affect Hessians.

\begin{theorem}[Tame background comparison]\label{thm:mixedhessian}
There are $\varepsilon_\theta,C_\theta>0$, independent of the cutoff and
mode count, such that real trigonometric $f,u$ with
$\mathsf A(f)\le\varepsilon_\theta$ satisfy
\begin{equation}\label{eq:mixedhessian}
 |D^2E_t^\infty(f)[u,u]-D^2E_t^\infty(0)[u,u]|
 \le C_\theta\bigl[
 \mathsf A(f)\mathsf H(u)^2+
 \mathsf A(u)\mathsf H(f)\mathsf H(u)+
 \mathsf A(u)^2\mathsf H(f)^2\bigr].
\end{equation}
The same estimate holds for $R_t^\infty=E_t^\infty-\mathcal C$.
For mean-zero $f,u$, it also holds for the compact functionals at all
sufficiently small times, with uniform constants. The direction $u$ need
not have small amplitude.
\end{theorem}

\subsection{Two-parameter Cauchy estimation for the scalar terms}
Let $W_\Lambda$ be the space of complex-valued trigonometric
polynomials supported in $|k|\le\Lambda$ (no symmetry imposed on the
coefficients), with the norms $\mathsf A,\mathsf H,\mathsf G$ given by the
same formulas, and $V_\Lambda\subset W_\Lambda$ its real-valued subspace;
$W_\Lambda=V_\Lambda\oplus iV_\Lambda$, as $\operatorname{Re}g$,
$\operatorname{Im}g$ have spectrum in $|k|\le\Lambda$.

\begin{lemma}[A holomorphic tame Hessian lemma]\label{lem:tametwocauchy}
Let $\Phi$ be holomorphic on $\{g\in W_\Lambda:\mathsf A(g)<r\}$ and
suppose $|\Phi(g)|\le M_\Phi\mathsf H(g)^2$ there.
Then, for $\mathsf A(f)\le r/8$,
$|D^2\Phi(f)[u,u]-D^2\Phi(0)[u,u]|\le C_rM_\Phi[\mathsf A(f)\mathsf H(u)^2+
\mathsf A(u)\mathsf H(f)\mathsf H(u)+\mathsf A(u)^2\mathsf H(f)^2]$. The constant is independent of the dimension of the Fourier space.
\end{lemma}
\begin{proof}
Put $e=\mathsf A(f)$, $b=\mathsf A(u)$, $X=\mathsf H(f)$,
$Y=\mathsf H(u)$; zero directions are immediate. Choose
$2\le R\le R_{\max}=r/(4e)$ and $S=r/(4b)$.
On $|z|=R$, $|\eta|=S$, the argument $zf+\eta u$ lies in
$\mathsf A\le r/2$ and $|\Phi(zf+\eta u)|\le M_\Phi(RX+SY)^2$, so the
$\eta^2$ coefficient $a_2(z)=D^2\Phi(zf)[u,u]/2$ satisfies
$|a_2(z)|\le M_\Phi(RX+SY)^2/S^2$. Cauchy's formula
$a_2(1)-a_2(0)=\frac1{2\pi i}\int_{|z|=R}a_2(z)\,dz/(z(z-1))$ and $R\ge2$ give
\begin{equation}\label{eq:tameoptimize}
 |D^2\Phi(f)[u,u]-D^2\Phi(0)[u,u]|
 \le C_rM_\Phi\left[Y^2/R+bXY+Rb^2X^2\right].
\end{equation}
Choose $R$ to be $Y/(bX)$ clipped to $[2,R_{\max}]$.
If this ratio is above $R_{\max}$, then $Y^2/R_{\max}=4eY^2/r$, and
$bX<Y/R_{\max}$ gives $R_{\max}b^2X^2<Y^2/R_{\max}$, so the first and
third terms are bounded by $C_reY^2$; if it is inside the interval they
equal $bXY$; if it is below $2$, they are bounded by $Cb^2X^2$.
For $X=0$ choose $R_{\max}$, and for $Y=0$ choose $R=2$.
\end{proof}

Lemma~\ref{lem:multicomplex} supplies the complex bounds on a ball of
$W_\Lambda$ (not only on the families $1+zF$, $F$ real, of
Sections~\ref{sec:kernel}--\ref{mom:section}) that Lemma~\ref{lem:tametwocauchy} needs. By Lemma~\ref{lem:interp}(b), for $g\in W_\Lambda$,
\begin{equation}\label{eq:complexGNmixed}
 \mathsf G(g)^2\le2\sqrt5\,\|g\|_\infty\mathsf H(g)
              \le2\sqrt5\,\mathsf A(g)\mathsf H(g).
\end{equation}

\begin{lemma}[Several complex coefficient parameters]\label{lem:multicomplex}
Fix $\theta<\infty$ and $t\Lambda^2\le\theta$. There are $r_W>0$ and
$C_\theta<\infty$, independent of $\Lambda$, $t$ and the mode count, with
the following properties on the open ball
$B_W=\{g\in W_\Lambda:\mathsf A(g)<r_W\}$. Put $\rho=1+g$, $a=\rho^{-1}$,
$L=a\Delta$. In (b)--(d), the vectors called $g$ in
Section~\ref{mom:section} ($\nabla\rho$) and Section~\ref{sec:secondscore}
($\nabla\rho_z(y)$) become $\nabla\rho=\nabla g$.
\begin{enumerate}
\item[(a)] The Levi series of Section~\ref{sec:kernel} for $\rho$ defines
a plane kernel $r_g(u,x,y)$, holomorphic in the coefficients of $g$, with
$|r_g|\le C_\theta u^{-1}e^{-\kappa_-|x-y|^2/u}$ and
$|\nabla_xr_g|\le C_\theta u^{-3/2}e^{-\kappa_-|x-y|^2/u}$ for $0<u\le t$.
\item[(b)] Its periodization $P_u$ satisfies \eqref{mom:Lp},
\eqref{mom:gradprop}, $P_u1=1$, the weighted transpose identity
\eqref{mom:transpose}, the commutation \eqref{mom:H2}, and the exact
moment identities \eqref{mom:M3exact} and \eqref{mom:M4exact}, all with
$zF$ replaced by $g$.
\item[(c)] The functionals $\mathfrak K_t$ and $\mathfrak M_t$ defined
below and in \eqref{mom:Mdef}, and the curvature energy $\mathcal C$ written
with bilinear products, are holomorphic on $B_W$ and satisfy
$|\mathfrak K_t(1+g)|+|\mathfrak M_t(1+g)|\le C_\theta t^2\mathsf H(g)^2$ and
$|\mathcal C(g)|\le C\mathsf H(g)^2$.
\item[(d)] The residual numerators of Sections~\ref{naani:sec}
and~\ref{sec:secondscore}, with $zf$ replaced by $g$, satisfy
$\|\mathcal N(g;t,w)\|_{[4]}\le C_\theta\mathsf G(g)t^{-1}e^{-\kappa_-|w|^2/t}$ and
$\|U(g;t,w)\|_{[2]}\le C_\theta\sqrt t\,\mathsf H(g)G_{\kappa_-}(t,w)$.
\end{enumerate}
\end{lemma}
\begin{proof}
\emph{(a)} Steps~1 and~2 of the proof of Theorem~\ref{thm:neargaussian}
use the coefficient only through $|\rho-1|\le r_0$, which gives
$\Re(\rho(y)/\omega)\ge1-\gamma/4$, and the Lipschitz bound
\eqref{eq:paramlip} for $\rho^{-1}$ in rescaled variables. For
$g\in W_\Lambda$, $\|g\|_\infty\le\mathsf A(g)$ and
$\sqrt t\|\nabla g\|_\infty\le\sqrt\theta\,\mathsf A(g)$, hence
$\sqrt t\|\nabla\rho^{-1}\|_\infty\le\sqrt\theta\,\mathsf A(g)/(1-r_0)^2$.
This is \eqref{eq:paramlip} with $\mathsf A(g)$ for $|\zeta|$,
so for $r_W\le r_0$ the bounds \eqref{eq:levibounds}--\eqref{eq:levicorrection}
and the stated kernel bounds hold uniformly on $B_W$, as in the
proof of Theorem~\ref{thm:neargaussian}(b). The Levi series converges normally on $B_W$ and its terms are holomorphic
in the finitely many coefficients of $g$, so it is holomorphic on this open
subset of $\mathbb C^M$.

\emph{(b)} The bounds follow from (a) by periodization and the Schur and
segment arguments after \eqref{mom:kernelinput}, which use
only the kernel bounds. Each identity has the form $\Psi(g)=0$ with
$\Psi$ holomorphic on $B_W$ by normal convergence, and holds for
real-valued $g\in B_W\cap V_\Lambda$, where $P_u$ is the Markov semigroup
of a positive coefficient. Since $W_\Lambda=V_\Lambda\oplus iV_\Lambda$,
every complex-multilinear Taylor form of $\Psi$ at $0$ is determined by
its values on $V_\Lambda$; so the Taylor series at $0$
vanishes, and the identity theorem on the connected ball $B_W$ gives $\Psi\equiv0$.

\emph{(c)} For $\mathfrak K_t$, the formulas
$L\rho=a\Delta g$ and $La=-a^3\Delta g+2a^4\nabla g\cdot\nabla g$, with
$|a|\le(1-r_W)^{-1}$ and \eqref{eq:complexGNmixed}, give
$\|L\rho\|_2+\|La\|_2\le C\mathsf H(g)$; with \eqref{mom:Lp} at $p=2$ and
$t^{-1}\int_0^t(t-u)u\,du=t^2/6$ this gives
$|\mathfrak K_t|\le C_\theta t^2\mathsf H(g)^2$. For $\mathfrak M_t$, the proof of
Theorem~\ref{mom:main} uses (b), the bounds
$\|a\|_\infty+\|\rho\|_\infty\le C$, the explicit derivative formulas for
$a$, and \eqref{mom:GN}. There a factor $|z|\sH(F)$ is
$\mathsf H(zF)$ and $|z|^2\sG(F)^2$ is $\mathsf G(zF)^2$; replacing
$zF$ by $g$ and \eqref{mom:GN} by \eqref{eq:complexGNmixed} changes only
constants (each extra factor $|z|$ becomes $\mathsf A(g)\le r_W$), giving the analogue of \eqref{mom:Mbound},
$|\mathfrak M_t(1+g)|\le C_\theta t^2\mathsf H(g)^2$. Finally
$\mathcal C(g)=\frac14\int_{\mathbb T^2}a^3(\Delta g-a\nabla g\cdot\nabla g)^2dx$,
so $|\mathcal C(g)|\le C(\mathsf H(g)+\mathsf G(g)^2)^2\le C\mathsf H(g)^2$.
Holomorphy follows from (a), (b) and normal convergence.

\emph{(d)} The proofs of \eqref{naani:Ncomplex} and
\eqref{score:eq:complex} use the Levi representation, the marked
translation estimates \eqref{score:eq:marked}, and the coefficient bounds
\eqref{score:eq:rho-diff}--\eqref{score:eq:gradient-diff}. With $zf$
replaced by $g$, these hold with $|z|F_4\mapsto\mathsf G(g)$, $|z|H\mapsto\mathsf H(g)$, since only
$\|g\|_\infty\le r_W$, the Lipschitz bound of (a) and the pointwise
derivative formulas enter; the bridge identity of
Lemma~\ref{score:lem:bridge} is pointwise in the target and holds for
complex $\rho(y)$ near $1$. Hence the two bounds hold with
$|z|F_4\mapsto\mathsf G(g)$ and $|z|H+|z|^2F_4^2\mapsto
\mathsf H(g)+\mathsf G(g)^2\le C\mathsf H(g)$, by
\eqref{eq:complexGNmixed}.
\end{proof}

Write the clock identity as $\mathfrak A_t=\mathfrak K_t+\mathcal J_t$ with
$\mathfrak K_t=-\frac1t\int_0^t(t-u)\int_0^u
\int_{\mathbb T^2}\rho(P_vL\rho)La\,dx\,dv\,du$.
Since $\mathfrak B_t=\mathcal S_t+\mathfrak M_t-\mathcal J_t$, the
reciprocal-clock square cancels exactly:
\begin{equation}\label{eq:mixedmaster}
 E_t^\infty=\frac9{t^2}
       (-\mathfrak K_t-\mathcal S_t-\mathfrak M_t+D_t).
\end{equation}
By Lemma~\ref{lem:multicomplex}(c), $\mathfrak K_t/t^2$, $\mathfrak M_t/t^2$
and $\mathcal C$ satisfy the hypotheses of Lemma~\ref{lem:tametwocauchy}
with $r=r_W$, which gives the Hessian comparison for these three terms.

\subsection{Mixed score jets and real quotient derivatives}
For fixed displacement write $G_{\kappa_-}(t,w)=(\kappa_-/\pi)t^{-1}e^{-\kappa_-|w|^2/t}$,
so that Lemma~\ref{lem:multicomplex}(d) bounds $\|U(g;t,w)\|_{[2]}$ by
$C_\theta\sqrt t\,\mathsf H(g)G_{\kappa_-}(t,w)$ on $B_W$.
At the real background $g=sf$, $0\le s\le1$, apply the Banach-valued
Cauchy formula on a polydisc with directional radii
$r_W/(4m\mathsf A(v_i))$; shrinking the background ball once keeps all
arguments in the coefficient domain. For $1\le m\le3$ this gives
\begin{equation}
 \|D^mU(g)[v_1,\ldots,v_m]\|_{[2]}
 \le C_m\sqrt t\,G_{\kappa_-}
 \Bigl[\mathsf H(g)\prod_i\mathsf A(v_i)
 +\sum_i\mathsf H(v_i)\prod_{j\ne i}\mathsf A(v_j)\Bigr].
 \label{eq:Umixedjets}
\end{equation}
Zero directions are treated directly. For $\mathcal N$ replace
$\mathsf H$ by $\mathsf G$ and omit $\sqrt t$; mixed density jets are bounded by
$C_mG_{\kappa_-}\prod_i\mathsf A(v_i)$. These Cauchy arguments concern analytic numerators only.

With $e,b,X,Y$ as in the proof of Lemma~\ref{lem:tametwocauchy}, and
suppressing the factor $C\sqrt tG_{\kappa_-}$, \eqref{eq:Umixedjets} bounds
$U,U_f,U_u,U_{fu},U_{uu},U_{fuu}$ at $sf$ (with $e\le1$) by
$X,\ X,\ Y+bX,\ eY+bX,\ bY+b^2X,\ ebY+b^2X$ respectively.
For real coefficients use the positive lower Gaussian bound for $r$ and
expand the finite real derivative of $\mathcal S_t=t\int |U|^2/r$.
In $D^3\mathcal S_t(sf)[f,u,u]$, representative products are
$U_fU_{uu}$, $U_uU_{fu}$, $UU_{fuu}$,
$r_fU_u^2/r$, and $r_uU_fU_u/r$; the remaining terms contain two or three
density derivatives. These bounds and the amplitude factors bound every term by
$\Xi=eY^2+bXY+b^2X^2$; for example $X(bY+b^2X)\le \Xi$ and
$(Y+bX)(eY+bX)\le C \Xi$, and extra density derivatives supply precisely the
missing factors $e$ or $b$ when fewer derivatives land on $U$.
With score numerators in $L^2_x$, density jets in supremum norm, and the
Gaussian quotient exponent positive for derivative order $\le3$ under the margin
$6\kappa_--5\kappa_+>0$, displacement integration, the two $\sqrt t$ factors, and the exterior
factor $t$ give
\begin{equation}\label{eq:Smixedthird}
 |D^3\mathcal S_t(sf)[f,u,u]|\le C_\theta t^2\Xi,
\end{equation}
and integration in $s$ proves the Hessian comparison for
$\mathcal S_t/t^2$.

For anisotropy use
$D_t/t^2=4\int \mathfrak b_t\,|\Sigma^{\rm tf}|^2$, with $\Sigma=\int\mathcal N\mathcal N/r$ and $\mathfrak b_t=(\operatorname{tr}h+2\sqrt{\det h})^{-1}$.
Mixed kernel quotient estimates and uniform positivity of $h$ give
$\|D^m\mathfrak b_t[v_1,\ldots,v_m]\|_\infty
\le C_m\prod_i\mathsf A(v_i)$ for $m\le3$.
The expression has four $\mathcal N$ numerator factors and two inverse densities
in independent displacement variables. After three derivatives in the directions $f,u,u$, at least one numerator
remains a background factor; the $f$ derivative contributes a
$\mathsf G(f)$ marker or a factor $e$, and each $u$ derivative
a $\mathsf G(u)$ marker or a factor $b$.
The mixed version of \eqref{eq:Umixedjets} also covers repeated derivatives
on a single numerator. H\"older with four $L^4_x$ numerator factors and
real Gaussian integration give
$|D^3(D_t/t^2)(sf)[f,u,u]|\le C_\theta[\mathsf G(f)^2\mathsf G(u)^2
+b\mathsf G(f)^3\mathsf G(u)+b^2\mathsf G(f)^4]$. The middle term is at most half the sum of the other two, which
interpolation bounds by $10ebXY$ and $10e^2b^2X^2$, hence by $C\Xi$.
Integration in $s$ proves the anisotropy comparison; with
\eqref{eq:mixedmaster} and Lemma~\ref{lem:tametwocauchy} this proves the
plane assertions of Theorem~\ref{thm:mixedhessian}.

\subsection{Compact transfer and the scope of convexity}
\label{subsec:compacttransfer}
The differentiated compact comparison extends to mixed parameter labels,
each density or source jet contributing the product of the Wiener norms
of its directions; the finite real quotient-product proof is otherwise
unchanged. It gives, uniformly in $0\le s\le1$,
$|D^3(E_t^T-E_t^\infty)(sf)[f,u,u]|\le Ct^{-3}e^{-c\ell_0^2/t}\mathsf A(f)\mathsf A(u)^2$
(polynomial powers of $t^{-1}$ being absorbed by reducing $c$). For mean-zero $u$,
$\mathsf A(u)\le C_W\mathsf H(u)$ (Lemma~\ref{lem:interp}(c)) bounds
the compact correction by a constant times
$\mathsf A(f)\mathsf H(u)^2$ for small times.
This completes Theorem~\ref{thm:mixedhessian}.

In particular, the tame error on the mean-zero slice is at most
$C_\theta[\mathsf A(f)+C_W\mathsf H(f)+C_W^2\mathsf H(f)^2]\mathsf H(u)^2$.
A positive flat response band therefore gives a uniform positive Hessian
on a small convex $H^2$ background ball in each resolved mean-zero Fourier
space, with cutoff-independent radius; since a small Wiener amplitude does
not bound $\mathsf H(f)$, nothing is claimed on the whole small Wiener ball.

\subsection{Consequence: an invariant Galerkin gradient flow}
\label{sec:galerkin}
Fix $t$; let $\varsigma$ be an optimization time and define
$\operatorname{grad}_{\dot H^2}\mathcal E$ on $X_\Lambda$ by
$\int_{\mathbb T^2}\Delta(\operatorname{grad}_{\dot H^2}\mathcal E(f))\Delta h\,dx
=D\mathcal E(f)[h]$ for all $h\in X_\Lambda$.

\begin{corollary}[Uniform contraction of the finite-dimensional flow]
\label{cor:galerkinflow}
Fix $0<\theta<\tau_*$ and a real mean-zero Fourier space $X_\Lambda$ with
$t\Lambda^2\le\theta$; let $\mathcal E=E_t^\infty$ or, for small $t$,
$\mathcal E=E_t^T$. There are $R,\lambda,\bar\lambda>0$, independent of
cutoff and dimension (possibly deteriorating as $\theta\uparrow\tau_*$),
with $\lambda\|h\|_{\dot H^2}^2\le D^2\mathcal E(f)[h,h]
\le\bar\lambda\|h\|_{\dot H^2}^2$ for $f,h\in X_\Lambda$,
$\|f\|_{\dot H^2}\le R$. Hence the flow
$df/d\varsigma=-\operatorname{grad}_{\dot H^2}\mathcal E(f)$ started in the
radius-$R$ ball stays there for all $\varsigma\ge0$,
\begin{equation}\label{eq:flowdecay}
 \|f(\varsigma)\|_{\dot H^2}\le e^{-\lambda\varsigma}\|f(0)\|_{\dot H^2},\qquad
 \mathcal E(f(\varsigma))-\mathcal E(0)\le e^{-2\lambda\varsigma}
 [\mathcal E(f(0))-\mathcal E(0)],
\end{equation}
and two solutions contract at rate $e^{-\lambda\varsigma}$. If
$\ell\in X_\Lambda^*$ has $\|\ell\|_{-2,\Lambda}<\lambda R$, then $D\mathcal E(f)=\ell$ has exactly one
solution $f_\ell$ in the closed ball, $\|f_\ell\|_{\dot H^2}\le
\lambda^{-1}\|\ell\|_{-2,\Lambda}$, $\ell\mapsto f_\ell$ is
$\lambda^{-1}$-Lipschitz, and the flow of $\mathcal E-\ell$ preserves the
closed ball and converges to $f_\ell$ at rate $e^{-\lambda\varsigma}$.
\end{corollary}
\begin{proof}
The Hessian bounds follow from Theorem~\ref{thm:mixedhessian}, the embedding $\sA(f)\le C_W\sH(f)$ and the positive flat response (in the
compact case after shrinking $t$ so that the compact correction is below the
flat margin). Integrating along segments shows that
$\mathscr G=\operatorname{grad}_{\dot H^2}\mathcal E$, with $\mathscr G(0)=0$,
is $\lambda$-strongly monotone on the ball:
$\langle\mathscr G(f_1)-\mathscr G(f_2),f_1-f_2\rangle_{\dot H^2}
\ge\lambda\|f_1-f_2\|_{\dot H^2}^2$. Writing $\|\cdot\|=\|\cdot\|_{\dot H^2}$,
monotonicity gives $\frac d{d\varsigma}\frac12\|f\|^2\le-\lambda\|f\|^2$
(with $f_2=0$), hence invariance of the ball, global existence by
finite-dimensional continuation and decay of $\|f\|$, and
$\frac d{d\varsigma}\frac12\|f_1-f_2\|^2\le-\lambda\|f_1-f_2\|^2$ for two
solutions. The dissipation identity
$\frac d{d\varsigma}\mathcal E(f)=-\|\mathscr G(f)\|^2$ and the
Polyak--\L ojasiewicz inequality
$\mathcal E(f)-\mathcal E(0)\le(2\lambda)^{-1}\|\mathscr G(f)\|^2$
give the energy decay. In the forced problem the radial derivative
$\lambda R^2-\|\ell\|_{-2,\Lambda}R>0$ on the boundary sphere puts the
minimizer of $\mathcal E-\ell$ on the closed ball in the interior, so it solves
$D\mathcal E(f)=\ell$; strong monotonicity gives uniqueness,
$\|f_\ell\|\le\lambda^{-1}\|\ell\|$,
$\|f_\ell-f_{\ell'}\|\le\lambda^{-1}\|\ell-\ell'\|$, invariance of the ball
under the forced flow
($\frac d{d\varsigma}\frac12\|f\|^2\le-\lambda\|f\|^2+\|\ell\|\|f\|<0$ on the
sphere), and convergence at rate $e^{-\lambda\varsigma}$ by comparison with
the stationary solution $f_\ell$.
\end{proof}

\begin{remark}[The gradient metric matters]\label{rem:Calabiflow}
This finite-dimensional evolution at fixed observation time is a
$\dot H^2$-preconditioned descent for a heat-regularized Calabi-type energy.
On the plane lift its linearization at the flat metric is
$d\widehat f_\xi/d\varsigma=-2q(t|\xi|^2)\widehat f_\xi$, whereas the
$\dot H^2$ descent for $\mathcal C$ has multiplier $1/2$; the two agree in
the low-resolution limit. On a compact flat torus the Hessian (hence the
linearized gradient) has an extra operator error at most
$C_{A,\ell_0,\theta}t^{-3}e^{-c\ell_0^2/t}$ in the normalized $\dot H^2$
norm. The surface Calabi flow $\partial_\varsigma g=(\Delta_gK_g)g$
linearizes as $\partial_\varsigma f=-\frac12\Delta^2f$, with
multiplier $|\xi|^4/2$, and uses a different gradient metric
\cite{StruweCalabi}; we do not construct an approximation to it or an
infinite-dimensional limit, which would require further uniform estimates,
a specified gradient metric and compatible observation scales. The same
finite-mode results hold for the flat families of Theorem~\ref{thm:flatlattices}.
\end{remark}

\subsection{Intrinsic genus-one interpretation}
\label{subsec:latticetransfer}
\begin{corollary}[Uniformization and fixed smooth near-flat metrics]
\label{cor:genusone}
Let $g$ be a smooth metric of area $A$ on a closed oriented surface $\mathcal M$ of
genus one, and let $\bar g$ be the flat metric of area $A$ in its conformal
class. Identify $(\mathcal M,\bar g)$ isometrically with $T_{\mathcal L}$ and write
$g=(1+f)\bar g$. Then $f$ has mean zero. If $f$ satisfies the amplitude
and resolved-band hypotheses of Theorem~\ref{thm:flatlattices} (so
$f$ is a trigonometric polynomial with small Wiener norm and $t\Lambda^2\le\theta$),
its intrinsic heat-Fisher deficit satisfies the comparison there.
The same holds with the soft Lipschitz and
spectral-margin conditions of Theorem~\ref{thm:softstability}.
Moreover there is a Wiener radius, depending only on $A$ and a lower bound
for $\operatorname{sys}(\mathcal L)$, such that every
\emph{fixed} smooth $f$ in that radius satisfies
$\frac1{16}N_{\mathcal L}(f)\le E_t(g)-E_t(\bar g)\le\frac12N_{\mathcal L}(f)$
for $0<t<t_f$, for some profile-dependent $t_f>0$, where
$E_t(g)=9(A-\operatorname{Vol}(\mathcal M,2tJ_t^g))/t^2$; no common time
interval is asserted over arbitrary oscillatory profiles, and Hessian
positivity is in the conformal variable, not in moduli-changing metric
directions. Without any amplitude restriction,
Theorem~\ref{thm:fixedlimit} gives $E_t(g)-E_t(\bar g)\to\operatorname{Ca}(g)$ as
$t\to0$ for every smooth metric $g$ of genus one.
\end{corollary}
\begin{proof}
Uniformization supplies $\bar g$, the area constraint gives
$\int f\,dA_{\bar g}=0$, and the asserted quantities (source Fisher tensors,
their volume, curvature energy) are natural under pullback by the isometry.
For the quantitative fixed-profile statement take $f\ne0$ (the case $f=0$ is immediate), use Theorem~\ref{thm:softstability}
with $L_*=1$ and $\mu=1/8$, and choose its fixed small amplitude radius.
For this one smooth profile,
$\sqrt t\|\nabla f\|_\infty\le\mathsf A_{\mathcal L}(f)$ for all small $t$, and
boundedness of $q$ with dominated convergence in the summable
curvature-energy Fourier series gives
$A\sum_{\xi\ne0}|\xi|^4q(t|\xi|^2)|\widehat f_\xi|^2/N_{\mathcal L}(f)\to\frac14$.
Thus its spectral margin is eventually at least $1/8$, and the compact
comparison is small for this fixed lattice. The constants
$c_*=\mu/2=1/16$ and $C_*=1/2$ of that theorem prove the assertion.
\end{proof}
Since $E_t(\bar g)$ may depend on the lattice shape, Corollary~\ref{cor:genusone}
compares each metric only with its own flat representative; it gives no
convexity in Teichm\"uller directions, and the systole bound provides
uniformity over conformal structures.

\section{The flat structure: loss of information and asymptotic lattice selection}
\label{sec:latticeselection}\label{sec:latticestatistics}
The conformal factor and the flat structure enter at different orders.  We now
make this distinction precise.  Throughout this section the area is fixed at
$A$, and the generator is $\Delta$, so the lifted Brownian displacement has
covariance $2tI$.  For a lattice $\mathcal L$ let
\[
 \phi_t(z)=(4\pi t)^{-1}e^{-|z|^2/(4t)},\qquad
 p_t^{\mathcal L}(z)=\sum_{\ell\in\mathcal L}\phi_t(z+\ell),\qquad
 H_t(\mathcal L)=2tJ_t(\mathcal L),\quad
 d_t(\mathcal L)=1-\frac{V_t(\mathcal L)}A.
\]
The Fisher tensor is constant on a flat torus, and
$V_t(\mathcal L)/A=\sqrt{\det H_t(\mathcal L)}$.

\begin{proposition}[The information lost by wrapping]\label{prop:wrappingloss}
Let $Z\sim N(0,2tI)$ and $Y=Z\bmod\mathcal L$.  Then
\begin{equation}\label{eq:missinginformation}
 I-H_t(\mathcal L)=\frac1{2t}\E\operatorname{Cov}(Z\mid Y)
 =\frac1{4t}\sum_{\ell\in\mathcal L\setminus\{0\}}
     \ell\otimes\ell\int_{\mathbb R^2}
       \frac{\phi_t(z)\phi_t(z+\ell)}{p_t^{\mathcal L}(z)}\,dz.
\end{equation}
In particular, $0\prec H_t(\mathcal L)\prec I$ and $0<V_t(\mathcal L)<A$
for every $t>0$.
\end{proposition}
\begin{proof}
The location score of the unwrapped Gaussian is $Z/(2t)$, and the score
after wrapping is its conditional expectation given $Y$; this follows by
differentiating the absolutely convergent Gaussian sum.  Total covariance
therefore gives the first equality.  At a representative $z$ of $Y$, the
posterior weights are $w_\ell=\phi_t(z+\ell)/p_t^{\mathcal L}(z)$.
The identity
$\operatorname{Cov}_w(z+\ell)=\frac12\sum_{\ell,m}w_\ell w_m
(\ell-m)\otimes(\ell-m)$, followed by integration over a cell and unfolding,
gives the second equality; all sums converge absolutely.  The posterior puts
positive weight on every lattice point, so its covariance is positive definite.
To see that $J_t$ is positive definite, a null direction would make the
directional derivative of $p_t^{\mathcal L}$ identically zero.  Its Fourier
coefficients are $A^{-1}e^{-t|\xi|^2}>0$ for every dual vector $\xi$; since the
dual lattice spans $\mathbb R^2$, that direction must be zero.
\end{proof}

\begin{figure}[tbp]
\centering
\begin{tikzpicture}
\begin{axis}[width=0.85\textwidth,height=5.2cm,xmin=-3.2,xmax=10.5,ymin=0,ymax=1.12,
 axis lines=left,xtick={0,2,4},xticklabels={$0$,$s/2$,$s$},ytick={0.0183,0.3679,1},
 yticklabels={$e^{-s^2/4t}$,$e^{-s^2/16t}$,$1$},tick label style={font=\scriptsize},
 xlabel={position along a shortest lattice vector},xlabel style={font=\small},samples=200]
\addplot[name path=A,very thick,blue!60!black,domain=-3.2:7.2] {exp(-x^2/4)};
\addplot[name path=B,very thick,red!60!black,domain=-3.2:7.2] {exp(-(x-4)^2/4)};
\addplot[name path=M,draw=none,domain=-3.2:7.2] {min(exp(-x^2/4),exp(-(x-4)^2/4))};
\addplot[name path=Z,draw=none,domain=-3.2:7.2] {0};
\addplot[orange!40] fill between[of=M and Z];
\draw[densely dotted] (axis cs:2,0) -- (axis cs:2,0.3679) -- (axis cs:-3.2,0.3679);
\draw[densely dotted] (axis cs:4,0) -- (axis cs:4,0.0183);
\fill (axis cs:2,0.3679) circle (1.6pt);
\node[font=\scriptsize,blue!50!black,anchor=south] at (axis cs:-0.2,1.0) {source $x$};
\node[font=\scriptsize,red!50!black,anchor=south] at (axis cs:4.2,1.0) {image $x+\ell$};
\node[font=\scriptsize,align=left,anchor=west] at (axis cs:6.4,0.5) {overlap: endpoints here\\ cannot tell $x$ from $x+\ell$};
\draw[->,thin] (axis cs:6.35,0.47) -- (axis cs:2.7,0.2);
\end{axis}
\end{tikzpicture}
\caption{Why the winding loss has exponent $s^2/16t$. Along a shortest
lattice vector $\ell$, $|\ell|=s$, the endpoint densities of a diffusion
from $x$ and from its image $x+\ell$ are Gaussians of variance $2t$. Their
overlap (shaded), where an observed endpoint is ambiguous between the two
lifts, is largest at the midpoint, at height $e^{-(s/2)^2/4t}=e^{-s^2/16t}$.
The value of one Gaussian at the other lattice point, $e^{-s^2/4t}$, is much
smaller and does not set the scale (Theorem~\ref{thm:systolicloss}). Drawn
with $s=4$, $t=1$.}
\label{fig:winding}
\end{figure}
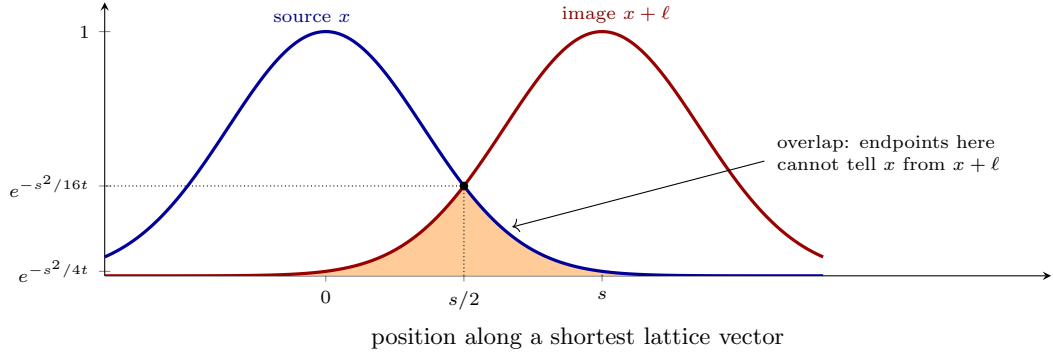

\begin{theorem}[The systolic scale of the volume defect]\label{thm:systolicloss}
Write $s(\mathcal L)=\operatorname{sys}(\mathcal L)$ and let
$m(\mathcal L)$ count its nonzero shortest vectors, counting both signs.
For each fixed lattice,
\begin{equation}\label{eq:shortestloss}
 d_t(\mathcal L)=\frac{\sqrt\pi\,m(\mathcal L)s(\mathcal L)}{8\sqrt t}
  e^{-s(\mathcal L)^2/(16t)}\bigl(1+O_{\mathcal L}(t)\bigr)
 \qquad(t\downarrow0).
\end{equation}
Moreover, for fixed $A,\ell_0>0$, there are $c,C,t_0>0$ such that
all area-$A$ lattices with $s(\mathcal L)\ge\ell_0$ satisfy
\begin{equation}\label{eq:uniformsystolicloss}
 c\,t^{-1/2}e^{-s(\mathcal L)^2/(16t)}
 \le d_t(\mathcal L)\le
 C\,t^{-1/2}e^{-s(\mathcal L)^2/(16t)},\qquad 0<t<t_0.
\end{equation}
Consequently $-16t\log d_t(\mathcal L)\to s(\mathcal L)^2$ uniformly on this
noncollapsing family.  The exponent $s^2/(16t)$ comes from overlapping
images at the midpoint of a shortest vector, rather than the pointwise
image size $e^{-s^2/(4t)}$ at a lattice point.
\end{theorem}
\begin{proof}
For $\ell\ne0$, put $r=|\ell|$ and denote the integral in
\eqref{eq:missinginformation} by $I_\ell(t)$.  Since
$p_t^{\mathcal L}\ge\phi_t(z)+\phi_t(z+\ell)$, coordinates
$z=-\ell/2+u\ell/r+v\ell^\perp/r$ give
\begin{align}
 I_\ell(t)&\le \int_{\mathbb R^2}
  \frac{\phi_t(z)\phi_t(z+\ell)}{\phi_t(z)+\phi_t(z+\ell)}\,dz\notag\\
 &=\frac{\sqrt t}{\sqrt\pi r}e^{-r^2/(16t)}
   \int_{\mathbb R}e^{-4tw^2/r^2}\operatorname{sech}w\,dw
 \le\frac{\sqrt{\pi t}}r e^{-r^2/(16t)}.\label{eq:twogaussianoverlap}
\end{align}
Here $w=ru/(4t)$ and $\int\operatorname{sech}w\,dw=\pi$.
For a shortest vector $\ell$, every image other than $0,\ell$ is strictly
farther from $-\ell/2$: indeed, for $m\ne0,\ell$,
\[
 |m-\ell/2|^2=\tfrac12(|m|^2+|m-\ell|^2)-r^2/4\ge3r^2/4.
\]
On a sufficiently small fixed neighborhood of this midpoint the other
images are exponentially smaller than the two displayed images.  On its
complement the two-image integral has an exponentially smaller value than
\eqref{eq:twogaussianoverlap}: in the normal direction this follows from
$\operatorname{sech}(ru/(4t))$, and in the tangential direction from
$e^{-v^2/(4t)}$.  Gaussian lattice summation justifies the comparison with
all images.  Hence for a shortest vector,
\[
 I_\ell(t)=\frac{\sqrt{\pi t}}r e^{-r^2/(16t)}(1+O_{\mathcal L}(t)),
\]
where the $O(t)$ follows by expanding $e^{-4tw^2/r^2}$ under the integrable
weight $\operatorname{sech}w$.  For a fixed lattice its next length exceeds
$r$, so \eqref{eq:twogaussianoverlap} makes the sum of nonshortest terms
exponentially smaller.  With $Q_t=I-H_t$, formula
\eqref{eq:missinginformation} now gives
$\operatorname{tr}Q_t\sim\sqrt\pi m(\mathcal L)r
 e^{-r^2/(16t)}/(4\sqrt t)$.
Since $1-\sqrt{\det(I-Q)}=\frac12\operatorname{tr}Q+O(|Q|^2)$ as $Q\to0$,
this proves \eqref{eq:shortestloss}.

For uniformity one need not keep a gap between the first two lattice lengths.
The elementary area bound $s(\mathcal L)^2\le2A/\sqrt3$ and packing disjoint
disks of radius $\ell_0/2$ give, for $0<t<t_0(A,\ell_0)$,
\[
 \sum_{\ell\ne0}|\ell|e^{-|\ell|^2/(16t)}
 \le C_{A,\ell_0}e^{-s(\mathcal L)^2/(16t)}.
\]
For example, split at $|\ell|=2\sqrt{2A/\sqrt3}$, where the number of
points is uniformly bounded, and sum the remaining Gaussian annuli.
Together with \eqref{eq:twogaussianoverlap} this bounds
$\operatorname{tr}Q_t$ from above.  For a shortest vector, restrict its
integral to $|u|\le t/r$, $|v|\le\sqrt t$.
On this box the two principal images are comparable.  The preceding
midpoint gap, $r\ge\ell_0$, and the same lattice packing estimate show
$p_t^{\mathcal L}\le C(\phi_t(z)+\phi_t(z+\ell))$ uniformly for small $t$.
The integrand is therefore at least
$c t^{-1}e^{-r^2/(16t)}$ there, and the box has area $4t^{3/2}/r$.
It follows that $I_\ell(t)\ge c\sqrt t\,r^{-1}e^{-r^2/(16t)}$,
giving the uniform lower bound for $\operatorname{tr}Q_t$.
Finally, when $0\preceq Q\prec I$,
$\frac12\operatorname{tr}Q\le1-\sqrt{\det(I-Q)}\le\operatorname{tr}Q$.
This proves \eqref{eq:uniformsystolicloss}; taking logarithms proves the
uniform limit.
\end{proof}

\begin{corollary}[Asymptotic selection of the hexagonal flat structure]
\label{cor:hexselection}
Let $\mathcal L_{\rm hex}$ be the area-$A$ hexagonal lattice.  On the moduli
space of area-$A$ lattices with systole at least $\ell_0>0$, assume this
lattice is admissible.  Every sequence of maximizers of $V_t$ converges,
as $t\downarrow0$, to $\mathcal L_{\rm hex}$ modulo rotations and changes
of lattice basis.  The same holds for a sequence $\mathcal L_t$ satisfying
\[
 V_t(\mathcal L_t)\ge\sup_{\mathcal L}V_t(\mathcal L)-A\eta_t,
 \qquad \eta_t=O(d_t(\mathcal L_{\rm hex})).
\]
For every fixed nonhexagonal lattice, $V_t(\mathcal L_{\rm hex})>
V_t(\mathcal L)$ for all sufficiently small $t$, with a threshold that may
depend on that lattice.  For such a sequence, more quantitatively,
$0\le2A/\sqrt3-s(\mathcal L_t)^2=O(t)$, and its distance to the hexagonal
class in reduced modulus coordinates is $O_{A,\ell_0}(t)$.  These conclusions do not
assert that the exact maximizer equals the hexagonal lattice for each small
positive time.
\end{corollary}
\begin{proof}
The noncollapsing moduli space is compact: take a shortest primitive basis
vector of length $s$, reduce the projection of the second basis vector to
$[-s/2,s/2]$, and note that its perpendicular component is $A/s$.
The Gaussian sums and the Fisher quotient vary continuously there, so
maximizers exist.  The same reduced basis shows
$A\ge(\sqrt3/2)s^2$, with equality precisely at the hexagonal lattice.
The near-maximality assumption gives
$d_t(\mathcal L_t)\le d_t(\mathcal L_{\rm hex})+\eta_t
\le C_0d_t(\mathcal L_{\rm hex})$.  Apply the uniform logarithmic limit
of Theorem~\ref{thm:systolicloss}: then
$s(\mathcal L_t)^2\to2A/\sqrt3$.  The uniform two-sided bounds give the stronger estimate
\[
 0\le2A/\sqrt3-s(\mathcal L_t)^2
       \le16t\log(C_0C/c).
\]
For a reduced basis the modulus is $z=x+iy$, with
$y=A/s(\mathcal L_t)^2$, $|x|\le1/2$ and $x^2+y^2\ge1$.
Consequently $y-\sqrt3/2=O_{A,\ell_0}(t)$ and
$1/2-|x|=O_{A,\ell_0}(t)$; the two boundary points
$\pm1/2+i\sqrt3/2$ represent the same hexagonal class.
This proves both convergence and its asserted rate.
For a fixed competitor with strictly smaller systole,
\eqref{eq:shortestloss} shows
$d_t(\mathcal L_{\rm hex})/d_t(\mathcal L)\to0$.
\end{proof}

\begin{remark}[Two scales, and the accuracy needed for lattice selection]
\label{rem:twoscales}
On a fixed smooth torus the curvature contribution to $A-V_t(g)$ is
$t^2\mathcal K_2(g)/9+O(t^3)$ by Theorem~\ref{thm:fixedlimit}; for a
flat torus the residual defect is instead exponentially small as in
\eqref{eq:shortestloss}.  Thus polynomial-accuracy small-time observations
can detect a nonzero curvature energy while failing to resolve the flat
structure.  In particular an $o(1)$ or polynomially small error in maximizing
$V_t$ cannot imply hexagonal selection: every fixed flat torus has
$V_t/A\to1$.  The exponentially accurate hypothesis in
Corollary~\ref{cor:hexselection} is essential to its conclusion.  Selection
of a preferred lattice concerns optimization across conformal structures;
it is not recovery of an arbitrary unknown lattice from the single scalar
$V_t$.
\end{remark}

\begin{figure}[tbp]
\centering
\includegraphics[width=\textwidth]{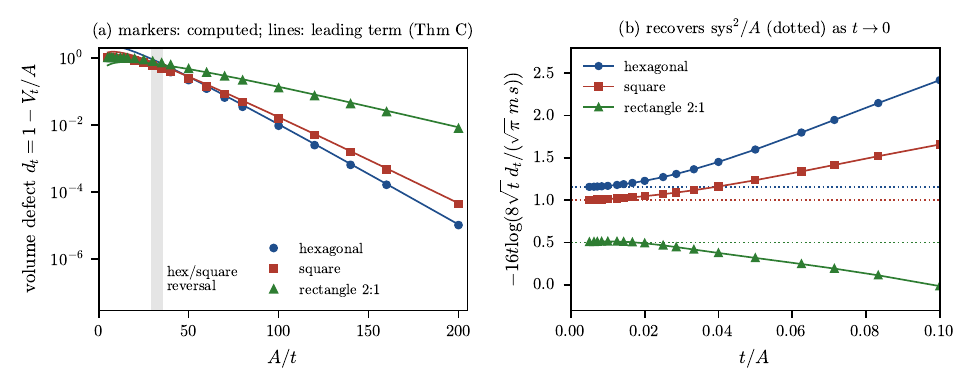}
\caption{Winding loss on three area-one flat tori (hexagonal, square,
and $2{:}1$ rectangle). (a) Volume defect $d_t=1-V_t/A$ against $A/t$:
markers are direct quadrature, solid lines the leading term of
Theorem~\ref{thm:systolicloss}. The lattice with the longest systole
loses least as $t\to0$; the hexagonal and square rankings reverse between
$t/A=.02$ and $.03$ (shaded; cf.\ Remark~\ref{rem:hexnotalltimes}).
(b) Removing the prefactor, $-16t\log\bigl(8\sqrt t\,d_t/(\sqrt\pi\,m\,s)\bigr)$
tends to $\operatorname{sys}^2/A$ (dotted), the exponent in
Theorem~\ref{thm:systolicloss}. Numerical comparisons illustrate, but do
not prove, the asymptotic theorems.}
\label{fig:latticecomparison}
\end{figure}

\begin{remark}[No all-time hexagonal maximum]\label{rem:hexnotalltimes}
For area one, direct Gaussian quadrature gives the following values of
$V_t/A$ (the rectangular lattice has aspect ratio $2$):
\[
\begin{array}{c|ccc}
 t/A&\text{rectangle}&\text{square}&\text{hexagonal}\\\hline
 .005&.9921060721&.9999553590&.9999895519\\
 .010&.8735375480&.9842744586&.9904796886\\
 .020&.5244020171&.7631019134&.7831314880\\
 .030&.2745773086&.4813933474&.4798687430\\
 .050&.0601635672&.1552417827&.1339212031
\end{array}
\]
These are numerical checks, not interval-certified bounds: midpoint grids
of $256^2$ and $512^2$, using image indices $-8\le i,j\le8$, agreed to
$2\cdot10^{-15}$ or better, and the direct Fisher matrix agreed with the
posterior-covariance expression (the first equality in
\eqref{eq:missinginformation}) to the same absolute precision. An independent reciprocal-lattice
Fourier calculation at $t/A=.02,.03,.05$ reproduced the square and
hexagonal values within $3\cdot10^{-16}$.
The reversal is also rigorous at large time.  The reciprocal-lattice
Fourier series gives
\[
 \frac{V_t(\mathcal L_{\rm sq})}A\sim16\pi^2t e^{-8\pi^2t},\qquad
 \frac{V_t(\mathcal L_{\rm hex})}A\sim
 \frac{48\pi^2}{\sqrt3}t e^{-16\pi^2t/\sqrt3}\quad(A=1).
\]
Indeed the leading Fisher matrix is
$\sum_{|\xi|^2=\lambda_1}\xi\otimes\xi\,e^{-2t\lambda_1}$;
the first reciprocal shells are isotropic for these two lattices.
Thus the hexagonal volume is smaller for all sufficiently large $t$.
Montgomery's theta-minimization theorem \cite{Montgomery1988} concerns a
linear Gaussian lattice sum; the nonlinear Fisher quotient is a different
functional, so that theorem alone cannot establish its extrema.
\end{remark}

\section{Curved backgrounds and the remaining resolution problem}
\label{sec:curvedquestions}
The uniform limit and the flat-band theorem answer different questions.
The first controls a family with uniformly bounded spatial derivatives;
the second permits frequencies of order $t^{-1/2}$. Neither alone
transfers the flat multiplier to a curved high-frequency background.

\begin{proposition}[An exact low-frequency obstruction on the sphere]
\label{prop:sphericalgauge}
Let $g_c$ be round, with curvature $K_c>0$, and let $f$ be a first
spherical harmonic, so $-\Delta_{g_c}f=2K_cf$. For every $t>0$,
\[
 D^2E_t(g_c)[f,f]=0
\]
when the derivatives use the area-preserving conformal variable
$g=(1+f)g_c$. Thus a formula
$\tfrac12D^2E_t(g_c)[f,f]/\|\Delta f\|_2^2=q(t\lambda)+O(t)$
cannot hold uniformly on the full band $0<t\lambda\le\Theta$.
\end{proposition}
\begin{proof}
Rotation invariance makes the first conformal variation of $E_t$ at
$g_c$ a constant multiple of $\int f\,dA_{g_c}$; it therefore vanishes
on the fixed-area slice. A conformal automorphism path pulls $g_c$ back
to area-preserving metrics with the same value of $E_t$. Its tangent
conformal factors are exactly the first spherical harmonics. Twice
differentiating along such a path gives the claimed Hessian value,
because its acceleration also has zero integral and the constrained
first variation is zero. Positive-time differentiability follows from
smooth dependence of the heat kernel on the metric on the compact
surface; this argument requires no uniform small-time derivative bound.
For a nonzero first harmonic the proposed ratio is zero, whereas
$q(2K_ct)\to1/4$.
\end{proof}

Removing these gauge modes is necessary, but it does not remove all
curvature dependence at low frequency: the Calabi quadratic form is
$\| (\Delta+2K_c)f\|_2^2/4$, not $\|\Delta f\|_2^2/4$.
For example, a degree-two harmonic on the unit sphere has the former
quantity divided by $\|\Delta f\|_2^2$ equal to $1/9$. The finite-amplitude
statement of Proposition~\ref{prop:curvedfiniteamplitude} rigorously recovers
this number. Passage to the actual small-time Hessian is subject to the
parameter-remainder condition described there.

\paragraph{A correctly localized universality question.}
On the round sphere, rotational symmetry makes each degree-$l$ harmonic
space an eigenspace of the conformal Hessian. For $0<\delta<\Theta<\infty$
and $\lambda_l=l(l+1)K_c$, a meaningful high-frequency target is
\[
 \sup_{\delta\le t\lambda_l\le\Theta}
 \left|\frac{D^2E_t(g_c)[f_l,f_l]}{2\lambda_l^2\|f_l\|_2^2}
          -q(t\lambda_l)\right|\le C_{\delta,\Theta,g_c}t.
\]
This is an open estimate here. A proof must control two metric
variations of the off-diagonal Fisher quotient, retain curvature terms,
and show that the resulting error is uniform for $l\asymp t^{-1/2}$.
The flat quartic anisotropy cancellation cannot be presumed at a curved
background. A numerical critical value near $4.408$ would motivate this
estimate, but would not supply its uniform error bound. No such spherical
critical values are used as established results in this paper.
On a general constant-curvature quotient, a spectral-subspace or
microlocal formulation may be needed, since an eigenspace need not
carry a scalar Hessian multiplier.

\paragraph{Convexity requires a background scale.}
A separate number $\theta_c<\tau_*$ cannot be called a universal
nonlinear convexity boundary without specifying the background radius
and norm. Theorem~\ref{thm:mixedhessian} gives a positive Hessian on a
sufficiently small $H^2$ ball for every fixed $\theta<\tau_*$. An observed
loss of convexity at a chosen scale can therefore concern a
specified nonzero background or a prescribed radius, but not all small
backgrounds. We make no numerical universal-threshold claim of this kind.

\paragraph{What remains global.}
The bounded-geometry hypotheses in the uniformization theorem have not
been derived from the heat-Fisher objective itself. Removing them would
require an equicoercivity result controlling concentration and conformal
degeneration. Likewise, asymptotic selection of hexagonal moduli does
not identify the exact optimizer at each sufficiently small positive
time. These are distinct open problems, beyond the two selection
limits proved above.

\appendix
\section{Coefficients in the proof of Lemma~\ref{lem:closedform}}
\label{app:closedform}
Write the right side of the termwise integration in the proof of
Lemma~\ref{lem:closedform} as $G_1+G_2+G_3$, with
$G_1=\frac18(c_0-c_2+\mathscr L_0-4\mathscr L_1-3\mathscr L_2)$,
$G_2=-\mathscr L_1/(2z)$ and
$G_3=\frac z{16}(-c_0+c_1+c_2-c_3+\mathscr L_0+3\mathscr L_1-3\mathscr L_2-\mathscr L_3)$.
Inserting the formulas for $c_k$ and $\mathscr L_k$, each group is a
combination of $\operatorname{Shi}z$, $\cosh z$ and $\sinh z$; the table
lists $8z^3$ times the coefficients.

\begin{table}[htbp]
\centering
\renewcommand{\arraystretch}{1.25}
\begin{tabular}{c|ccc}
 & $\operatorname{Shi}z$ & $\cosh z$ & $\sinh z$\\\hline
 $8z^3G_1$ & $-2z^2-8z+12$ & $10z$ & $8z-22$\\
 $8z^3G_2$ & $-8$ & $0$ & $8$\\
 $8z^3G_3$ & $-z^3+3z^2+6z-6$ & $z^2-11z$ & $-7z+17$\\\hline
 sum & $-(z^3-z^2+2z+2)$ & $z^2-z$ & $z+3$
\end{tabular}
\caption{Contributions to $8z^3\cdot2e^z\mathscr F(2z)$. The column sums
give the displayed formula in the proof of Lemma~\ref{lem:closedform}.}
\label{tab:closedform}
\end{table}

For example, $G_2=-\mathscr L_1/(2z)=-(\operatorname{Shi}z-\sinh z)/z^3$.
In $G_1$, the $\operatorname{Shi}$ coefficients of $\mathscr L_0,\mathscr L_1,\mathscr L_2$
are $-2/z$, $2/z^2$, $-4/z^3$, so $8z^3G_1$ has
$\operatorname{Shi}$-coefficient $z^3(-2/z-8/z^2+12/z^3)=-2z^2-8z+12$; the other
entries are obtained in the same way.

\section*{Computational checks and preparation}
Numerical quadrature is used only to illustrate the analytic formulas,
compare lattices, and check normalizations. The accompanying scripts
record quadrature orders and convergence comparisons; floating-point
results are not interval certificates. An AI assistant (Claude,
Anthropic) was used for proof review, numerical checks, and editing.
The authors have checked all mathematical content and take full
responsibility for it.

\printbibliography

\end{document}